\documentclass[11pt]{article}
\usepackage[a4paper,margin=26mm]{geometry}
\usepackage[T1]{fontenc}
\usepackage{lmodern}
\usepackage{microtype}
\usepackage{xspace}
\usepackage{amsmath,amssymb,amsthm,mathtools,bm}
\usepackage{mathrsfs}
\usepackage{booktabs,tabularx,array,multirow}
\usepackage{graphicx}
\usepackage{tikz}
\usetikzlibrary{arrows.meta,positioning}
\usepackage{float}
\usepackage{xcolor}
\usepackage{enumitem}
\usepackage{placeins}
\usepackage[numbers,sort&compress]{natbib}
\usepackage[colorlinks=true,linkcolor=blue!45!black,citecolor=blue!45!black,
  urlcolor=blue!45!black,
  pdftitle={MORTIS: Quadrature-compatible cofactor gauges and optimal material margins in finite elasticity},
  pdfauthor={Yanlin Liu, Kaixiang Yao, Chao Huang, Yao Shen},
  pdfsubject={Discrete cofactor-gauge compatibility, optimal material margins and coercive reference forms},
  pdfkeywords={finite elasticity, cofactor gauges, quadrature compatibility, optimal material margins, tensor-product elements, serendipity}]{hyperref}
\usepackage[nameinlink,capitalise]{cleveref}

\newcommand{\R}{\mathbb{R}}
\newcommand{\GLp}{\mathrm{GL}^{+}(3)}
\newcommand{\cof}{\operatorname{cof}}
\newcommand{\tr}{\operatorname{tr}}
\newcommand{\sym}{\operatorname{sym}}
\newcommand{\dev}{\operatorname{dev}}
\newcommand{\rank}{\operatorname{rank}}

\newcommand{\In}{\operatorname{In}}
\newcommand{\norm}[1]{\left\lVert #1\right\rVert}
\newcommand{\Qfour}{Q_4}
\newcommand{\That}{\widehat T}

\newtheorem{theorem}{Theorem}
\newtheorem{proposition}[theorem]{Proposition}
\newtheorem{lemma}[theorem]{Lemma}
\newtheorem{corollary}[theorem]{Corollary}
\theoremstyle{definition}

\theoremstyle{remark}

\AddToHook{env/theorem/begin}{\crefalias{theorem}{theorem}}
\AddToHook{env/proposition/begin}{\crefalias{theorem}{proposition}}
\AddToHook{env/lemma/begin}{\crefalias{theorem}{lemma}}
\AddToHook{env/corollary/begin}{\crefalias{theorem}{corollary}}
\AddToHook{env/definition/begin}{\crefalias{theorem}{definition}}
\AddToHook{env/remark/begin}{\crefalias{theorem}{remark}}

\AddToHook{cmd/appendix/after}{%
  \crefalias{section}{appendix}%
  \crefalias{subsection}{subappendix}%
  \crefalias{subsubsection}{subsubappendix}}

\title{\textbf{MORTIS: Quadrature-compatible cofactor gauges\\
and optimal material margins in finite elasticity}}
\author{Yanlin Liu$^1$ \and Kaixiang Yao$^{3,\dagger}$ \and
Chao Huang$^{2,\dagger}$ \and Yao Shen$^{2,*}$}
\date{\small $^1$The University of Melbourne, Melbourne, Australia\\
$^2$Shanghai Jiao Tong University, Shanghai, China\\
$^3$Southern University of Science and Technology, Shenzhen, China\\
$^\dagger$Kaixiang Yao and Chao Huang share second authorship.\\
$^*$Corresponding author: \texttt{yaoshen@sjtu.edu.cn}}

\begin{document}
\maketitle

\begin{abstract}
We develop a discrete design theory for cofactor reference forms that
preserve the complete finite-element elasticity tangent.
A fixed continuous potential generates the coefficient, while the
variation space and quadrature determine compatibility.
We classify the exact local potential spaces for all-order tensor
elements, anisotropic spaces, cubic serendipity and quadratic tetrahedra.
At a stress-free Mooney--Rivlin state with a common positive
material-coefficient sum, a convex loss expresses the attainable
quadrature-point material margin.
The optimum is attained; under stated space and sampling assumptions,
its ideal value is attained precisely when the physical identity
potential is compatible.
A conforming fixed-candidate example proves a strict margin loss,
linear in curvature amplitude and uniform in cell size, under
displacement enrichment. Explicit positive candidates accompany this
restriction. A Gram identity gives the sharp constant-coefficient
defect at a prescribed margin in the stated norm.
Compatible and reference-defect-corrected forms retain the original
tangent through a deformation-independent external-boundary Hessian;
equal potential traces give equal exact assembled gauges.
Quadratic tensor geometry admits a stress-free physical coercivity
bound uniform in mesh size and displacement order under explicit
material, shape, quadrature and boundary assumptions.
Certified curved tetrahedral shear families with zero recovered cell
pressure establish finite-deformation regimes, while a loaded
pressure-curvature counterexample identifies
their limitation. The theory and observations distinguish local
exactness, material certificates, complete-reference coercivity and
unchanged equations.
\end{abstract}

\noindent\textbf{Keywords:} finite elasticity; cofactor gauges; quadrature
compatibility; material coercivity; tensor-product elements; serendipity

\section{Introduction}
\label{sec:introduction}

A positive reference for a finite-element tangent must respect the
discretization that produced the equation. Cofactor translations offer
a useful freedom: a fixed auxiliary potential generates an added
quadratic form, whose exact compensation leaves the original tangent
unchanged. The potential cannot be chosen by material positivity alone.
Its gradient, the displacement variation space and the quadrature rule
jointly determine whether the added form is integrated exactly.

MORTIS stands for \emph{Margin Optimization and Reference Translation in
Integration-compatible Spaces}. It studies this discrete design problem.
We call the chosen quadratic reference a \emph{gauge}. The analysis
connects three questions: which potentials are compatible, how much
material positivity they can certify, and when the resulting complete
reference is coercive. The construction retains the original material
quadrature and every cell-pressure term.

The first result classifies the exact local potential spaces, testing
against every pair of unrestricted displacement variations.
Let \(\mathbb Q_p\) denote scalar polynomials of degree at most \(p\)
in each coordinate, and let \(n\) be the Gauss count per coordinate.
For vector variations and candidates in \(\mathbb Q_p^3\), the
compatible class is
\[
 \mathbb Q_{\min\{p,\,2(n-p)\}}^3,\qquad p\ge2,\quad n\ge p+1.
\]
The usual \(n=p+1\) rule therefore permits exactly quadratic tensor
potentials at every higher displacement order. The converse proof
resolves cancellation through leading Gauss moments and transverse
boundary moments. Anisotropic spaces have a one-high-axis exception.
Cubic serendipity has three extra mixed-cubic modes; enriching only
the variations to full tensor cubics removes them.

The second result determines attainable material margins at the
stress-free Mooney--Rivlin state. With common material-coefficient sum
\(c>0\), a convex coefficient loss expresses the best lower bound
relative to a fixed positive matrix metric, at the selected quadrature
points and for every matrix increment. The optimum is attained.
Under explicit candidate-space and gradient-sampling assumptions,
its ideal value \(c\) is attainable precisely when the physical identity
potential \(\iota(X)=X\) is compatible. A conforming checkerboard
construction then proves a strict loss, linear in curvature amplitude
and uniform in cell size, when the candidate space and quadrature
are fixed and only the variations are enriched.
A separate Gram identity gives the sharp minimum constant-coefficient
defect at a prescribed margin in a stated basis and norm.

The third part establishes complete reference forms with exact
boundary compensation. For a fixed continuous compatible potential,
the added Hessian is a deformation-independent external-boundary form.
A positive reference gives a positive interior block and bounds the
complete tangent's nonpositive index by the free displacement-trace
dimension. Explicit reference-defect correction extends the
representation beyond cellwise exactness.
Equal full potential traces give equal exactly assembled gauges,
even when their local material certificates differ.
With a full displacement clamp the compatible reference equals the
original tangent. These identities separate a material-certificate
comparison from an ordering of assembled spectra or solver times.

Quadratic tensor geometry supplies a stress-free physical coercivity
bound uniform in mesh size and displacement order, under whole-cell
shape, quadrature, material and Poincar\'e assumptions.
Curved tetrahedral constructions cover aligned, defect-corrected and
variable-compatible finite-shear families, with rational or interval
proofs. Their recovered pressures vanish, so the complete volumetric
term adds a nonnegative outer product at arbitrary nonnegative bulk.
The loaded counterexample retains the signed pressure-geometric term
that those bounds do not control.
Numerical observations test the same choices and preserve failed
references, the distinction between explicit margins and optimum bounds,
and the absence of a consistent partial-boundary iterative advantage.

The underlying variational freedom is classical.
The deformation gradient, cofactor and determinant represent line,
area and volume changes; polyconvex and mixed-minor formulations
exploit this structure
\citep{ball1976convexity,schroder2011mixed,bonet2015polyconvex}.
Potential-generated null Lagrangians allow reference-dependent
coefficients \citep[Eq.~(3.3)]{edelen1986null}
and admit the Piola boundary representation
\citep{olver1988structure,kupferman2019piola,olver2024boundary}.
Translations have established stability and optimization theories,
including attainment
\citep{sivaloganathan1989hamilton,firoozye1991optimal,chernyavsky2023convex}.
The discrete kernel, ideal-margin rigidity and fixed-candidate
enrichment restriction distinguish the present design problem;
semidefinite eigenvalue optimization is a standard tool
\citep{vandenberghe1996semidefinite}.

Discrete geometric conservation also constructs compatible metric
fields \citep{chan2019discretely,bach2025mimetic}.
Here the rule and local variations are fixed, and exactness is required
on every variation pair. The actual tensor or serendipity space
\citep{arnold2011serendipity}, curved geometry and quadrature
\citep{ciarlet1972curved,lenoir1986optimal} therefore enter the
classification explicitly.
Minor calculus and tangent eigensystems supply the material identities
\citep{bonet2016cross,poya2023variational,poya2025generalised};
mean-coercivity and coefficient-jump mechanisms have continuum
precedents \citep{bevan2024hadamard,bevan2025applications}.
The companion RUPA study changes the original nonlinear volume
constraints \citep{liu2026factor}. MORTIS keeps those equations and
studies their reference representations, using the same polarized
determinant foundation.

\Cref{sec:setting} defines the complete construction and gives a
concrete potential choice. \Cref{sec:compatible-kernels,sec:optimal-margin,sec:enrichment-loss}
develop the compatible spaces and their attainable margins.
The constant-coefficient comparison, exact reference identities and
coercivity regimes follow in
\cref{sec:constant-margin-comparison,sec:complete-reference-forms,sec:coercivity-regimes}.
\Cref{sec:evidence} presents the principal observations.
The appendices retain all supporting coordinate, material,
certificate, trace-repair and numerical details.

\section{Discrete geometry, cofactor gauges, and the complete energy}
\label{sec:setting}
\label{sec:discrete-gauge-setting}

\subsection{Matrices, minors, and variations}

Let $\mathbb M=\R^{3\times3}$, with Frobenius product
$H:L=\tr(H^TL)$ and norm $\|H\|_F=(H:H)^{1/2}$.
Vector norms with subscript $2$ are Euclidean; the matrix norm $\|H\|_2$
is its largest singular value. Write $I_d$ for the identity matrix of
dimension $d$ and
$\GLp=\{F\in\mathbb M:\det F>0\}$.
For symmetric forms, $\mathcal B\succeq\mathcal A$ means
$\mathcal B[H,H]\ge\mathcal A[H,H]$ for every increment $H$;
$\mathcal A\succ0$ denotes positive definiteness. Matrices of forms on
$\mathbb M$ use a Frobenius-orthonormal entry basis unless another
basis is stated.

For a deformation gradient $F\in\GLp$, define
\[
J=\det F,\qquad C=\cof F=JF^{-T},\qquad I_1=F:F,\qquad I_2=C:C.
\]
The deformation gradient, cofactor, and determinant map infinitesimal
line, area, and volume elements. Their role in finite elasticity follows
the classical minor and polyconvex formulations
\citep{ball1976convexity,schroder2011mixed,bonet2015polyconvex}.
The cofactor is also defined on singular matrices by its signed
second-order minors.

The symbols $D,D^2$ denote first and second Fr\'echet derivatives.
For matrix increments $H,L\in\mathbb M$, put
$\theta_H=F^{-T}:H$ and define $\theta_L$ similarly.
Differentiation of the determinant gives
\[
DJ(F)[H]=J\theta_H,\qquad
D(F^{-1})[H]=-F^{-1}HF^{-1},
\]
and
\begin{equation}
D^2J(F)[H,L]
=J\{\theta_H\theta_L-\tr(F^{-1}HF^{-1}L)\}.
\label{eq:det-second}
\end{equation}
For two scalar linear functionals $\alpha,\beta$,
$(\alpha\otimes\beta)[H,L]=\alpha[H]\beta[L]$.

We use the tensor cross product in its cofactor-polarized form
\citep{bonet2016cross}:
\begin{equation}
H\times L=\cof(H+L)-\cof H-\cof L,\qquad H\times H=2\cof H.
\label{eq:cross-definition}
\end{equation}
It is symmetric and bilinear. If $h_i,l_i$ are the columns of $H,L$,
its first column is
$h_2\times_{\R^3}l_3+l_2\times_{\R^3}h_3$; the others follow cyclically.
The subscript distinguishes the usual three-dimensional vector cross
product from the matrix operation.

\begin{lemma}[Cofactor derivatives]
\label{lem:cofactor-derivatives}
The linear cofactor derivative at $F\in\GLp$ is the invertible map
\begin{equation}
K_F(H)=D\cof(F)[H]
=J\{\theta_HF^{-T}-F^{-T}H^TF^{-T}\}.
\label{eq:dcof}
\end{equation}
The second derivative is independent of $F$:
\begin{equation}
D^2\cof(F)[H,L]=H\times L.
\label{eq:d2cof}
\end{equation}
Moreover,
\begin{equation}
F:(H\times L)=D^2J(F)[H,L].
\label{eq:det-cross}
\end{equation}
\end{lemma}
\begin{proof}
Differentiate $JF^{-T}$ to obtain \eqref{eq:dcof}. The cofactor is a
homogeneous quadratic map, so its second derivative is its bilinear
polarization. For $C=\cof F$, the identities
$\det C=J^2$ and $\cof C=JF$ give the smooth inverse
$F=\cof C/\sqrt{\det C}$ on $\GLp$. Differentiating the inverse
identity proves invertibility of $K_F$. Finally, the determinant is cubic;
the coefficient of the mixed product $st$ in $\det(F+sH+tL)$ gives
\eqref{eq:det-cross}.
\end{proof}

\subsection{Reference cells and three separate function spaces}

Let $\Omega_h\subset\R^3$ be a connected bounded Lipschitz reference
domain, partitioned into conforming cells $K_e$, $e=1,\ldots,m$.
Each cell is the image of a fixed parent cell $\widehat K_e$ under an
orientation-preserving, one-to-one map $X_e$ with bounded first
derivatives and inverse derivatives. The cell maps agree on the shared
physical faces, with the stated trace parameterizations.
The subscript $h$ identifies the mesh. Quantitative shape assumptions
are supplied in the results that claim uniformity under refinement.

Three spaces have distinct roles. The geometry maps $X_e$ are fixed.
The local variation space $\widehat V_e$ consists of vector fields on
$\widehat K_e$, before any displacement boundary elimination.
The local potential space $\widehat Z_e$ supplies auxiliary vector fields.
A fixed finite-dimensional global candidate space $\mathcal Z_h$ consists
of continuous, piecewise differentiable potentials whose pullbacks belong
to the local potential spaces. Its constants are included unless a
different restriction is explicitly imposed. Geometry, variation, and
potential spaces need not coincide.

Let $y_h:\Omega_h\to\R^3$ be the current deformation and let
$u_h,v_h$ be admissible conforming displacement variations.
Write $y_e=y_h\circ X_e$, $u_e=u_h\circ X_e$, and $v_e=v_h\circ X_e$.
The variations have $u_e,v_e\in\widehat V_e$; the essential boundary
conditions restrict their global assembly.
For a fixed potential $z_h\in\mathcal Z_h$, put $z_e=z_h\circ X_e$.
With $D_\xi$ denoting the parent derivative, define
$J_{X,e}=D_\xi X_e$, $J_{y,e}=D_\xi y_e$, and similarly $J_{u,e}$,
$J_{v,e}$, and $J_{z,e}$. The physical gradients are
\begin{equation}
F=J_yJ_X^{-1},\qquad H=J_uJ_X^{-1},\qquad L=J_vJ_X^{-1},
\qquad R_z=\nabla_Xz_h=J_zJ_X^{-1}.
\label{eq:parent-physical}
\end{equation}
Cell indices may be suppressed in local formulas. Each gradient has
vector components in its rows and differentiation coordinates in its
columns. The potential and $R_z$ remain fixed while the mechanical
deformation is differentiated. The physical identity potential is
$\iota(X)=X$; its parent pullbacks are $X_e$.
The current deformation is assumed to have $\det F>0$ wherever the
energy is evaluated. Global injectivity is verified separately for the
explicit deformation families.

An essential boundary portion $\Gamma_D$ prescribes zero displacement
variations. The remaining free external trace is denoted by $\Gamma$.
A full clamp has $\Gamma_D=\partial\Omega_h$.
The global variation space is a finite-dimensional conforming subspace
of $H^1(\Omega_h;\R^3)$, with the prescribed zero trace and piecewise
bounded gradients. The physical norm is
\[
\|u_h\|_{H^1}^2=\int_{\Omega_h}
 \bigl(|u_h|^2+\|\nabla_Xu_h\|_F^2\bigr)\,dX.
\]
A uniform coercivity statement will explicitly require a fixed
Poincar\'e inequality
$\|u_h\|_{L^2}\le C_P\|\nabla_Xu_h\|_{L^2}$
\citep{brezis2011functional}.
Removing constant nodal fields on one finite mesh does not alone give
a mesh-independent $C_P$.

\subsection{Cubature and the polarized defect}

For a parent-cell integrand $f$, define
\[
\mathcal I_e[f]=\int_{\widehat K_e}f(\xi)\,d\xi,\qquad
\mathcal Q_e[f]=\sum_{q=1}^{n_e}w_{eq}f(\xi_{eq}),\qquad
\mathcal L_e=\mathcal Q_e-\mathcal I_e.
\]
The points and weights are fixed independently of deformation and
potential coefficients. Positive weights are assumed in the positivity
results. Gradient values at a point are evaluated from its indicated cell,
including when two physical quadrature locations coincide. Write
$V_e=\mathcal I_e[\det J_{X,e}]>0$ for exact reference volume and
$V_e^Q=\mathcal Q_e[\det J_{X,e}]>0$ for its quadrature counterpart.
The superscript distinguishes these weights when they differ.

For the tetrahedral specializations, use
\[
\That=\{\xi\in\R^3:\xi_i\ge0,\ \xi_1+\xi_2+\xi_3\le1\},
\qquad |\That|=1/6.
\]
Write $\mathbb P_r$ for scalar polynomials of total degree at most $r$.
The ten-node quadratic tetrahedron, or Tet10, uses $\mathbb P_2$
interpolation at four vertices and six edge midpoints.
Its barycentric coordinates are
$\lambda_0=1-\xi_1-\xi_2-\xi_3$ and $\lambda_i=\xi_i$ for $i=1,2,3$.
The symmetric four-point rule $\Qfour$ has barycentric nodes given by
the four permutations of
\[
\left(\frac{5+3\sqrt5}{20},\frac{5-\sqrt5}{20},
      \frac{5-\sqrt5}{20},\frac{5-\sqrt5}{20}\right)
\]
and weight $1/24$ at each point. It is exact through total degree two
\citep{shunn2012symmetric}. Tensor-product spaces and their Gauss rules
are specified in \cref{sec:compatible-kernels}; $\Qfour$ always denotes
the tetrahedral rule.

Numerical integration and curved isoparametric geometry are classical
sources of variational consistency errors
\citep{ciarlet1972curved,lenoir1986optimal}. Here the error functional is
applied to a specified cofactor form, with the spaces and rules held fixed.
For parent vector fields define
\begin{equation}
\mathcal T_e(z;u,v)
 =\mathcal L_e\bigl[D_\xi z:(D_\xi u\times D_\xi v)\bigr].
\label{eq:common-trilinear-defect}
\end{equation}
It is symmetric and trilinear in the three fields. Determinant
polarization gives
\[
D^2\{\mathcal L_e[\det D_\xi y]\}[u,v]
 =\mathcal T_e(y;u,v).
\]
The companion RUPA study uses the current deformation in this slot to
analyze the original nonlinear volume constraint \citep{liu2026factor}.
For a fixed potential, cofactor quadraticity instead gives
\[
D^2\{\mathcal L_e[D_\xi z:\cof D_\xi y]\}[u,v]
 =\mathcal T_e(z;u,v).
\]
This common algebra has the classical determinant and null-Lagrangian
foundation; the discrete compatibility spaces and their consequences are
the objects studied here.

\begin{theorem}[Determinant and cofactor pullbacks]
\label{thm:dual-pullback}
At a fixed parent point, the physical and parent gradients in
\eqref{eq:parent-physical} satisfy
\begin{align}
\det J_X\,D^2J(F)[H,L]
 &=D^2\det(J_y)[J_u,J_v],\label{eq:det-parent}\\
\det J_X\,(H\times L)
 &=(J_u\times J_v)J_X^T.\label{eq:cof-parent}
\end{align}
Consequently
\[
\det J_X\,R_z:(H\times L)=J_z:(J_u\times J_v).
\]
\end{theorem}
\begin{proof}
Apply determinant multiplicativity to
$J_y+sJ_u+tJ_v=(F+sH+tL)J_X$ and differentiate once in each scalar.
For the cofactor identity, use
$\cof(AB)=(\cof A)(\cof B)$ and take the same mixed derivative.
This gives $J_u\times J_v=(H\times L)\cof J_X$.
Multiplication by $J_X^T$, using
$\cof J_XJ_X^T=(\det J_X)I_3$, proves the second formula.
The final Frobenius product follows from $R_zJ_X=J_z$.
\end{proof}

\subsection{Material energy and the meaning of a gauge}

Define the modified invariants by
\begin{equation}
\bar I_1=J^{-2/3}I_1,\quad \bar I_2=J^{-4/3}I_2.
\label{eq:modified-invariants}
\end{equation}
With cell coefficients $c_{1,e},c_{2,e}\ge0$ and
$c_{1,e}+c_{2,e}>0$, use the isochoric Mooney--Rivlin density
\begin{equation}
W_{{\rm iso},e}(F)
 =c_{1,e}(\bar I_1-3)+c_{2,e}(\bar I_2-3).
\label{eq:mooney-energy}
\end{equation}
The two invariants remove uniform dilation from the line and area measures
\citep{mooney1940theory,rivlin1948large}. The reference shear modulus is
$\mu_e=2(c_{1,e}+c_{2,e})$. The optimal-margin theorem will impose a
common positive sum across cells; the later reference-coercivity result
permits heterogeneous sums with a common positive lower bound.

For a coefficient $R\in\mathbb M$ held fixed in a variation, the translated
material form is
\[
\mathcal G_R(F)[H,L]=D^2W_{\rm iso}(F)[H,L]+R:(H\times L).
\]
We call a reference form of this type a cofactor gauge. The same cofactor
term is subtracted when reconstructing the complete Hessian.
A gauge's positive definiteness is therefore a property of the chosen
reference form, with its compensation still part of the equation.
A rank-one increment has the form $H=a\otimes n$, where
$(a\otimes n)_{ij}=a_in_j$. It has zero cofactor, so the added quadratic
value vanishes. Strict rank-one ellipticity means
$D^2W_{\rm iso}(F)[H,H]>0$ for every nonzero rank-one $H$.
Positive gauges provide sufficient rank-one ellipticity
certificates through the classical minor-translation mechanism
\citep{ball1976convexity,sivaloganathan1989hamilton}.
No converse for arbitrary constitutive laws is assumed.

The first-minor reference used in the finite-strain certificates is
\begin{align}
Q_F(H,L)
 &=
 \left(H-\frac23\theta_HF\right):
 \left(L-\frac23\theta_LF\right)
 +\frac{I_1}{9}\theta_H\theta_L,\label{eq:qf}\\
A_1(F)[H,L]&=2J^{-2/3}Q_F(H,L),\label{eq:a1}\\
\pi_1(F)&=\frac{2I_1}{3J^{5/3}}.\label{eq:pi1}
\end{align}
The sum of squares makes $Q_F$ positive definite.
The Flory first-minor identity, also used by RUPA \citep{liu2026factor}, is
\begin{equation}
D^2\bar I_1(F)=A_1(F)-\pi_1(F)D^2J(F).
\label{eq:first-invariant-import}
\end{equation}
Its short derivation and the full second-minor chain rule are given
in \cref{app:material-calculus}, so the present paper is self-contained
in these constitutive calculations. At rest define the fixed metric
\begin{equation}
\mathcal A[H,L]=A_1(I_3)[H,L]
 =2H:L+\frac23(\tr H)(\tr L).
\label{eq:rest-material-metric}
\end{equation}
It has eigenvalue two on trace-free matrices and four on the identity
direction in the Frobenius metric.

\subsection{Complete cell-pressure terms}

Use one constant pressure per cell, denoted by $P_0$, and define the
quadrature volume change by
$c_e(y_h)=\mathcal Q_e[\det J_{y,e}-\det J_{X,e}]$.
For a bulk modulus $\kappa_e\ge0$, the complete stored energy is
\begin{equation}
E_h(y_h)
 =\sum_e\mathcal Q_e[\det J_{X,e}\,W_{{\rm iso},e}(F)]
  +\sum_e\frac{\kappa_e}{2V_e^Q}c_e(y_h)^2,\qquad
K_h(y_h)=D^2E_h(y_h).
\label{eq:complete-pressure-energy}
\end{equation}
For positive bulk modulus, the penalty is stationary elimination of a
cell pressure from
$p_ec_e-V_e^Qp_e^2/(2\kappa_e)$. It gives
$p_e=\kappa_ec_e/V_e^Q$ and, for nodal directions $u,v$,
\begin{align}
Q_{P0,e}[u,v]&=\frac{\kappa_e}{V_e^Q}Dc_e[u]Dc_e[v],
 \label{eq:qp0-import}\\
P_{P0,e}[u,v]&=p_eD^2c_e[u,v].
 \label{eq:pp0-import}
\end{align}
Both terms follow by differentiating $c_e^2/2$.
The zero-bulk case contributes neither term. At nonzero pressure the
second term is signed; it remains part of every complete operator below.
The same energy-level distinction is developed in RUPA
\citep{liu2026factor}. Contact, inertia, and external-load tangents would
be additional specified terms and are not silently absorbed into a
stored-energy positivity claim.

For a symmetric finite matrix $A$, write
$\operatorname{In}(A)=(n_+(A),n_-(A),n_0(A))$ for its positive, negative,
and zero eigenvalue counts, with multiplicity. Its nonpositive index is
$n_-(A)+n_0(A)$. Computed sparse inertia is a numerical estimate, whose interpretation
is stated with each comparison. A later boundary representation controls these counts only
when its stated reference form is positive.

\subsection{The complete construction}
\label{sec:gauge-design-levels}

A compatible potential makes the added cofactor Hessian exact on every
local variation pair. With the spaces and defect defined above, set
\begin{equation}
\widehat Z_e^{\,0}
=\left\{z\in\widehat Z_e:
\mathcal T_e(z;u,v)=0
\quad\hbox{for every }u,v\in\widehat V_e\right\},
\label{eq:local-compatible-kernel}
\end{equation}
where \(\mathcal T_e\) is the fixed cubature defect
\eqref{eq:common-trilinear-defect}.
These tests use the unrestricted local variation space, before displacement
boundary conditions are imposed. The continuous global exact class is
\begin{equation}
\mathcal Z_h^0
=\{z_h\in\mathcal Z_h:z_h\circ X_e\in\widehat Z_e^{\,0}
                         \text{ for every }e\}.
\label{eq:global-compatible-kernel}
\end{equation}
Assembly may have additional cancellations. They do not enlarge the
cellwise class in \eqref{eq:local-compatible-kernel}.

For a fixed continuous \(z_h\), define the exact and numerical
cofactor functionals and their global Hessian matrices by
\[
\begin{aligned}
 \mathscr J_z^I(y_h)&=\sum_e\mathcal I_e[J_{z,e}:\cof J_{y,e}],
 &\mathscr J_z^Q(y_h)&=\sum_e\mathcal Q_e[J_{z,e}:\cof J_{y,e}],\\
 \mathcal C_h^r[z_h]&=D^2\mathscr J_z^r,\qquad r\in\{I,Q\}.
\end{aligned}
\]
The matrices act on the actual global displacement variations.
They are independent of \(y_h\), since the cofactor is quadratic and
the potential is fixed. Their difference is assembled from
\(\mathcal T_e(z_e;u_e,v_e)\).
For \(z_h\in\mathcal Z_h^0\), they coincide, and the reference is
\[
 G_{z,h}(y_h)=K_h(y_h)+\mathcal C_h^Q[z_h].
\]
Let \(R_\Gamma\) extract the free external displacement trace.
The boundary theorem below identifies a fixed symmetric trace matrix
\(B_\Gamma[z_h]\) and proves the complete representation
\[
 \mathcal C_h^Q[z_h]=R_\Gamma^TB_\Gamma[z_h]R_\Gamma,
 \qquad
 K_h=G_{z,h}-R_\Gamma^TB_\Gamma[z_h]R_\Gamma
 \quad(z_h\in\mathcal Z_h^0).
\]
Thus the same term is added to form the reference and subtracted to
recover the original tangent. The potential's full boundary values
determine the exact added operator; the displacement's free trace
determines where the compensation acts. Equal potential traces give
equal exact assembled gauges. With a full displacement clamp,
\(R_\Gamma=0\) and \(G_{z,h}=K_h\).
The complete proof and the extension to nonzero quadrature defects
are in \cref{sec:complete-reference-forms}.

\paragraph{A concrete compatible choice.}
Suppose the reference maps are quadratic tensors, the full local
variations have coordinate degree \(p_e\ge2\), and each coordinate
has at least \(p_e+1\) Gauss points. Let the candidate space contain
\(\iota(X)=X\), and assume the common material sum is \(c>0\).
At the stress-free state, choose \(z_h=2c\iota\).
Its physical coefficient is \(R_z=2cI_3\), and
\[
 D^2W_{{\rm iso},e}(I_3)[H,H]+R_z:(H\times H)
 =c\mathcal A[H,H].
\]
This is \eqref{eq:stress-free-translation-identity}, derived below.
The pullback \(2cX_e\) is compatible by the tensor kernel.
Every recovered pressure is zero, so the complete reference adds the
nonnegative cell-pressure outer term to the assembled material metric.
\Cref{thm:q2-uniform-coercivity} supplies the physical \(H^1\) estimate
under its whole-cell geometry and Poincar\'e assumptions.

\begin{figure}[H]
\centering
\begin{tikzpicture}[
 box/.style={draw=blue!45!black,rounded corners=2pt,
 fill=blue!3,align=center,text width=4.05cm,minimum height=1.55cm,
 inner sep=5pt,font=\small},
 >=Stealth,node distance=.55cm]
 \node[box] (space) {Compatible potential\\
 \(z_h\in\mathcal Z_h^0\)\\
 all local variation pairs};
 \node[box,right=of space] (margin) {Material certificate\\
 \(\gamma_h(z)\) at rest\\
 all matrix increments};
 \node[box,right=of margin] (form) {Complete reference\\
 \(G_{z,h}=K_h+\mathcal C_h^Q[z_h]\)\\
 actual global variations};
 \draw[->,thick] (space)--(margin);
 \draw[->,thick] (margin)--(form);
 \node[below=5pt of margin,align=center,font=\footnotesize,text width=14cm]
 {Physical coercivity also requires geometry, boundary and pressure control.\\
 Exact compensation retains \(K_h=G_{z,h}-R_\Gamma^TB_\Gamma[z_h]R_\Gamma\).};
\end{tikzpicture}
\caption{The three levels of the design problem for a chosen potential.
Local exactness restricts the available fields; the material margin
measures a finite-point certificate; the complete reference includes
all pressure terms and its exact boundary compensation.
The diagram summarizes mathematical dependencies and specifies no
numerical optimizer.}
\label{fig:reference-design}
\end{figure}

The material margin of a chosen potential, denoted by \(\gamma_h(z)\)
at rest, tests all matrix increments at the selected quadrature points relative
to \(\mathcal A\). Positive weights transfer a positive margin to the
assembled material-plus-gauge form. Physical coercivity additionally
needs geometric norm control, suitable boundary conditions and the
complete pressure contributions.
A material-margin comparison therefore does not by itself order
assembled eigenvalues, condition numbers or solve times.

\section{The compatible potential spaces}
\label{sec:compatible-kernels}

The local and global exact classes are defined by
\eqref{eq:local-compatible-kernel}--\eqref{eq:global-compatible-kernel}.
We now identify them for the actual polynomial spaces and Gauss rules.

This class has a finite linear description. If \(\psi_a\) is a basis of
\(\widehat Z_e\) and \(\varphi_i\) is a basis of \(\widehat V_e\), the
matrix with entries
\[
(\mathcal D_e)_{(i,j),a}
=\mathcal T_e(\psi_a;\varphi_i,\varphi_j),\qquad i\le j,
\]
maps potential coefficients to their symmetric Hessian-defect coefficients.
Its kernel represents \(\widehat Z_e^{\,0}\).
The following results identify that kernel analytically for specified
polynomial spaces and Gauss rules.

The meaning here is Hessian exactness. For a fixed potential \(z\), let
\(\Lambda_z(y)=\mathcal L_e[J_z:\cof J_y]\), where
\(J_z=D_\xi z\) and \(J_y=D_\xi y\).
Its Hessian on \(\widehat V_e\) is \(\mathcal T_e(z;\cdot,\cdot)\).
If every current map \(y\) belongs to the linear space \(\widehat V_e\),
cofactor homogeneity gives
\[
\Lambda_z(y)=\tfrac12\mathcal T_e(z;y,y)=0
\qquad(z\in\widehat Z_e^{\,0}).
\]
On an affine current-map family \(y_0+\widehat V_e\) with
\(y_0\notin\widehat V_e\), vanishing Hessian instead leaves possible
constant and linear terms. Equality of the full functionals then needs
those terms to be checked separately.

\subsection{Tensor Gauss rules at every polynomial order}

Use the parent cube \(\widehat K=[-1,1]^3\).
The scalar space \(\mathbb Q_p\) consists of polynomials of degree at most
\(p\) in each coordinate. The tensor rule \(\mathcal G_n\) has \(n\)
Gauss--Legendre points per coordinate and integrates coordinate degree
at most \(2n-1\) exactly. Cell indices will be suppressed in the local
pairing. The kernel concerns candidates in the stated finite-dimensional
space, rather than potentials of arbitrary degree.

\begin{theorem}[Tensor compatible-potential kernel]
\label{thm:tensor-compatible-kernel}
Let \(p\ge2\), \(n\ge p+1\), and take
\(\widehat V=\widehat Z=\mathbb Q_p^3\), with
\(\mathcal Q=\mathcal G_n\). Put \(r=\min\{p,2(n-p)\}\).
Then
\begin{equation}
\widehat Z^{\,0}
=\{z\in\mathbb Q_p^3:
       \mathcal T(z;u,v)=0\ \hbox{for all }u,v\in\mathbb Q_p^3\}
=\mathbb Q_r^3.
\label{eq:tensor-compatible-kernel}
\end{equation}
For \(p=1\) and \(n\ge2\), the entire space \(\mathbb Q_1^3\) is compatible.
\end{theorem}

\begin{proof}
If \(z\in\mathbb Q_s^3\) and \(u,v\in\mathbb Q_p^3\), each term of
\(D_\xi z:(D_\xi u\times D_\xi v)\) has coordinate degree at most
\(s+2p-1\). Across the three gradient factors, a determinant term
differentiates once in each coordinate. Thus the rule is exact for
\(s=r\), proving \(\mathbb Q_r^3\subset\widehat Z^{\,0}\).
This also proves the \(p=1\) assertion. If \(r=p\), the proof is complete.

Supporting \(u\) and \(v\) in two distinct physical components isolates
one scalar component \(f\) of the potential. Its pairing is
\[
\mathcal L\left[
\det\begin{pmatrix}
f_x&f_y&f_z\\ U_x&U_y&U_z\\ V_x&V_y&V_z
\end{pmatrix}\right],
\]
with scalar variations \(U,V\). Cyclic choices of the two components
isolate every component of \(z\).
It suffices to exclude a scalar leading degree above the proposed bound.

We give the coordinate argument with separate bounds
\((p_x,p_y,p_z)\) and rules \((n_x,n_y,n_z)\), all satisfying
\(n_j\ge p_j+1\). Assume \(p_y\ge2\) and \(p_z\ge1\).
For the present isotropic theorem these are the equal orders.
Suppose \(f\) has leading \(x\)-degree
\(a>r_x=2(n_x-p_x)\), and write
\[
f(x,y,z)=x^aH(y,z)+\text{terms of lower \(x\)-degree},
\]
where \(H\ne0\) has coordinate degrees at most \(p_y,p_z\).
Set
\[
b=p_x,\qquad c=2n_x+1-p_x-a,\qquad
\tau=\frac{a}{b+c}>0.
\]
Because \(r_x<a\le p_x\), one has \(3\le r_x+1\le c\le b\) and
\(a+b+c=2n_x+1\).

Choose integer totals \(1\le M\le p_y+1\) and \(1\le N\le p_z+1\), and
use
\[
U=x^by^jz^k,\qquad V=x^cy^{M-j}z^{N-k},
\]
with all exponents of each variation within their coordinate bounds.
There are two consecutive admissible choices for \(j\), because
\(p_y\ge2\); at least one admissible split \(k\) exists.
All transverse integrands have degrees at most \(2p_y,2p_z\) and are
integrated exactly. Lower \(x\)-terms of \(f\) give degree at most
\(2n_x-1\). The leading term alone reaches \(x^{2n_x}\).
If \(\pi_{n_x}\) is the monic Legendre polynomial, exactness below this
degree and its nodal zeros give
\[
\left(\mathcal G_{n_x}^{(1)}-\int_{-1}^1\right)x^{2n_x}
=-\int_{-1}^1\pi_{n_x}^2<0,
\]
where \(\mathcal G_{n_x}^{(1)}\) denotes the one-dimensional rule.
Hence the transverse coefficient of the leading term must integrate to
zero for every selected test.

For \(F=y^jz^k\) and \(G=y^{M-j}z^{N-k}\), that coefficient is
\[
\det\begin{pmatrix}
aH&H_y&H_z\\
bF&F_y&F_z\\
cG&G_y&G_z
\end{pmatrix}.
\]
Subtracting its expressions for two consecutive values of \(j\), while
keeping \(M,N,k\) fixed, gives the necessary moment identity
\[
\int_{[-1,1]^2} y^{M-1}z^{N-1}
       \{aNH-(b+c)zH_z\}\,dy\,dz=0.
\]
For fixed \(N\), the expression left after \(z\)-integration is a
polynomial in \(y\) of degree at most \(p_y\).
Its moments against \(1,y,\ldots,y^{p_y}\) vanish, so it is identically
zero: testing against the polynomial itself proves this moment
uniqueness. Consequently, for every fixed \(y\),
\begin{equation}
\int_{-1}^1z^NH_z(y,z)\,dz
=\tau N\int_{-1}^1z^{N-1}H(y,z)\,dz,
\qquad 1\le N\le p_z+1.
\label{eq:kernel-transverse-moments}
\end{equation}

A second set of tests determines the endpoint values of \(H\).
Take \(U=x^bF(y)\) and \(V=x^cG(y)\), independent of \(z\).
Their leading transverse coefficient is
\[
H_z(y,z)\{bF(y)G'(y)-cG(y)F'(y)\}.
\]
The choices \(F=1,\ G=y^{j+1}\), \(0\le j<p_y\), produce nonzero
multiples of \(1,y,\ldots,y^{p_y-1}\).
The choice \(F=y,\ G=y^{p_y}\) produces
\((bp_y-c)y^{p_y}\), with \(bp_y-c>0\).
These particular tests have transverse degree at most \(2p_y\), so
their transverse integration is exact as well.
Their vanishing moments imply
\[
H(y,1)-H(y,-1)=0
\]
as a polynomial in \(y\).
Denote the common endpoint value by \(A(y)\).
Integrating \eqref{eq:kernel-transverse-moments} by parts gives
\[
\int_{-1}^1 z^jH(y,z)\,dz
=\frac{A(y)}{1+\tau}\int_{-1}^1z^j\,dz,
\qquad 0\le j\le p_z.
\]
Moment uniqueness in \(z\) now yields
\(H(y,z)=A(y)/(1+\tau)\).
Evaluation at \(z=1\) forces \(A(y)=0\), because \(\tau>0\).
This contradicts \(H\ne0\).

The coordinate argument excludes every degree above \(r\) in each
coordinate of the isotropic candidate. Applying it to all scalar
components proves the reverse inclusion in
\eqref{eq:tensor-compatible-kernel}.
\end{proof}

For the usual \(n=p+1\) rule, the exact local potentials are precisely
\(\mathbb Q_2^3\) for every \(p\ge2\).
Additional quadrature points can enlarge this class: \(p=5,n=7\) gives
\(\mathbb Q_4^3\).
This order dependence also applies to constant physical coefficients.
If \(X_e\in\mathbb Q_p^3\) and \(R\) is a constant physical matrix,
the potential with pullback \(RX_e\) has \(R_z=R\) by
\eqref{eq:parent-physical}. Its Hessian is compatible exactly when
\[
RX_e\in\mathbb Q_r^3.
\]
Equivalently, \(R\) annihilates every coefficient of \(X_e\) whose
multi-index has a coordinate greater than \(r\).
A nonsingular \(R\) can satisfy this condition only if
\(X_e\in\mathbb Q_r^3\).

\subsection{Anisotropic orders and the one-axis exception}

Let \(\mathbb Q_{\boldsymbol p}\) denote coordinate degrees at most
\(\boldsymbol p=(p_1,p_2,p_3)\), and let
\(\mathcal G_{\boldsymbol n}\) be the corresponding tensor Gauss rule.
The candidate and variation spaces in the next theorem coincide with
\(\mathbb Q_{\boldsymbol p}^3\). For a univariate variable \(x\), write
\(\mathbb P_j(x)\) for polynomials of degree at most \(j\).

\begin{theorem}[Anisotropic compatible-potential kernel]
\label{thm:anisotropic-compatible-kernel}
Let \(p_i\ge1\), \(n_i\ge p_i+1\), and
\(r_i=\min\{p_i,2(n_i-p_i)\}\).
If at least two coordinate orders satisfy \(p_i\ge2\), then
\begin{equation}
\widehat Z^{\,0}=\mathbb Q_{(r_1,r_2,r_3)}^3.
\label{eq:anisotropic-compatible-kernel}
\end{equation}
If exactly one coordinate has order greater than one, permute coordinates
so that \(\boldsymbol p=(p,1,1)\), with \(p\ge2\), and put
\(r=\min\{p,2(n_1-p)\}\). The scalar kernel is
\begin{equation}
\mathcal K_{p,n_1}
=\mathbb P_r(x)\operatorname{span}\{1,y,z\}
 \oplus \mathbb P_{\min(p,r+1)}(x)\,yz,
\qquad
\widehat Z^{\,0}=\mathcal K_{p,n_1}^{\,3}.
\label{eq:one-high-axis-kernel}
\end{equation}
If all three \(p_i=1\), the entire vector candidate space is compatible.
\end{theorem}

\begin{proof}
For potentials of coordinate degrees \(r_i\), the mixed determinant
integrand has coordinate degrees at most \(r_i+2p_i-1\le2n_i-1\);
this proves sufficiency in \eqref{eq:anisotropic-compatible-kernel}.
To exclude a higher potential degree in one coordinate, use the coordinate
argument in \Cref{thm:tensor-compatible-kernel}. Any nontrivial excluded
degree lies in a coordinate of order at least three.
Another coordinate has order at least two and can serve as \(y\) in that
argument; the remaining coordinate has order at least one.
The permitted moment totals are exactly \(1,\ldots,p_y+1\) and
\(1,\ldots,p_z+1\), and the endpoint tests have degree at most \(2p_y\).
Thus every step of the leading-term and boundary-moment proof applies,
giving necessity in all coordinates.
The case with all orders one follows directly from degree exactness.

It remains to determine the one-high-axis case.
Write a scalar candidate uniquely as
\[
f(x,y,z)=A(x)+B(x)y+C(x)z+D(x)yz,
\qquad A,B,C,D\in\mathbb P_p(x).
\]
The transverse rules integrate all relevant \(y,z\) polynomials exactly.
For arbitrary \(F,G\in\mathbb P_p(x)\), the following scalar variation
pairs isolate the indicated coefficient through the transverse integral
of \(\det(\nabla f,\nabla U,\nabla V)\):
\[
\begin{array}{c|c|c}
\text{coefficient}&(U,V)&\displaystyle\int_{[-1,1]^2}
                         \det(\nabla f,\nabla U,\nabla V)\,dy\,dz\\ \hline
A&(Fy,Gz)&4A'FG\\
B&(F,Gz)&-4BF'G\\
C&(F,Gy)&4CF'G\\
D&(Fy,Gy)&\tfrac43D(F'G-FG').
\end{array}
\]
The other coefficient sectors vanish in each row by transverse parity.
The products \(FG\) span \(\mathbb P_{2p}\), and \(F'G\) span
\(\mathbb P_{2p-1}\).
The combinations \(F'G-FG'\) span \(\mathbb P_{2p-2}\): for monomials
\(F=x^b,G=x^c\), they equal
\((b-c)x^{b+c-1}\), and every degree from zero through \(2p-2\) admits
\(0\le b,c\le p\) with \(b\ne c\).

Suppose \(A\) has leading degree \(a>r\).
Choose \(FG\) of degree \(2n_1-a+1\); the bounds on \(a\) place that
degree in \(0,\ldots,2p\).
Its product with \(A'\) reaches the first missed moment \(x^{2n_1}\),
whereas all lower coefficients are integrated exactly.
The nonzero first-missed Gauss moment forces the leading coefficient to
vanish, a contradiction. The \(B,C\) rows use a product \(F'G\) of
degree \(2n_1-a\) to give the same contradiction for \(a>r\).
The \(D\) row uses \(F'G-FG'\) of degree \(2n_1-a\), available whenever
\(a>r+1\). Thus the four coefficient degrees are bounded by
\(r,r,r,\min(p,r+1)\).
If \(r=p\), there are no higher candidate degrees to exclude.

For sufficiency, the \(A,B,C\) sectors have maximal \(x\)-degree at most
\(r+2p-1\le2n_1-1\) in every mixed determinant term.
For the \(D(x)yz\) sector, take transverse monomials
\(U=F(x)y^jz^k\), \(V=G(x)y^\ell z^m\), with exponents zero or one.
A nonzero transverse moment requires \(j=\ell\) and \(k=m\) by parity.
In that case the term containing \(D'\) cancels, and the remaining
coefficient is proportional to \(D(F'G-FG')\).
Its degree is at most \(r+1+2p-2\le2n_1-1\).
All other transverse terms are odd and have zero exact and Gauss
integrals. This proves sufficiency and the full scalar kernel.
The complementary-component reduction gives its vector counterpart.
\end{proof}

When \(r<p\), \eqref{eq:one-high-axis-kernel} consists of
\(\mathbb Q_{(r,1,1)}\) and one additional scalar mode \(x^{r+1}yz\).
For example, \((p_1,p_2,p_3)=(3,1,1)\) with Gauss orders \((4,2,2)\)
admits \(x^3yz\). The transverse variation space controls this exception;
the isotropic kernel formula alone would omit it.

\subsection{Serendipity spaces and fixed-candidate enrichment}

We use the serendipity polynomial spaces defined by superlinear degree
\citep{arnold2011serendipity}. For a monomial
\(x^{\alpha_1}y^{\alpha_2}z^{\alpha_3}\), its superlinear degree is
\(\sum_{\alpha_i\ge2}\alpha_i\): exponents zero and one do not contribute.
The scalar space \(\mathcal S_p\) is the span of monomials with
superlinear degree at most \(p\).
Thus \(\mathcal S_2\) contains the eight multilinear monomials and twelve
monomials with exactly one exponent two; its dimension is 20.
The space \(\mathcal S_3\) adds twelve monomials with exactly one exponent
three and has dimension 32.
Both are subspaces of their corresponding tensor spaces.

\begin{theorem}[Cubic-serendipity compatible kernel]
\label{thm:serendipity-compatible-kernel}
For \(\widehat V=\widehat Z=\mathcal S_3^3\) and tensor Gauss4,
\begin{equation}
\widehat Z^{\,0}
=\left(\mathcal S_2\oplus
\operatorname{span}\{x^3yz,xy^3z,xyz^3\}\right)^3.
\label{eq:serendipity-compatible-kernel}
\end{equation}
Keep the candidate space fixed to \(\widehat Z=\mathcal S_3^3\), and
enlarge only the local variations to \(\widehat V=\mathbb Q_3^3\).
Then
\begin{equation}
\widehat Z^{\,0}=\mathcal S_2^3.
\label{eq:fixed-candidate-enriched-kernel}
\end{equation}
\end{theorem}

\begin{proof}
The candidates in \(\mathcal S_2\subset\mathbb Q_2\) are compatible
against the larger \(\mathbb Q_3\) variation space by
\Cref{thm:tensor-compatible-kernel}.
Consider next the scalar candidate \(f=x^3yz\).
For monomial variations with exponent triples \(\beta,\gamma\), the
mixed determinant has coefficient
\(\det[(3,1,1);\beta;\gamma]\) and powers
\((3,1,1)+\beta+\gamma-(1,1,1)\).
Only the \(x\)-degree can exceed Gauss4's exactness degree seven.
That requires \(\beta_x=\gamma_x=3\).
Serendipity then forces their other exponents to be zero or one.
Nonzero transverse moments require
\(\beta_y=\gamma_y\) and \(\beta_z=\gamma_z\), by parity.
Their full exponent triples coincide, making the determinant coefficient
zero. Every remaining term is integrated exactly or has an odd
transverse power. This proves compatibility of \(x^3yz\);
coordinate permutation supplies the other two mixed modes.

For necessity, write a candidate as its \(\mathcal S_2\) part plus three
coordinate groups of the form
\[
x^3(A+By+Cz+Dyz)
\]
and their coordinate permutations, with scalar coefficients \(A,B,C,D\).
For this \(x\)-group the variation pairs
\[
(x^3,x^3z),\qquad (x^3,x^3y),\qquad (x^3y,x^3z)
\]
give mixed determinants, respectively,
\[
-3x^8(B+Dz),\qquad
 3x^8(C+Dy),\qquad
 3x^8(A-Dyz).
\]
Transverse integration isolates \(B,C,A\) in that order.
The \(x^8\) Gauss defect is nonzero.
Terms in \(\mathcal S_2\) and in the other two cubic coordinate groups
remain within the rule's exactness in these tests, so they cannot cancel
the displayed defects.
Thus \(A=B=C=0\). The same argument in each coordinate leaves precisely
the three mixed modes in \eqref{eq:serendipity-compatible-kernel}.
The complementary-component reduction again proves the vector statement.

With \(\mathbb Q_3^3\) variations, the tensor theorem gives kernel
\(\mathbb Q_2^3\) among \(\mathbb Q_3^3\) candidates.
Restricting the candidates to the unchanged \(\mathcal S_3^3\) gives
\(\mathcal S_3^3\cap\mathbb Q_2^3=\mathcal S_2^3\).
Indeed, an \(\mathcal S_3\) monomial with all exponents at most two
has at most one exponent two, exactly the \(\mathcal S_2\) condition.
\end{proof}

The first kernel has 23 scalar dimensions; the fixed-candidate enriched
kernel has 20. The 27-dimensional scalar \(\mathbb Q_2\) kernel pertains
to the different, larger \(\mathbb Q_3\) candidate space.
These dimensions therefore describe three precisely specified local
space pairs.

One direct pairing exhibits the lost mode. Take
\[
z_*=(x^3yz,0,0),\qquad u_*=(0,x^3y^2,0),\qquad v_*=(0,0,x^3).
\]
The potential belongs to the fixed \(\mathcal S_3^3\) candidate space.
The variation \(u_*\) belongs to \(\mathbb Q_3^3\) and is excluded from
\(\mathcal S_3^3\).
Its scalar mixed determinant is \(-6x^8y^2\), giving
\begin{equation}
\mathcal T(z_*;u_*,v_*)
=-6\left(-\frac{128}{11025}\right)
       \left(\frac23\right)2
=\frac{1024}{11025}.
\label{eq:serendipity-enrichment-moment}
\end{equation}
Here \(128/11025\) is the squared integral norm of the monic fourth
Legendre polynomial. Thus enrichment exposes a nonzero Hessian defect
at the same candidate potential and quadrature rule.
The effect of this restricted compatible class on attainable material
certificates is developed below.

\subsection{The quadratic-tetrahedron specialization}

On a tetrahedral parent, write \(\mathbb P_j\) for total polynomial degree
at most \(j\). For
\(\widehat V=\widehat Z=\mathbb P_2^3\) and the symmetric four-point
degree-two rule, the local compatible space is
\[
\widehat Z^{\,0}=\mathbb P_1^3.
\]
Affine potentials are sufficient by degree exactness.
For necessity, the opposite-edge representation in
\Cref{cor:faithful-jet} makes the map from quadratic-modulo-affine
potential coordinates to its defect Hessian injective.
Its Hessian therefore vanishes precisely when all quadratic midside
deviations vanish, which is the parent-affine condition.
For a constant physical matrix \(R\), the potential pullback is \(RX_e\);
the same conclusion is the alignment criterion of
\Cref{prop:exact-gauge-alignment}.
The complete opposite-edge derivation is retained in that appendix.

All classifications in this section concern the local potential slot of
the fixed pairing, tested against every pair of unrestricted local
variations. Global continuity selects the corresponding
\(\mathcal Z_h^0\), and boundary compensation concerns its assembled
operator. Positivity of the resulting reference form is a further
material and geometric question.

\section{Optimal compatible material margins}
\label{sec:optimal-margin}

The exact-potential kernel specifies the admissible choices. We now
measure their positivity through the material certificate defined in
\cref{sec:gauge-design-levels}. The principal optimization ranges over
continuous compatible potentials and tests their translated material
forms at the prescribed quadrature points.
Attainment and ideal-margin rigidity lead to the strict enrichment
restriction in \cref{sec:enrichment-loss}. The complementary
constant-coefficient problem is treated in \cref{sec:constant-margin-comparison}.
The translation mechanism belongs to classical null-Lagrangian stability theory
\citep{edelen1986null,sivaloganathan1989hamilton}; the fixed cubature and
potential space determine the restrictions considered here.

\subsection{The stress-free translated form}

Assume throughout this section that the material coefficients in
\eqref{eq:mooney-energy} have a common positive sum
$c=c_{1,e}+c_{2,e}>0$ across the cells. Their individual values may differ.
At the identity, their Hessians therefore coincide; denote this common
form by $\mathcal H_c=D^2W_{{\rm iso},e}(I_3)$.
For a matrix increment $H$, define its symmetric trace-free part by
\[
 H_{\rm d}=\frac{H+H^T}{2}-\frac{\tr H}{3}I_3.
\]
Differentiation at rest gives
\[
 \mathcal H_c[H,H]=4c\|H_{\rm d}\|_F^2.
\]
To verify this directly, put $s=\tr H$, $r=\tr(H^2)$, and
$n_H=\|H\|_F^2$. Along $F(t)=I_3+tH$,
\[
 \det F(t)=1+st+\tfrac12(s^2-r)t^2+O(t^3),\qquad
 \cof F(t)=I_3+t(sI_3-H^T)+t^2\cof H.
\]
Substitution into each modified invariant shows that its coefficient of
$t^2$ is $n_H+r-2s^2/3=2\|H_{\rm d}\|_F^2$.
Using $2\tr(\cof H)=s^2-r$ then yields
\begin{equation}
 \mathcal H_c[H,H]+(2cI_3):(H\times H)
 =c\mathcal A[H,H],
 \label{eq:stress-free-translation-identity}
\end{equation}
where $\mathcal A=A_1(I_3)$ is the metric in
\eqref{eq:rest-material-metric}.

For $U\in\mathbb M$, introduce the symmetric bilinear form
\begin{equation}
 \mathcal C_U[H,L]=U:(H\times L).
 \label{eq:cofactor-translation-form}
\end{equation}
It depends linearly on $U$. We use the same symbols for
$\mathcal A$ and $\mathcal C_U$ and their nine-by-nine matrices in
the Frobenius-orthonormal entry basis. In particular,
$\mathcal A^{-1/2}$ denotes the inverse positive square root of that
positive matrix.

\subsection{A convex coefficient loss}

We now allow the physical coefficient to vary through a continuous
potential. Define
\begin{equation}
 \nu(U)
 =-\lambda_{\min}(\mathcal A^{-1/2}\mathcal C_U\mathcal A^{-1/2})
 =\sup_{H\ne0}\frac{-U:(H\times H)}{\mathcal A[H,H]},
 \qquad U\in\mathbb M.
 \label{eq:cofactor-coefficient-loss}
\end{equation}
Here $\lambda_{\min}$ is the smallest ordinary eigenvalue of the
displayed symmetric matrix. The sign selects the loss of a positive
reference margin caused by the cofactor translation.

\begin{lemma}[Coefficient-loss bounds]
\label{lem:cofactor-loss-bounds}
The function $\nu$ is continuous, convex, and positively homogeneous
for nonnegative scalars. It vanishes only at zero and satisfies
\begin{equation}
 \frac12\|U\|_{\max}\le\nu(U),\qquad
 \frac16\|U\|_F\le\nu(U)\le\|U\|_2,
 \quad\|U\|_{\max}:=\max_{i,j}|U_{ij}|.
 \label{eq:cofactor-loss-bounds}
\end{equation}
In general $\nu(U)\ne\nu(-U)$, so $\nu$ is a convex loss rather
than a norm.
\end{lemma}

\begin{proof}
The supremum in \eqref{eq:cofactor-coefficient-loss} gives convexity
and positive homogeneity; continuity follows from eigenvalue
continuity. The matrix of $\mathcal C_U$ has zero diagonal because
every matrix unit is rank one. The coefficient map is injective:
two complementary matrix units isolate each entry of $U$ in the
cofactor polarization. Thus for $U\ne0$ this is a nonzero symmetric
trace-zero matrix. It has both signs, as does its congruence by
$\mathcal A^{-1/2}$, proving $\nu(U)>0$.

By definition $\nu(U)\mathcal A+\mathcal C_U\succeq0$.
For each entry $U_{ij}$ there are distinct off-diagonal matrix units
$H_1,H_2$ with $H_1\times H_2=\pm E_{ij}$.
If $i=j$, use the opposite off-diagonal entries in the complementary
two-by-two block. If $i\ne j$ and $k$ is the remaining index, use
$E_{jk}$ and $E_{ki}$. For example,
\[
 E_{23}\times E_{31}=E_{12},\qquad
 E_{23}\times E_{32}=-E_{11}.
\]
These increments have unit Frobenius norm, zero trace, and zero
mutual Frobenius product. The associated principal block is
\[
 \begin{pmatrix}2\nu(U)&\pm U_{ij}\\
                 \pm U_{ij}&2\nu(U)\end{pmatrix}\succeq0.
\]
Its determinant gives $|U_{ij}|\le2\nu(U)$.
The bound $\|U\|_F\le3\|U\|_{\max}$ proves the Frobenius estimate.

For the upper bound, let $\sigma_1,\sigma_2,\sigma_3$ be the
singular values of $H$. The nuclear norm, denoted by
$\|\cdot\|_*$, is the sum of singular values. Hence
\begin{align*}
 |\mathcal C_U[H,H]|
 &=2|U:\cof H|
 \le2\|U\|_2\|\cof H\|_*\\
 &=2\|U\|_2(\sigma_1\sigma_2+\sigma_1\sigma_3+
                         \sigma_2\sigma_3)\\
 &\le2\|U\|_2\|H\|_F^2
 \le\|U\|_2\mathcal A[H,H].
\end{align*}
The penultimate inequality follows from the sum of the three squared
singular-value differences.

To see the asymmetry explicitly, take
$U=\operatorname{diag}(1,1,-1)$. On the diagonal increments,
the two matrices are
\[
 \mathcal A_{\rm diag}=2I_3+\tfrac23\mathbf1_3\mathbf1_3^T,
 \qquad
 (\mathcal C_U)_{\rm diag}=
 \begin{pmatrix}0&-1&1\\-1&0&1\\1&1&0\end{pmatrix},
\]
where $\mathbf1_3=(1,1,1)^T$.
Their generalized eigenvalues are $1/2$ and
$(-2\pm\sqrt{13})/6$: the direction $(1,-1,0)^T$ gives $1/2$,
and the remaining two-dimensional determinant gives
$12\lambda^2+8\lambda-3=0$.
Each off-diagonal increment pair has generalized eigenvalues
$1/2$ and $-1/2$. Therefore
\[
 \nu(U)=\frac{2+\sqrt{13}}6,\qquad \nu(-U)=\frac12.
\]
\end{proof}

\subsection{Optimization over exact continuous potentials}

Fix the geometry, the finite nonempty quadrature point set on each
cell, the unrestricted local variation spaces $\widehat V_e$, and
the finite-dimensional continuous candidate space $\mathcal Z_h$.
Include constant vectors and impose no additional prescribed potential
trace. Its cellwise exact subspace is
\[
 \mathcal Z_h^0=
 \{z\in\mathcal Z_h:
   \mathcal T_e(z_e;u,v)=0
   \text{ for every }e\text{ and all }u,v\in\widehat V_e\}.
\]
This is a linear space, so the zero potential is feasible.
For a candidate $z$, let
\begin{equation}
 \gamma_h(z)=\min_{e,q}\inf_{H\ne0}
 \frac{\mathcal H_c[H,H]+R_z(X_e(\xi_{eq})):(H\times H)}
      {\mathcal A[H,H]}.
 \label{eq:quadrature-point-material-margin}
\end{equation}
Every matrix increment is tested at each selected point. The minimum
does not range over all spatial points of the cell, and displacement
boundary constraints do not restrict its matrix tests.

\begin{theorem}[Attained optimal compatible margin]
\label{thm:optimal-compatible-margin}
For every $z\in\mathcal Z_h$,
\begin{equation}
 \gamma_h(z)=c-\max_{e,q}
       \nu\bigl(R_z(X_e(\xi_{eq}))-2cI_3\bigr).
 \label{eq:compatible-margin-formula}
\end{equation}
The best compatible margin is attained and satisfies
\begin{equation}
 \gamma_h^*:=\max_{z\in\mathcal Z_h^0}\gamma_h(z)
 =c-\min_{z\in\mathcal Z_h^0}\max_{e,q}
       \nu\bigl(R_z(X_e(\xi_{eq}))-2cI_3\bigr),
 \qquad 0\le\gamma_h^*\le c.
 \label{eq:optimal-compatible-margin}
\end{equation}
\end{theorem}

\begin{proof}
At each quadrature point, \eqref{eq:stress-free-translation-identity}
gives
\[
 \mathcal H_c+\mathcal C_{R_z}
 =c\mathcal A+\mathcal C_{R_z-2cI_3}.
\]
Taking the smallest generalized eigenvalue proves
\eqref{eq:compatible-margin-formula}.
The evaluation image
\[
 \mathscr E_h=
 \{(R_z(X_e(\xi_{eq})))_{e,q}:z\in\mathcal Z_h^0\}
\]
is a closed linear subspace of a finite-dimensional matrix product space.
The Frobenius lower bound in \eqref{eq:cofactor-loss-bounds} makes
each sublevel set of the loss in \eqref{eq:optimal-compatible-margin}
bounded in that image. Continuity therefore gives a minimizing
evaluation vector and a potential realizing it. Constants and any other
invisible potential directions need not have unique coefficients.
The zero potential has margin zero, since $\mathcal H_c$ is
nonnegative and vanishes on skew increments. Nonnegativity of $\nu$
gives the upper bound $c$.
\end{proof}

In a basis of $\mathcal Z_h$, this is a finite semidefinite program:
maximize $\gamma$ subject to the linear exactness equations and
\[
 \mathcal H_c+\mathcal C_{R_z(X_e(\xi_{eq}))}
       -\gamma\mathcal A\succeq0
 \quad\text{for every selected }(e,q).
\]
Continuity is already incorporated in the candidate basis. Affine
matrix inequalities and this eigenvalue optimization are standard
semidefinite-programming constructions
\citep{vandenberghe1996semidefinite}. The theorem identifies the
best compatible material certificate in the prescribed space; it
supplies no optimizer cost or robustness result.
\Cref{sec:enrichment-loss} uses explicit potentials and analytic bounds.

\begin{corollary}[Rigidity of the ideal margin]
\label{cor:optimal-margin-rigidity}
Suppose in addition that the physical identity potential
$\iota(X)=X$ belongs to $\mathcal Z_h$. Assume gradient-sampling
unisolvence: for every candidate $z$ and every cell,
vanishing of $D_\xi(z_e-2cX_e)$ at all selected cell points implies
that this parent gradient vanishes identically. On a connected mesh,
\begin{equation}
 \gamma_h^*=c
 \quad\Longleftrightarrow\quad
 \iota\in\mathcal Z_h^0.
 \label{eq:ideal-margin-rigidity}
\end{equation}
Every potential attaining this ideal margin equals $2c\iota+b$
for a constant vector $b\in\R^3$.
\end{corollary}

\begin{proof}
Attainment and \eqref{eq:compatible-margin-formula} show that an
ideal-margin potential satisfies $R_z=2cI_3$ at every selected point.
Invertibility of $J_{X,e}$ then gives zero sampled parent gradient
of $z_e-2cX_e$. Unisolvence makes the difference constant on each
cell. Continuity and connectedness give one global constant vector.
Constants belong to the exact kernel, so $\iota\in\mathcal Z_h^0$.
Conversely, $z=2c\iota$ is feasible under that condition and has
margin $c$ by \eqref{eq:stress-free-translation-identity}.
\end{proof}

The required unisolvence holds for quadratic simplex gradients sampled
at the four affinely independent Tet10 Q4 points, and for tensor
$\mathbb Q_p$ gradients on at least $p+1$ distinct points per
coordinate. With full local $\mathbb Q_p^3$ variations, potential
pullbacks in $\mathbb Q_p^3$, $p\ge2$, geometry
$X_e\in\mathbb Q_p^3$, and a common Gauss order $n\ge p+1$,
the tensor kernel therefore makes the ideal
margin attainable precisely when every $X_e\in\mathbb Q_r^3$,
provided the identity belongs to the global candidate space.
Here again $r=\min\{p,2(n-p)\}$ for the stated isotropic rule.

Enlarging the local variation spaces, while keeping geometry, candidate
space, and quadrature fixed, can only shrink $\mathcal Z_h^0$;
its optimal margin cannot increase. This statement concerns a finite-point
material certificate. Assembled constraints can remove unfavorable
matrix directions, and potentials with the same boundary trace can
represent the same exactly assembled gauge. Their local margins
therefore do not determine the ordering of complete operator spectra.
The next section gives a strict, quantitative instance of the
compatible-margin restriction.

\section{A strict loss of compatible margin under enrichment}
\label{sec:enrichment-loss}

The kernel inclusion associated with displacement enrichment can impose a
strict restriction on every compatible potential. We exhibit a conforming
mesh family for which the best material margin decreases linearly with the
curvature amplitude. Geometry, cubature, and the candidate-potential space
are held fixed in each comparison; only the unrestricted local variation
space changes. The mechanical state is the identity deformation on the
curved reference domain, with the common material sum $c>0$ used in
\cref{sec:optimal-margin}.

\subsection{Conforming checkerboard geometry}

Partition the parameter cube $[0,1]^3$ into $N^3$ cubes of side $h=1/N$,
where $N$ is a positive integer. Index the cubes by $e=(i,j,k)$, with
$i,j,k\in\{0,\ldots,N-1\}$, and define the affine maps
\[
 T_e(\xi)=\frac h2
       (2i+1+\xi_1,\,2j+1+\xi_2,\,2k+1+\xi_3)^T,
 \qquad\xi\in[-1,1]^3.
\]
Let $e_1=(1,0,0)^T$ and $\theta_e=(-1)^{i+j+k}$. For a real
amplitude $\alpha$, set
\begin{equation}
 X_e(\xi)=T_e(\xi)
       +\frac{h\alpha}{2}\theta_e\xi_1^3\xi_2\xi_3e_1.
 \label{eq:checkerboard-geometry}
\end{equation}
The induced global map on the parameter cube is denoted by $X_h$;
its restriction to cell $e$ is $X_e\circ T_e^{-1}$.

Across a shared coordinate face, $\theta_e$ changes sign, as does the
corresponding odd local-coordinate factor in the perturbation. The other
local coordinates agree. Thus the perturbation has the same trace from
both cells, and $X_h$ is continuous. With $t$ denoting global parameter
coordinates, its derivative on each cell is
\[
 D_tX_h-I_3
 =\alpha\theta_e e_1\otimes
       (3\xi_1^2\xi_2\xi_3,\,\xi_1^3\xi_3,\,\xi_1^3\xi_2).
\]
Consequently, for $\delta=\sqrt{11}|\alpha|$,
\begin{equation}
 \|D_tX_h-I_3\|_2\le\delta,\qquad
 \det D_tX_h
 =1+3\alpha\theta_e\xi_1^2\xi_2\xi_3
 \ge1-3|\alpha|.
 \label{eq:checkerboard-geometry-bounds}
\end{equation}
These estimates hold throughout every cell.

Continuity and the piecewise derivative bound make $X_h(t)-t$
globally Lipschitz on the convex parameter cube, with constant at most
$\delta$. This follows by integrating along segments through the cells;
continuity covers segments lying on cell faces. If $\delta<1$, then
\begin{equation}
 (1-\delta)\|t-s\|_2
 \le\|X_h(t)-X_h(s)\|_2
 \le(1+\delta)\|t-s\|_2.
 \label{eq:checkerboard-bilipschitz}
\end{equation}
The map is therefore one-to-one. In the range used below,
$\delta<1/3$, the determinant bound is positive as well. The physical
reference domain is $\Omega_h=X_h((0,1)^3)$. Its cell Jacobians obey
\[
 \|D_\xi X_e\|_2\le\frac h2(1+\delta),\qquad
 \|(D_\xi X_e)^{-1}\|_2\le\frac{2}{h(1-\delta)}.
\]
Thus the reference maps have shape bounds independent of $h$ after the
usual affine scaling.

Use the same tensor Gauss4 rule, with all $64$ points per cell, in both
comparisons. The candidate space is also identical:
\begin{equation}
 \mathcal Z_h=
 \{z\in C^0(\overline\Omega_h;\R^3):
             z\circ X_e\in\mathcal S_3^3\text{ on every cell}\}.
 \label{eq:checkerboard-candidate-space}
\end{equation}
No potential trace is prescribed. Compare the unrestricted local
variation spaces $\widehat V_e^S=\mathcal S_3^3$ and
$\widehat V_e^Q=\mathbb Q_3^3$. Their cellwise exact subspaces of
\eqref{eq:checkerboard-candidate-space} are denoted by
$\mathcal Z_{S,h}^0$ and $\mathcal Z_{Q,h}^0$, respectively.
Write $\gamma_{S,h}^*(\alpha)$ and $\gamma_{Q,h}^*(\alpha)$ for the
two attained optima in \eqref{eq:optimal-compatible-margin}.
These are material margins at the selected quadrature points, tested
on all matrix increments before displacement boundary elimination.

By \cref{thm:serendipity-compatible-kernel}, the $\mathcal S_3$ exact class
contains each pullback $X_e$ in \eqref{eq:checkerboard-geometry}.
For the $\mathbb Q_3$ variations, \cref{thm:tensor-compatible-kernel} gives
$\mathbb Q_2$ as the scalar potential kernel. Intersecting with the fixed
$\mathcal S_3$ candidate space gives
\begin{equation}
 \mathcal S_3\cap\mathbb Q_2=\mathcal S_2,
 \qquad
 \mathcal Z_{Q,h}^0
 =\{z\in\mathcal Z_h:z\circ X_e\in\mathcal S_2^3
                           \text{ for every }e\}.
 \label{eq:enriched-exact-potential-class}
\end{equation}
Indeed, an $\mathcal S_3$ monomial with no exponent above two has at most one
exponent equal to two, which is exactly the $\mathcal S_2$ condition. The
variation enrichment removes the mixed cubic potential modes without
enlarging the candidate space.

\subsection{A quantitative restriction on the optimum}

Let $g_+>g_->0$ be the two positive univariate Gauss4 nodes, with
\[
 g_\pm^2=\frac{15\pm2\sqrt{30}}{35}.
\]
Define the fixed positive constant
\begin{equation}
 K_G=g_+^2(g_+^2-g_-^2)
     =\frac{12(4+\sqrt{30})}{245}.
 \label{eq:gauss-margin-constant}
\end{equation}

\begin{theorem}[Strict compatible-certificate loss]
\label{thm:enrichment-certificate-loss}
For the geometry, point set, and common candidate space above, assume
$0<|\alpha|<1/(3\sqrt{11})$ and put $\delta=\sqrt{11}|\alpha|$.
Then
\begin{equation}
 \gamma_{S,h}^*(\alpha)=c,
 \qquad
 c\frac{1-3\delta}{1-\delta}
 \le\gamma_{Q,h}^*(\alpha)
 \le c\left(1-\frac{K_G|\alpha|}{2(1+\delta)}\right)<c.
 \label{eq:checkerboard-margin-bounds}
\end{equation}
Both exact-potential classes therefore admit positive material margins.
For every fixed $0<\alpha_0<1/(3\sqrt{11})$, let
$\delta_0=\sqrt{11}\alpha_0$. Uniformly in $h$ and
$0<|\alpha|\le\alpha_0$,
\begin{equation}
 \frac{cK_G}{2(1+\delta_0)}|\alpha|
 \le c-\gamma_{Q,h}^*(\alpha)
 \le\frac{2c\sqrt{11}}{1-\delta_0}|\alpha|.
 \label{eq:checkerboard-linear-loss}
\end{equation}
Thus the optimal-margin loss has linear order in $|\alpha|$, with
constants independent of the mesh size. These constants are not
asserted to be optimal.
\end{theorem}

\begin{proof}
The physical identity potential $z=2c\iota$ has pullbacks $2cX_e$,
which belong to $\mathcal Z_{S,h}^0$ by the serendipity kernel.
Its physical coefficient is $R_z=2cI_3$, so
\eqref{eq:stress-free-translation-identity} gives margin $c$.
\Cref{thm:optimal-compatible-margin} supplies the matching upper
bound, proving the $\mathcal S_3$ equality.

For an explicit $\mathbb Q_3$-compatible construction, define
\begin{equation}
 z(X_h(t))=2ct.
 \label{eq:checkerboard-affine-potential}
\end{equation}
Its pullbacks are $2cT_e$, so the potential is continuous and belongs
to $\mathcal Z_{Q,h}^0$. The chain rule gives
$R_z=2c(D_tX_h)^{-1}$. From
\eqref{eq:checkerboard-geometry-bounds},
\[
 \|(D_tX_h)^{-1}-I_3\|_2
 \le\frac\delta{1-\delta}.
\]
Using the upper estimate for $\nu$ in
\cref{lem:cofactor-loss-bounds},
\[
 \gamma_h(z)
 \ge c-\frac{2c\delta}{1-\delta}
 =c\frac{1-3\delta}{1-\delta}>0.
\]
This proves the lower bound on the $\mathbb Q_3$ optimum.

For the upper bound, take any $z\in\mathcal Z_{Q,h}^0$, put
$\gamma=\gamma_h(z)$, and write $d_\gamma=c-\gamma\ge0$.
At each selected point define $U=R_z-2cI_3$. Equation
\eqref{eq:compatible-margin-formula} and
\cref{lem:cofactor-loss-bounds} imply
\begin{equation}
 |U_{ij}|\le2d_\gamma,\qquad \|U\|_F\le6d_\gamma.
 \label{eq:enrichment-coefficient-bound}
\end{equation}
For a fixed cell set
$\widehat J_X=(2/h)D_\xi X_e=D_tX_h$, so
$\|\widehat J_X\|_2\le1+\delta$. The physical-to-parent chain rule
gives the exact identity
\begin{equation}
 \frac2hD_\xi z_e-2c\widehat J_X=U\widehat J_X.
 \label{eq:checkerboard-normalized-chain-rule}
\end{equation}
Every entry on the left has magnitude at most
$6d_\gamma(1+\delta)$ at the quadrature points.

Fix the two transverse coordinates at
$\xi_2=\xi_3=g_+$. Since $z_e\in\mathcal S_2^3$, the scalar
polynomial
\[
 p(x)=\frac2h\partial_{\xi_1}(z_e)_1(x,g_+,g_+)
\]
is affine in $x$. Its even averages at $g_+$ and $g_-$ agree.
The corresponding target entry of $2c\widehat J_X$ is
\[
 2c(1+3\alpha\theta_e x^2g_+^2).
\]
Its two even averages differ by
$6c\alpha\theta_e g_+^2(g_+^2-g_-^2)=6c\alpha\theta_e K_G$.
All four points $(\pm g_+,g_+,g_+)$ and
$(\pm g_-,g_+,g_+)$ belong to the fixed quadrature set.
The error bound from \eqref{eq:checkerboard-normalized-chain-rule}
therefore gives
\[
 6c|\alpha|K_G\le12d_\gamma(1+\delta),
 \qquad
 c-\gamma\ge\frac{cK_G|\alpha|}{2(1+\delta)}.
\]
This holds for every admitted potential, and hence for the attained
optimum. It proves the strict upper bound in
\eqref{eq:checkerboard-margin-bounds}.
The factor $2/h$ is retained throughout the chain rule; its
cancellation is what gives a bound in physical gradients independent
of $h$. Finally, the lower candidate bound gives
$c-\gamma_{Q,h}^*\le2c\delta/(1-\delta)$.
Bounding $\delta$ by $\delta_0$ in both denominators yields
\eqref{eq:checkerboard-linear-loss}.
\end{proof}

The theorem quantifies the restriction on every cellwise-exact
potential in the fixed candidate class. It also supplies positive
explicit choices for both variation spaces. No numerical optimizer
is needed to establish the strict loss or its linear order.
The pointwise constants and the geometric bounds are uniform in $h$.
Conversion to a physical $H^1$ reference bound additionally uses a
fixed Poincar\'e inequality and the quadrature comparison for squared
parent gradients, as developed in the reference-form analysis below.

The reference maps in this family oscillate on the cell scale. Their
uniform geometric bounds do not make this a convergence study for
approximation of a single fixed smooth geometry. The example has
fixed degree three; its refinement estimate introduces no additional
polynomial-order-uniform assertion.

Finally, a loss of this material certificate does not order the
eigenvalues of the assembled reference forms or of the complete
mechanical tangents. The later boundary representation makes the
distinction explicit: potentials with the same external trace give
the same exactly integrated gauge on a fixed variation space, and a
complete displacement clamp removes its boundary Hessian. The present
comparison establishes a restriction caused by variation-space
enrichment. It provides no solver-speed or general finite-strain
stability conclusion.

\section{The constant-coefficient comparison}
\label{sec:constant-margin-comparison}

The potential-space problem maximizes material margin subject to exact
compatibility. For a separate restricted-class comparison, take the
potential \(z(X)=RX\) with a fixed physical matrix \(R\in\mathbb M\).
The following comparison instead prescribes a positive margin and
asks how small the local quadrature defect can be.
It quantifies the limitation of this restricted class in a specified
basis and norm.

Retain the stress-free state, common material sum \(c>0\), Hessian
\(\mathcal H_c\), and metric \(\mathcal A\) of
\cref{sec:optimal-margin}. This separate comparison does not impose
membership of \(z(X)=RX\) in an otherwise prescribed global candidate space.

Fix one parent cell, its geometry $X_e$, and its error functional
$\mathcal L_e=\mathcal Q_e-\mathcal I_e$. Choose scalar basis functions
$\phi_1,\ldots,\phi_d$, where $d$ is their number, and use this same
basis for all three physical components of the local variations. For a
scalar field $f$, define the skew-symmetric matrix
\[
 \mathsf B(f)_{ab}
 =\mathcal L_e\!\left[
 D_\xi f\cdot
    (D_\xi\phi_a\times_{\R^3}D_\xi\phi_b)\right],
 \qquad 1\le a,b\le d.
\]
For a constant physical matrix $R\in\mathbb M$, let
$\mathsf E_R$ be the matrix of
$\mathcal T_e(RX_e;\cdot,\cdot)$ in the corresponding vector basis.
This is the cubature defect of the constant-coefficient cofactor gauge,
by \eqref{eq:parent-physical} and \eqref{eq:common-trilinear-defect}.
Write $X_{e,i}$ for component $i$ of the geometry, and define
\[
 (\mathsf S_X)_{ij}
 =\langle\mathsf B(X_{e,i}),\mathsf B(X_{e,j})\rangle_F,
 \qquad i,j\in\{1,2,3\},
\]
where $\langle A,B\rangle_F=\tr(A^TB)$ also denotes the matrix
Frobenius product for these $d$-by-$d$ matrices.

\begin{proposition}[Constant-gauge defect Gram identity]
\label{prop:constant-gauge-gram}
The matrix $\mathsf S_X$ is positive semidefinite, and
\begin{equation}
 \|\mathsf E_R\|_F^2=2\tr(R\mathsf S_XR^T),\qquad
 \|\mathsf E_{I_3}\|_F^2=2\tr\mathsf S_X.
 \label{eq:constant-gauge-gram}
\end{equation}
The identities hold for the fixed rule and basis, without a positivity
or exactness assumption on that rule.
\end{proposition}

\begin{proof}
The matrix $\mathsf S_X$ is a Gram matrix and hence positive
semidefinite. Put $f_i=(RX_e)_i$ and $B_i=\mathsf B(f_i)$.
Ordering the variation coordinates by physical component gives
\[
 \mathsf E_R=
 \begin{pmatrix}
 0&B_3&-B_2\\
 -B_3&0&B_1\\
 B_2&-B_1&0
 \end{pmatrix}.
\]
Indeed, two variations supported in distinct physical components
leave only the complementary row in their cofactor polarization.
Their scalar pairing is precisely the indicated entry of $B_i$,
with the sign of the component permutation.
Since $B_i^T=-B_i$, the displayed full matrix is symmetric, and
its squared Frobenius norm is $2\sum_i\|B_i\|_F^2$.
Linearity gives
$B_i=\sum_jR_{ij}\mathsf B(X_{e,j})$, proving
\eqref{eq:constant-gauge-gram}. A reordering from component-major to
node-major coordinates preserves this norm.
\end{proof}

For a nonnegative pointwise volumetric term, fix $\kappa\ge0$ and
write
\[
 \mathcal G_R^\kappa[H,L]
 =\mathcal H_c[H,L]+\mathcal C_R[H,L]
      +\kappa(\tr H)(\tr L).
\]
The last term is the resting Hessian of
$\kappa(\det F-1)^2/2$. At $\kappa=0$ this is the isochoric
translated form.

\begin{theorem}[Minimum local defect at prescribed coercivity]
\label{thm:optimal-rest-defect}
For $0<\gamma\le c$ and any fixed $\kappa\ge0$,
\begin{equation}
 \min_{\mathcal G_R^\kappa\succeq\gamma\mathcal A}
       \|\mathsf E_R\|_F
 =2\gamma\|\mathsf E_{I_3}\|_F.
 \label{eq:minimum-constant-gauge-defect}
\end{equation}
The minimum is attained by $R=2\gamma I_3$. No coefficient satisfies
the prescribed inequality for $\gamma>c$.
\end{theorem}

\begin{proof}
For $a\in\R^3$, let $H_a$ be the skew matrix defined by
$H_ax=a\times_{\R^3}x$. Direct calculation gives
$\cof H_a=aa^T$, $\tr H_a=0$, and
$\mathcal A[H_a,H_a]=4|a|^2$. The material Hessian vanishes on
$H_a$, so the assumed inequality gives
\[
 2a^TRa\ge4\gamma|a|^2.
\]
Thus $(R+R^T)/2\succeq2\gamma I_3$. Cauchy--Schwarz implies
$\|Ra\|_2\ge2\gamma\|a\|_2$, or
$R^TR\succeq4\gamma^2I_3$. Using the positive square root of the
Gram matrix in \eqref{eq:constant-gauge-gram},
\[
 \tr(R\mathsf S_XR^T)
 =\tr(\mathsf S_X^{1/2}R^TR\mathsf S_X^{1/2})
 \ge4\gamma^2\tr\mathsf S_X.
\]
This gives the lower bound in
\eqref{eq:minimum-constant-gauge-defect}.

For $R=2\gamma I_3$, \eqref{eq:stress-free-translation-identity}
gives
\[
 (\mathcal G_R^\kappa-\gamma\mathcal A)[H,H]
 =4(c-\gamma)\|H_{\rm d}\|_F^2+\kappa(\tr H)^2\ge0.
\]
Linearity of $\mathsf E_R$ proves attainment of the bound.
Finally, for the matrix unit $E_{12}$, which has a single unit entry
in position $(1,2)$, the cofactor translation and volumetric term vanish.
Its material and reference values are $2c$ and $2$, respectively,
so the inequality requires $\gamma\le c$.
\end{proof}

The norm in this theorem belongs to the stated common scalar basis and
the full component space. Component-dependent boundary elimination,
a different matrix norm, or an assembled-only positivity requirement
defines a different optimization problem. In the Tet10
midside-deviation coordinates of \cref{cor:faithful-jet},
$\mathsf E_{I_3}=\mathscr R(d_e(X))$, where $\mathscr R$ is the
eighteen-coordinate defect Hessian and $d_e(X)$ contains the six
reference midside deviations. Equation \eqref{eq:jet-exact-norm}
therefore recovers the Tet10 minimum in those coordinates. With the full
local quadratic variation space, every genuinely curved Tet10 cell has a
positive minimum at positive margin.

The associated geometry restriction follows directly from the kernel
classification. For full local $\mathbb Q_p^3$ variations with $p\ge2$, geometry
$X_e\in\mathbb Q_p^3$, and tensor Gauss with $n\ge p+1$, put
\[
 r=\min\{p,2(n-p)\}.
\]
By \cref{thm:tensor-compatible-kernel}, the constant coefficient $R$
is compatible precisely when $RX_e\in\mathbb Q_r^3$. Equivalently,
$R$ annihilates every reference-coordinate coefficient with a
coordinate degree above $r$. A positive prescribed margin makes $R$
invertible by the skew test, so such a compatible coefficient exists
at rest precisely when $X_e\in\mathbb Q_r^3$; the identity coefficient
$2\gamma I_3$ supplies the converse for $\gamma\le c$.
The usual $n=p+1$ choice has $r=2$. Overintegration can enlarge the
class. The Tet10 alignment criterion in
\cref{prop:exact-gauge-alignment} is its separate simplex counterpart.

\section{Complete reference forms and boundary compensation}
\label{sec:complete-reference-forms}
\label{sec:compatible-potentials}
\label{sec:defect-corrected-core}

A compatible potential and its material margin specify a local
reference construction. We now identify its assembled operator and
the compensation that preserves the complete equation.
Potential-generated null Lagrangians supply
the exact boundary identity
\citep{edelen1986null,olver1988structure,kupferman2019piola}.
Their contribution to natural boundary conditions is consequential
\citep{olver2024boundary}. We use this identity to represent the complete
stored-energy Hessian from \eqref{eq:complete-pressure-energy}.
No physical energy or boundary condition is replaced by its gauge.

\subsection{The exact and numerical added forms}

Use the functionals \(\mathscr J_z^I,\mathscr J_z^Q\) and their
Hessians \(\mathcal C_h^I[z_h],\mathcal C_h^Q[z_h]\) from
\cref{sec:gauge-design-levels}. Cellwise compatibility makes the
two assembled Hessians equal.

This last statement concerns Hessians. If the current parent maps lie in
the same linear spaces $\widehat V_e$ used to define the kernel, equality
of the full functionals follows as well, because
\[
\mathcal L_e[J_z:\cof J_y]
=\tfrac12\mathcal T_e(z_e;y_e,y_e).
\]
For a more general affine family of current maps, Hessian equality permits
a remaining affine function of the free displacement coefficients.
The reference-form conclusions below require the Hessian identity.
They always retain the original energy and residual.

Let $n_X$ be the outward unit normal on the physical reference boundary.
Let $R_\Gamma$ extract the free external displacement trace from a global
variation vector. A boundary matrix is understood in these trace
coordinates; interior degrees of freedom are those in $\ker R_\Gamma$.

\begin{theorem}[Compatible boundary reference form]
\label{thm:compatible-boundary-gauge}
Assume the fields are continuous and sufficiently piecewise regular for
the displayed boundary integrals; the polynomial finite-element families
considered here satisfy this condition. The exact gauge functional obeys
\begin{equation}
\mathscr J_z^I(y_h)
=\int_{\Omega_h}\nabla_Xz_h:\cof(\nabla_Xy_h)\,dX
=\int_{\partial\Omega_h}
 z_h\cdot\{\cof(\nabla_Xy_h)n_X\}\,dS.
\label{eq:potential-boundary-functional}
\end{equation}
There is a symmetric matrix $B_\Gamma[z_h]$, independent of $y_h$, such that
\begin{equation}
\mathcal C_h^I[z_h]=R_\Gamma^TB_\Gamma[z_h]R_\Gamma.
\label{eq:exact-boundary-gauge}
\end{equation}
If $z_h\in\mathcal Z_h^0$, define the numerical reference form
\[
G_{z,h}(y_h)=K_h(y_h)+\mathcal C_h^Q[z_h].
\]
Then
\begin{equation}
K_h(y_h)=G_{z,h}(y_h)-R_\Gamma^TB_\Gamma[z_h]R_\Gamma.
\label{eq:compatible-complete-identity}
\end{equation}
Whenever $G_{z,h}(y_h)\succ0$, the complete nonpositive index is at most
$\rank R_\Gamma$, and the interior block is positive definite.
With a full displacement clamp, $K_h(y_h)=G_{z,h}(y_h)$.
\end{theorem}

\begin{proof}
Cofactor multiplicativity gives
$\det J_X\cof F=(\cof J_y)J_X^T$.
Taking its Frobenius product with $R_z=J_zJ_X^{-1}$ proves
\[
\det J_X\,R_z:\cof F=J_z:\cof J_y.
\]
This gives the first equality in \eqref{eq:potential-boundary-functional}.
The row-wise Piola identity is
$\operatorname{Div}_X\cof(\nabla_Xy_h)=0$.
The product rule therefore converts
$\nabla_Xz_h:\cof(\nabla_Xy_h)$ to a divergence on each cell.
Its normal cofactor flux depends on tangential derivatives of the
deformation trace. The common trace and opposite face orientations give
opposite fluxes on an internal face. Continuity of $z_h$ cancels those
contributions, leaving the stated external integral.

The boundary functional depends on $y_h$ only through its boundary trace.
Its second variation consequently vanishes whenever either argument has
zero full boundary trace. Essential conditions have already removed
variations on $\Gamma_D$. The bilinear form therefore descends to the
free trace space and has the matrix representation
\eqref{eq:exact-boundary-gauge}. Its symmetry follows from being a
Hessian, and its independence of $y_h$ follows from cofactor quadraticity.

Cellwise compatibility makes $\mathcal C_h^Q[z_h]=\mathcal C_h^I[z_h]$,
which proves \eqref{eq:compatible-complete-identity}.
On $\ker R_\Gamma$, the complete form equals $G_{z,h}$.
If the latter is positive, trace extraction is injective on any subspace
where $K_h$ is nonpositive. Such a subspace has dimension at most
$\rank R_\Gamma$. The interior block is the positive restriction to
$\ker R_\Gamma$. The full-clamp case has no free external trace.
\end{proof}

The local kernel is tested before displacement boundary elimination.
Additional assembled cancellations may make
$\mathcal C_h^Q[z_h]=\mathcal C_h^I[z_h]$ for potentials outside that
cellwise class. The theorem supplies a constructive sufficient route
through the classified kernel; it does not characterize every
assembled-only cancellation.

A discontinuous potential leaves interface contributions in the exact
boundary calculation. They can be retained explicitly, but they are not
removed by the continuity argument above. The current results use the
continuous candidate space.

\subsection{Constant coefficients and explicit reference-defect compensation}

For a constant physical coefficient $Z\in\mathbb M$, the associated
potential is $z_h(X)=ZX$. Its parent maps are $ZX_e$.
In this subsection it is sufficient to consider this smooth potential
directly; it need not belong to a separately chosen optimization class.
Denote its exact and numerical gauge Hessians by
$C_{Z,h}^I$ and $C_{Z,h}^Q$, and put
\[
\Delta_{Z,h}=C_{Z,h}^Q-C_{Z,h}^I,\qquad
G_h(y_h)=K_h(y_h)+C_{Z,h}^Q.
\]
The matrices $C_{Z,h}^I$, $C_{Z,h}^Q$, and $\Delta_{Z,h}$ depend on
the reference geometry, spaces, quadrature and coefficient, and are fixed
as $y_h$ changes.

\begin{theorem}[Static compatible constant coefficient]
\label{thm:static-gauge-boundary}
If $\mathcal T_e(ZX_e;u_e,v_e)=0$ for every local variation pair on
every cell,
then $\Delta_{Z,h}=0$ and
\[
K_h(y_h)=G_h(y_h)-R_\Gamma^TB_\Gamma(Z)R_\Gamma.
\]
The positive-core, trace-index and full-clamp conclusions of
\cref{thm:compatible-boundary-gauge} apply.
For quadratic tetrahedral geometry and variations with the symmetric
four-point rule, the local condition is precisely that $ZX_e$ be
parent-affine.
\end{theorem}
\begin{proof}
Use the fixed potential $z_h=Z\iota$ in
\cref{thm:compatible-boundary-gauge}. The local hypothesis makes its
gauge Hessian exact. The Tet10 characterization is
\cref{prop:exact-gauge-alignment}; the coefficient is allowed to be
singular. No positivity assumption is needed for the algebraic identity.
\end{proof}

For a general fixed continuous potential, define
\[
\Delta_h[z_h]=\mathcal C_h^Q[z_h]-\mathcal C_h^I[z_h],
\qquad
\widehat G_{z,h}(y_h)
=K_h(y_h)+\mathcal C_h^Q[z_h]-\Delta_h[z_h].
\]

\begin{theorem}[Reference-only defect-corrected form]
\label{thm:reference-corrected-core}
The exact identity
\begin{equation}
K_h(y_h)=\widehat G_{z,h}(y_h)
             -R_\Gamma^TB_\Gamma[z_h]R_\Gamma
\label{eq:defect-corrected-complete-identity}
\end{equation}
holds without cellwise compatibility. Let $A_h\succ0$ be a specified
reference matrix and let $\gamma>0$. If
\[
K_h(y_h)+\mathcal C_h^Q[z_h]\succeq\gamma A_h,\qquad
\gamma A_h-\Delta_h[z_h]\succ0,
\]
then $\widehat G_{z,h}(y_h)\succ0$.
Its complete nonpositive-index and interior-block consequences are the
same as in \cref{thm:compatible-boundary-gauge}.
Both the defect and its exact boundary compensation may be precomputed
for fixed reference geometry and potential.
\end{theorem}
\begin{proof}
By definition $\widehat G_{z,h}=K_h+\mathcal C_h^I[z_h]$, so
\eqref{eq:exact-boundary-gauge} proves
\eqref{eq:defect-corrected-complete-identity}.
The inequalities give
$\widehat G_{z,h}\succeq\gamma A_h-\Delta_h[z_h]\succ0$.
The trace argument is identical to that in
\cref{thm:compatible-boundary-gauge}.
The fixed-potential cofactor Hessians are independent of $y_h$, proving
the precomputation assertion.
\end{proof}

For $z_h=Z\iota$, this is the constant-coefficient corrected core
$\widehat G_h=G_h-\Delta_{Z,h}$ used in the finite-strain examples.
The correction is exact and preserves both $K_h$ and its quadrature.
A cellwise estimate may prove the required global inequality. Failure
of a cellwise sufficient estimate leaves room for positivity after
conforming assembly and essential constraints.

\subsection{The full potential trace and operator equivalence}

\begin{corollary}[Cellwise exactness and boundary equivalence]
\label{cor:compatible-maximality}
Fixed continuous potentials with the same full external trace have the
same exactly integrated global gauge Hessian on the same variation space.
If both are cellwise compatible, their numerical reference forms coincide.
More generally, a compatible potential and a defect-corrected potential
with that trace give the same assembled reference form.
For Tet10 candidates and variations in $[\mathbb P_2]^3$, the local
compatible class is exactly $[\mathbb P_1]^3$.
\end{corollary}
\begin{proof}
Equal traces make the boundary functionals in
\eqref{eq:potential-boundary-functional} equal for every deformation.
Their Hessians are therefore equal.
A compatible reference and a corrected reference both equal $K_h$
plus that exact Hessian. The Tet10 class follows from the
opposite-edge faithfulness and affine-potential characterization in
\cref{cor:faithful-jet,prop:exact-gauge-alignment}, applied to $z_e$ in the
first slot of \eqref{eq:common-trilinear-defect}.
\end{proof}

In particular, let $\overline X_h$ be a continuous lower-order geometry
field used as a potential, and suppose
$Z\overline X_h=Z\iota$ on the full external boundary.
If $Z\overline X_h$ is cellwise compatible, its reference form equals
the defect-corrected form from the constant physical coefficient $Z$.
Interior potential changes can alter local densities and their
positivity bounds while leaving the exactly assembled matrix unchanged.
A different boundary trace may give a different reference matrix, with
its matching compensation still required.

A prescribed potential trace may also admit an exact representative
obtained by least-norm correction.
\Cref{app:trace-representative} gives the solvability obstruction,
the complete construction and proof, and its metric-dependent
sensitivity. This changes the local representation of a fixed exact
operator, and neither a uniform repair cost nor a solver benefit follows from the representation alone.

\section{Coercive reference forms and finite-deformation regimes}
\label{sec:coercivity-regimes}

We now prove positivity of the complete reference forms.
Quadratic tensor geometry supplies the mesh- and order-uniform
stress-free result; curved tetrahedral families give finite-deformation
constructions. Every bound retains both condensed-pressure derivatives
from \eqref{eq:complete-pressure-energy}.

For a compatible potential $z_h\in\mathcal Z_h^0$, write
\[
G_{z,h}(y_h)=K_h(y_h)+\mathcal C_h^I[z_h]
           =K_h(y_h)+\mathcal C_h^Q[z_h].
\]
For admissible nodal vectors $u,v$ with interpolants $u_h,v_h$, define
the reference metric $A_{I,h}$ by
\begin{equation}
u^TA_{I,h}v
=\sum_e\mathcal Q_e\!\left[
 \det J_{X,e}\,
 A_1(I_3)[\nabla_Xu_h,\nabla_Xv_h]\right],
\qquad
A_1(I_3)[H,H]=2\|H\|_F^2+\frac23(\tr H)^2.
\label{eq:assembled-identity-metric}
\end{equation}
This is the quadrature-defined physical gradient metric. Its uniform
comparison with the continuous $H^1$ norm is proved below using explicit
shape and Poincar\'e assumptions. Removing constants on one mesh alone
does not supply a uniform bound.

\subsection{Stress-free reference constructions}

At rest, the design question can be answered by direct material
identities and geometric norm comparisons. We first obtain a bound
uniform in mesh size and displacement order on quadratic tensor
geometry. A parameter-domain estimate makes the physical norm
conversion explicit for the examples, and a separate distortion
bound treats quadratic tetrahedra.

\subsubsection{A uniform construction on quadratic tensor geometry}

Let $c_{\Sigma,e}=c_{1,e}+c_{2,e}$ and let $c_*>0$ satisfy $c_*\le c_{\Sigma,e}$ on
every cell. Write $y_{\rm ref}=\iota$ for the identity deformation.
For $H_s=\dev\sym H$, direct stress-free differentiation of
the two modified invariants gives
\begin{equation}
D^2W_{{\rm iso},e}(I_3)[H,H]
 +(2c_*I_3):(H\times H)
 =c_*A_1(I_3)[H,H]+4(c_{\Sigma,e}-c_*)\|H_s\|_F^2.
\label{eq:common-reference-margin}
\end{equation}
Indeed, each modified invariant has stress-free Hessian
$4\|H_s\|_F^2$, and
$I_3:(H\times H)=(\tr H)^2-\tr(H^2)$.
Decomposing $H$ into trace, symmetric trace-free and skew parts gives
the displayed identity. The constitutive derivations in
\Cref{app:material-calculus} give the same formula.

\begin{theorem}[Mesh- and order-uniform reference coercivity]
\label{thm:q2-uniform-coercivity}
Consider conforming tensor cells with parent cube $[-1,1]^3$, reference
maps $X_e\in\mathbb Q_2^3$, and full local variation spaces
$\mathbb Q_{p_e}^3$, $p_e\ge2$, assembled into a conforming global
space. The orders may differ between cells and between meshes.
On cell $e$, coordinate $i$ uses $n_{e,i}\ge p_e+1$ Gauss points,
with the positive tensor-product weights. The chosen global linear
candidate space $\mathcal Z_h$ must contain the physical identity
$\iota(X)=X$; continuous piecewise $\mathbb Q_{p_e}^3$ candidates
with these geometry pullbacks satisfy that requirement.

Assume that, throughout every parent cell,
\[
\|J_{X,e}\|_2\le b h_e,\qquad
\|J_{X,e}^{-1}\|_2\le (a h_e)^{-1},\qquad
\det J_{X,e}>0,
\]
where $h_e>0$ is a cell size and $0<a\le b$ are independent of
mesh size and all cell orders. The admissible conforming variation spaces
satisfy a Poincar\'e inequality with the same constant $C_P$.
Let $c_{1,e},c_{2,e}\ge0$ and choose one coefficient
\[
c_*\ge c_0>0,\qquad c_*\le\min_e(c_{1,e}+c_{2,e}),
\]
with $c_0$ independent of mesh size and all cell orders.
At the stress-free identity deformation and for arbitrary nonnegative
cell bulk moduli, the potential $z_h=2c_*\iota$ belongs to
$\mathcal Z_h^0$ and satisfies
\begin{equation}
G_{z,h}(y_{\rm ref})[u_h,u_h]
 \ge \frac{2c_0}{1+C_P^2}\left(\frac ab\right)^5
        \|u_h\|_{H^1(\Omega_h)}^2.
\label{eq:q2-uniform-coercivity}
\end{equation}
The exact complete-tangent boundary representation of
\Cref{thm:compatible-boundary-gauge} holds.
\end{theorem}

\begin{proof}
The pullback $z_e=2c_*X_e$ lies in $\mathbb Q_2^3$ and is globally
continuous. In each coordinate, every monomial in
$J_{z,e}:(J_{u,e}\times J_{v,e})$ has degree at most $2p_e+1$:
the determinant polarization differentiates once in each coordinate
across the three factors. Tensor Gauss with at least $p_e+1$ points
per coordinate is exact for these degrees. Thus the fixed potential
is compatible. Its physical coefficient is $R_z=2c_*I_3$.

Equation \eqref{eq:common-reference-margin}, the nonnegative coefficient
$c_{\Sigma,e}-c_*$ and positive quadrature weights give
$G_{z,h}(y_{\rm ref})\succeq c_*A_{I,h}$.
At $y_{\rm ref}$ every cell-volume change and recovered pressure is
zero. The complete volumetric Hessian therefore adds only a nonnegative
cell-volume outer product, for every stated bulk modulus.

It remains to compare the physical quadrature metric with the continuous
gradient norm uniformly in the cell orders. The whole-cell shape bounds imply
\[
a^3h_e^3\le\det J_{X,e}\le b^3h_e^3,\qquad
\|D_\xi u_eJ_{X,e}^{-1}\|_F
 \ge (b h_e)^{-1}\|D_\xi u_e\|_F.
\]
Consequently
\[
G_{z,h}(y_{\rm ref})[u_h,u_h]
 \ge \sum_e \frac{2c_0a^3}{b^2}\,h_e\,
       \mathcal Q_e[\|D_\xi u_e\|_F^2].
\]
The squared parent gradient has coordinate degree at most $2p_e$,
so its quadrature is exact. On the other hand, the same shape bounds
give
\[
\|\nabla_Xu_h\|_{L^2(K_e)}^2
 \le \frac{b^3}{a^2}\,h_e
       \int_{[-1,1]^3}\|D_\xi u_e\|_F^2\,d\xi.
\]
Combining the two estimates and summing over cells yields
$G_{z,h}[u_h,u_h]\ge2c_0(a/b)^5\|\nabla_Xu_h\|_{L^2}^2$.
Finally, the fixed Poincar\'e inequality gives
$\|u_h\|_{H^1}^2\le(1+C_P^2)\|\nabla_Xu_h\|_{L^2}^2$.
No polynomial-degree-dependent inverse inequality has entered the proof.
\end{proof}

The candidate-space hypothesis is essential: the argument does not
place the identity into an arbitrarily restricted potential class.
The coefficient $c_*$ is common to the mesh, even when the material
sums vary by cell. This preserves continuity of the potential.
For these quadratic tensor geometries the same rule also integrates
the reference determinant exactly, so $V_e^Q=V_e$.
The lower bound is uniform in the nonnegative bulk moduli; it does
not bound the upper spectrum or the condition number.

\subsubsection{A parameter-domain estimate for compatible margins}
\label{sec:parameter-norm-transfer}

For the parameterized examples, a quadrature-point material margin can
be converted into a continuous physical norm bound using explicit geometry
constants. This is a change-of-variables and Poincar\'e estimate
\citep{brezis2011functional}; it does not require an order-dependent
inverse inequality.

\begin{lemma}[Parameter-domain coercivity transfer]
\label{lem:parameter-norm-transfer}
Let $T=(0,1)^3$, and let $X_h:T\to\Omega_h$ be an orientation-preserving
bi-Lipschitz reference map, piecewise continuously differentiable.
Suppose the positive constants $L,M_X,d_-,d_+$ satisfy
\[
 \|D_tX_h\|_2\le L,\qquad
 \|(D_tX_h)^{-1}\|_2\le M_X,\qquad
 0<d_-\le\det D_tX_h\le d_+ .
\]
Let $U=u_h\circ X_h$ be the pullback of an admitted conforming variation, with
$U=0$ on $t_1=0$.
Let $\mathcal Q_h^t$ be the positive parameter-domain form of the
quadrature defining $G_{z,h}$, including the affine parameter-cell measures.
Assume that it
integrates $\|D_tU\|_F^2$ exactly for every such $U$.
For a compatible potential $z_h\in\mathcal Z_h^0$, suppose the material-plus-gauge
density of $G_{z,h}$ has margin at least $\gamma>0$ relative to
$A_1(I_3)$ at every quadrature point. Let all recovered cell pressures
be zero and all bulk moduli be nonnegative.
Define
\[
 C_P=\frac{2L}{\pi}\sqrt{\frac{d_+}{d_-}}.
\]
Then
\begin{equation}
 G_{z,h}[u_h,u_h]\ge
 \frac{2\gamma d_-}
 {L^2d_+M_X^2(1+C_P^2)}
 \|u_h\|_{H^1(\Omega_h)}^2 .
 \label{eq:parameter-h1-transfer}
\end{equation}
For affine tensor parameter cells with pullback degree $p_e$,
tensor Gauss counts at least $p_e+1$ satisfy the squared-gradient
exactness assumption. If $\gamma,L,M_X,d_-,d_+$ are uniform, the
bound is uniform in mesh size and these polynomial orders.
\end{lemma}

\begin{proof}
The parameter quadrature includes the affine parameter-cell measures;
changing variables expresses the physical quadrature as
$\mathcal Q_h^t[\det D_tX_h\,(\cdot)\circ X_h]$.
Since $D_tU=(\nabla_Xu_h)\circ X_h\,D_tX_h$ and
$A_1(I_3)[H,H]\ge2\|H\|_F^2$, positivity of the weights gives
\[
 G_{z,h}[u_h,u_h]\ge
 \frac{2\gamma d_-}{L^2}\mathcal Q_h^t[\|D_tU\|_F^2]
 =\frac{2\gamma d_-}{L^2}\int_T\|D_tU\|_F^2.
\]
The pressure-geometric term is zero and the complete pressure outer
term is nonnegative, so neither changes this lower estimate.

The one-dimensional Poincar\'e inequality with one endpoint fixed has
constant $2/\pi$. It follows by even reflection about $t_1=1$ and the
Dirichlet inequality on $(0,2)$.
Applying it along the first coordinate and changing variables yields
\[
 \|u_h\|_{L^2(\Omega_h)}^2
 \le \frac{4d_+}{\pi^2}\|D_tU\|_{L^2(T)}^2
 \le C_P^2\|\nabla_Xu_h\|_{L^2(\Omega_h)}^2.
\]
The inverse-map bound also gives
\[
 \|\nabla_Xu_h\|_{L^2(\Omega_h)}^2
 \le d_+M_X^2\|D_tU\|_{L^2(T)}^2.
\]
Combining these inequalities proves \eqref{eq:parameter-h1-transfer}.
For a tensor polynomial of degree $p_e$ in each coordinate,
every squared parent derivative has coordinatewise degree at most
$2p_e$. The stated Gauss rule integrates it exactly. Affine conversion
to the parameter coordinate preserves this property.
\end{proof}

\subsubsection{A distortion bound for quadratic tetrahedra}

The following specialization uses quadratic tetrahedral maps and the
symmetric four-point rule $\Qfour$, exact through total degree two.
The parent is the standard tetrahedron $\That$, and $\lambda_a$ are
its four barycentric coordinates, as specified in
\Cref{sec:setting}.
The energy is \eqref{eq:complete-pressure-energy} with
$\mathcal Q_e=\Qfour$ and the positive quadrature reference volumes
$V_e^Q=\Qfour[\det J_{X,e}]$. Its recovered pressure is
$\kappa_ec_e(y_h)/V_e^Q$. The potential class for this construction
contains continuous corner-linear fields.

On a cell, the corner interpolant is
$\overline X_e=\sum_{a=0}^3\lambda_a X_{e,a}$, where $X_{e,a}$
are its reference corner positions.
These local interpolants define the continuous corner-linear potential
$\overline X_h$ on the conforming reference mesh.
For a matrix, $\norm{\cdot}_2$ denotes its largest singular value.
Let $\overline J_{X,e}=D_\xi\overline X_e$ and define
\[
\delta_h=\max_{e,q}
\norm{\overline J_{X,e}J_{X,e}(\xi_q)^{-1}-I_3}_2,
\]
where $\xi_q$ are the four quadrature points.
Use the metric $A_{I,h}$ defined in \eqref{eq:assembled-identity-metric},
with $\mathcal Q_e=\Qfour$ on these tetrahedra.

\begin{theorem}[Stress-free compatible core]
\label{thm:compatible-rest-core}
Let $y_{\rm ref}$ be the identity deformation on $\Omega_h$, whose
parent maps are $X_e$. Consider its complete condensed-$P_0$
stored-energy Hessian.
Let the cellwise isochoric Mooney coefficients satisfy
$c_{1,e},c_{2,e}\ge0$ and choose a fixed
$0<c_*\le\min_e(c_{1,e}+c_{2,e})$.
Allow arbitrary nonnegative cell bulk moduli.
For $z_h=2c_*\overline X_h$,
\[
G_{z,h}(y_{\rm ref})\succeq c_*(1-2\delta_h)A_{I,h}.
\]
If $\delta_h<1/2$ and the essential constraints eliminate constant
displacements, then $G_{z,h}(y_{\rm ref})\succ0$.
The boundary identity and nonpositive-index conclusions of
\Cref{thm:compatible-boundary-gauge} hold without a volume-defect
correction.
\end{theorem}

\begin{proof}
Write
$H_s=(H+H^T)/2-(\tr H)I_3/3$.
The stress-free identity \eqref{eq:common-reference-margin} gives
\[
D^2W_{{\rm iso},e}(I_3)[H,H]
 +(2c_*I_3){:}(H\times H)
=c_*A_1(I_3)[H,H]
 +4(c_{1,e}+c_{2,e}-c_*)\norm{H_s}_F^2.
\]
At a quadrature point write
$E=\overline J_XJ_X^{-1}-I_3$, so that
$R_z=2c_*(I_3+E)$.
If $\sigma_1,\sigma_2,\sigma_3$ are the singular values of $H$,
then
\[
\norm{\cof H}_F^2
=\sigma_1^2\sigma_2^2+\sigma_2^2\sigma_3^2+\sigma_3^2\sigma_1^2
\le\frac13\norm H_F^4.
\]
Since $\norm E_F\le\sqrt3\norm E_2$ and
$H\times H=2\cof H$, the change of gauge is bounded by
\[
\bigl|(2c_*E){:}(H\times H)\bigr|
\le4c_*\delta_h\norm H_F^2
\le2c_*\delta_h A_1(I_3)[H,H].
\]
The pointwise estimate follows.
At the stress-free state every condensed pressure is zero.
The volumetric second variation is a sum of nonnegative
cell-volume outer products. Positive quadrature weights therefore
give the assembled estimate for every nonnegative bulk modulus.

It remains to show that $A_{I,h}$ is positive on the constrained
space. A zero value forces the physical variation gradient to
vanish at every quadrature point. Invertibility of $J_X$ then
forces the affine parent gradient to vanish at those points.
Quadratic exactness, applied to its squared norm, makes that
parent gradient identically zero. The variation is constant on
each cell, continuity makes it globally constant, and the
essential constraints remove it.
\end{proof}

The distortion condition is sufficient. Its failure leaves positivity
undecided. The theorem supplies a bulk-independent lower bound in the
displayed metric. The largest eigenvalue can grow with the bulk moduli.
The same cofactor perturbation estimate can be applied to
any strictly positive constant-gauge material chart.
For the more general choice $z_h=Z_*\overline X_h$, define
\[
\eta_h(Z_*)=\max_{e,q}
\norm{Z_*(\overline J_{X,e}J_{X,e}(\xi_q)^{-1}-I_3)}_F.
\]
Then
\[
\bigl|(R_z-Z_*){:}(H\times H)\bigr|
\le\frac2{\sqrt3}\eta_h(Z_*)\norm H_F^2
\le\frac2{\sqrt3}\norm{Z_*}_F\delta_h\norm H_F^2.
\]
This gives a coefficient-dependent geometry margin.
Exact alignment makes $\eta_h(Z_*)=0$ even when the reference
curvature is large.
The coefficient and potential remain fixed during differentiation.
Away from zero cell pressure, the pressure-geometric term also
requires its own positivity bound.

\subsection{Certified finite-deformation constructions}

The following families show how the reference construction extends
beyond the stress-free state. They use three different routes:
an exactly aligned constant coefficient, an explicitly defect-corrected
constant coefficient, and a variable compatible potential on fully
curved geometry. Each route has its own stated geometric and material
hypotheses. Their volume-preserving deformation families have zero
recovered pressure, so the complete pressure contribution is
nonnegative. The pressure-dependent limitation is considered afterward.

\subsubsection{A fixed aligned gauge for a nonuniform deformation family}

Write $e_1,e_2,e_3$ for the Cartesian unit vectors and
$E_{12}=e_1\otimes e_2$ for the matrix unit. Set
$D=\operatorname{diag}(2,2,1/4)$ and $F(s)=D+sE_{12}$ for $|s|\le1/4$.
For this and the remaining finite-deformation families, let $W_{\rm MR}$
denote \eqref{eq:mooney-energy} with $c_1=c_2=1$.
Use the fixed coefficient $Z_*=\operatorname{diag}(191/4,191/4,0)$.
The metric $A_1(D)$ is the positive first-minor form \eqref{eq:a1}
evaluated at this fixed matrix $D$.

\begin{proposition}[Uniform shear-family certificate]
\label{prop:aligned-shear-chart}
On this family,
\[
D^2W_{\rm MR}(F(s))+Z_*{:}D^2\cof
\succ\tfrac15 A_1(D).
\]
The same coefficient is exactly aligned with every reference map whose
midside deviations lie in the third coordinate direction.
\end{proposition}

\begin{proof}
The determinant of $F(s)$ is one.
The explicit invariant derivatives make its Hessian a matrix
polynomial of degree at most four in $s$.
Put $x=2(s+1/4)$, so $x\in[0,1]$, and express the gauge minus
$A_1(D)/5$ in the Bernstein basis
$\binom4j x^j(1-x)^{4-j}$, $j=0,\ldots,4$.
Its five coefficient matrices are determined by the five rational
values at $x=0,1/4,1/2,3/4,1$.
Symmetric elimination writes each coefficient matrix as
$L\Lambda L^T$, with $L$ unit lower triangular and $\Lambda$ diagonal.
Positive diagonal entries, called pivots, establish positive definiteness.
Unpivoted rational elimination gives the following minimum pivots:
\[
\begin{array}{c|c}
\text{coefficient indices}&\text{smallest pivot}\\ \hline
0,4&1562178416193/2996866827776\\
1,3&490754683601/913773077488\\
2&2606415/4805888.
\end{array}
\]
All nine pivots of every coefficient matrix are positive.
The nonnegative Bernstein weights sum to one, so their convex
combination is positive definite throughout the interval.
The rational coefficient construction and all pivots are retained in
\texttt{aligned-shear.json}, identified in \Cref{sec:certificate-files}.
Finally $Z_*e_3=0$, so
\Cref{prop:exact-gauge-alignment} proves the alignment assertion.
\end{proof}

For the certified shear family, partition the parameter cube
$[0,1]^3$, with coordinate $t=(t_1,t_2,t_3)$, into tetrahedra.
Take $X_1=t_1$, $X_2=t_2$ and a conforming piecewise quadratic
$X_3$ for which the reference map is orientation-preserving and injective.
The current map
$y=DX+\epsilon(X_2-1/2)^2e_1$, $|\epsilon|\le1/4$,
has gradient in \Cref{prop:aligned-shear-chart} and determinant one.
Every condensed pressure value is zero, and its second variation
adds only the positive cell-volume outer product.
When the essential conditions remove global translations, $Z_*$
supplies a positive translated reference for this nonuniform
complete-$P_0$ family at every $\kappa\ge0$.
Indeed, the positive quadrature weights and the pointwise certificate
force a zero-energy variation to have zero gradient at all four
quadrature points. Its parent gradient is affine; those four
affinely independent points determine it, so it vanishes on the
whole cell. Conformity and connectedness then give one global
translation, removed by the essential conditions.
The correction is exact on each admitted curved mesh; it needs no
small-curvature or asymptotically affine hypothesis.

Nonzero pressures add their complete pressure-geometric terms inside
$G_h$. Its positivity then needs another bound or a direct matrix
positivity check. The algebraic boundary identity still holds.
This family supplies a fixed, exactly compensated reference form with a
proved positive-reference regime for nonuniform curved-element states.

\subsubsection{A verified non-aligned curved-mesh family}
\label{sec:nonaligned-certificate}

For the constant coefficient $Z_*$, take the physical potential
$z_*(X)=Z_*X$ and define
\[
G_h(y)=K_h(y)+\mathcal C_h^Q[z_*],\qquad
\Delta_{Z_*,h}=\mathcal C_h^Q[z_*]-\mathcal C_h^I[z_*],\qquad
\widehat G_h(y)=G_h(y)-\Delta_{Z_*,h}.
\]
These are the constant-coefficient specializations of the reference-only
construction in \Cref{thm:reference-corrected-core}. In particular,
$\widehat G_h=K_h+\mathcal C_h^I[z_*]$ preserves the complete original
stored-energy tangent through its exact boundary compensation.

Use the explicit $48$-cell mesh of
\Cref{app:nonaligned-geometry}. Its parameter domain
is $[0,1]^3$, with coordinate $t$.
Only the first reference component is curved;
$X_2=t_2$ and $X_3=t_3$.
The first component has zero boundary perturbation.
Its smallest parent-vertex Jacobian determinant is exactly
\[
\frac{112589990684263}{9007199254740992}>0.
\]
Since the other two components are affine, the
determinant is affine on every cell. Vertex positivity
therefore gives whole-cell orientation.
Monotonicity of $X_1$ on lines parallel to the first
axis, together with the fixed boundary, gives
global injectivity.

Take the same $D=\operatorname{diag}(2,2,1/4)$ and
$Z_*=\operatorname{diag}(191/4,191/4,0)$ as before.
For a real deformation parameter $\epsilon$, define
\[
y_h^\epsilon(X_h(t))=DX_h(t)+\epsilon(t_2-1/2)^2e_1,
\qquad |\epsilon|\le1/4.
\]
This field is quadratic on each cell. Its physical
gradient is $D+sE_{12}$, with
$s=2\epsilon(X_2-1/2)$, and its determinant is one.
Define the fixed reference matrix $A_h$ by assembling
\[
\sum_e Q_4\!\left[
 \det J_{X,e}\,
 A_1(D)[\nabla_Xu_h,\nabla_Xv_h]\right].
\]
Impose zero variations on the face $t_1=0$.
This removes global translations, so $A_h\succ0$
by positivity of the point metric and the
quadratic-element gradient argument used earlier.
There are $300$ free displacement coordinates.

\begin{proposition}[Uniform corrected core on a non-aligned mesh]
\label{prop:nonaligned-uniform-core}
On this explicit mesh, for every
$|\epsilon|\le1/4$ and every $\kappa\ge0$, the
complete condensed-$P_0$ Hessian has the
reference-only boundary representation of
\Cref{thm:reference-corrected-core}, with
$\widehat G_h(y_h^\epsilon)\succ0$.
The reference curvature is not annihilated by $Z_*$.
\end{proposition}

\begin{proof}
Let $b_4=(5-\sqrt5)/20$ be the smaller barycentric
coordinate of $Q_4$. On this parameter mesh,
all quadrature values of $X_2$ lie between
$b_4/2$ and $1-b_4/2$. Hence their shear parameters satisfy
\[
|s|\le\frac{15+\sqrt5}{80}<\frac{27}{125}.
\]
The last inequality follows, for example, from
$\sqrt5<57/25$.

On $|s|\le27/125$, the same rational Bernstein
construction as in \Cref{prop:aligned-shear-chart}
gives the stronger quadrature-resolved bound
\[
D^2W_{\rm MR}(D+sE_{12})+Z_*{:}D^2\cof
\succ\frac{23}{100}A_1(D).
\]
For this interval and margin, all nine rational
elimination pivots of each of the five Bernstein
coefficient matrices exceed $47/100$.
The complete rational elimination data are in the non-aligned
certificate package identified in \Cref{sec:certificate-files}.
The physical parameter interval
$|\epsilon|\le1/4$ is unchanged; this estimate uses
the actual quadrature locations.

It remains to qualify the global geometry defect.
Set
$S_h=(23/100)A_h-\Delta_{Z_*,h}$.
The interval certificate described below encloses
this exact matrix and proves $S_h\succ0$.
At every quadrature point, $\det F=1$, so the
condensed pressure value is zero. Its complete
second variation adds the nonnegative
cell-volume outer product for every $\kappa\ge0$.
Thus $G_h(y_h^\epsilon)\succeq(23/100)A_h$,
and \Cref{thm:reference-corrected-core} applies.
Finally, the nonzero first-component reference
curvature is multiplied by $191/4$ under $Z_*$;
the mesh is non-aligned.
\end{proof}

\subsubsection{Interval positive-definiteness certificate}

Verified numerical methods combine a numerical
approximation with rigorously enclosed residual
or matrix quantities \citep{rump2010verification}.
Here all reference coordinates are the exact
dyadic values specified in \Cref{app:nonaligned-geometry}.
Interval evaluation at $60$ decimal digits encloses
the ideal four-point metric and the centered
cofactor defect. It uses rational operations and
an enclosed $\sqrt5$. The resulting binary64
bounds $S^-\le S_h\le S^+$ hold entrywise and are rounded outward.

Let $L$ be a numerical lower Cholesky factor of the
interval midpoint, and let $Y$ be its computed
lower-triangular inverse, retained in binary64.
Every diagonal entry of $Y$ is nonzero.
Its entries are treated as exact real numbers;
triangularity and nonzero diagonal make it invertible.
Outward-rounded products enclose
$C=YS_hY^T$. If $C^-\le C\le C^+$ entrywise, define
\[
\ell_i=C^-_{ii}
 -\sum_{j\ne i}\max\{|C^-_{ij}|,|C^+_{ij}|\}.
\]
For every vector $v$, symmetry and
$2|v_iv_j|\le v_i^2+v_j^2$ give
$v^TCv\ge\sum_i\ell_i v_i^2$.
All $300$ certified lower bounds satisfy
$\ell_i>0.9999999997$.
Therefore $C\succ0$, and invertibility of $Y$
proves $S_h\succ0$.
The interval bounds and $Y$ are retained in
\texttt{retained\_matrix\_certificate.npz} in the non-aligned package
of \Cref{sec:certificate-files}; its verifier reconstructs every row bound.

Each binary64 interval matrix product uses explicit
multiply, outward rounding to adjacent representable
numbers, addition, and outward rounding.
An induction over the summed products gives the
entrywise inclusion. Exact rational tests check
$480$ independently generated product entries.
A separate conventional assembly differs from
the enclosed matrix midpoint by
$1.269\times10^{-15}$ relatively in Frobenius norm.
Its approximate smallest eigenvalue is $0.002738$;
the sign proof uses the enclosure and congruence above.

\subsubsection{A certified fully curved finite-deformation family}

Use the rational $48$-cell reference mesh specified in
\Cref{app:compatible-geometry}, and let $t\in[0,1]^3$ denote
its construction coordinate. All reference corner positions are
unchanged, so $\overline X_h(X_h(t))=t$.
Take $c_1=c_2=1$ and specify the potential by
$z_h(X_h(t))=4t$. The physical deformation family is
\[
y_h^s(X)=(I_3+sE_{12})X,\qquad X\in\Omega_h,\quad |s|\le1/4.
\]
Here $E_{12}=e_1\otimes e_2$ is the matrix unit.
Every cell has curvature spanning all three physical directions.
Impose zero variations on the face $t_1=0$; this eliminates
constant displacements.

\begin{proposition}[Uniform compatible core beyond the distortion bound]
\label{prop:compatible-uniform-family}
For this reference mesh, all $|s|\le1/4$, and arbitrary nonnegative
cell bulk moduli,
\[
G_{z,h}(y_h^s)\succeq\frac1{16}A_{I,h}\succ0.
\]
The complete original four-point stored-energy Hessian retains the exact boundary
representation of \Cref{thm:compatible-boundary-gauge}.
The geometry has $\delta_h>1/2$.
\end{proposition}

\begin{proof}
The exact integer geometry check in
\Cref{app:compatible-geometry} gives a global reference-gradient
perturbation bound whose square is $18865/24649<1$.
The reference map is consequently injective and orientation
preserving. The linear map $I_3+sE_{12}$ has determinant one,
so the entire deformation family is admissible and its condensed
pressures vanish.

At a fixed quadrature point, the isochoric Hessian at
$I_3+sE_{12}$ is a matrix polynomial of degree at most four.
The compatible gauge coefficient is independent of $s$.
Subtract $(1/16)A_1(I_3)$ and express the resulting polynomial
in the degree-four Bernstein basis on $[-1/4,1/4]$.
Exact rational arithmetic supplies the material coefficients.
Outward interval arithmetic at $60$ decimal digits encloses
$J_X^{-1}$ and the compatible coefficient at all $192$
quadrature locations. It also encloses the binary64 realization
used in the operator observations.

Each of the five Bernstein matrices at each point has an
interval $LDL^T$ factorization with all nine pivots strictly
positive. There are $8640$ such pivots; every reported lower
endpoint exceeds $0.1265$.
The nonnegative Bernstein weights sum to one, proving the
pointwise lower bound throughout the interval.
The retained coefficient and elimination data are checked by the
authors' verification script \texttt{verify.py}, described in
\Cref{sec:certificate-files}.
The nonnegative condensed-pressure outer products and
\Cref{thm:compatible-rest-core}'s gradient argument give
the assembled conclusion.

Finally, set $E=\overline J_{X,e}J_{X,e}^{-1}-I_3$
at zero-based cell $34$ and local quadrature point $0$ in the compatible
package's stored connectivity and rule ordering. For $v=(-1,1,0)^T$,
the file \texttt{distortion\_witness.json} records an outward enclosure
of $Ev$ and the exact value $\norm v_2^2=2$. Together these give
$\norm{Ev}_2/\norm v_2>0.598$.
Thus $\delta_h>1/2$, even though the direct interval certificate
proves positivity.
\end{proof}

\subsection{Pressure dependence}

The stress-free and finite-shear constructions above have zero recovered
cell pressure. Their complete volumetric Hessian therefore adds a
nonnegative outer product, and the lower bounds hold at every stated
nonnegative bulk modulus. The upper spectrum can still grow with bulk;
no bulk-independent condition number follows.

At a loaded state, the additional term \(p_eD^2c_e\) is signed.
The zero-pressure proofs do not control it.
\Cref{sec:compatible-loaded-evidence} exhibits this limitation with
the complete pressure terms retained.
The boundary identity remains exact; positivity then requires an
additional state-dependent estimate.
The proof objects and verification instructions for all finite
families are specified in \cref{sec:certificate-files}.

\section{Evidence for compatible reference design}
\label{sec:evidence}

The comparisons follow compatible spaces, attainable margins and
complete reference forms. Tensor calculations use an independent
polynomial implementation in binary64, with ideal-rule algebraic
checks distinguished below. Explicit potentials and analytic bounds
stand in place of an uncomputed numerical optimum.
Detailed tetrahedral operator and arithmetic controls are retained
in \cref{app:supporting-observations}.

\subsection{Exact moment checks of the compatible spaces}

An independent monomial implementation forms the scalar defect from
\[
\mathcal L\!\left[
 \det(D_\xi\xi^\alpha,D_\xi\xi^\beta,D_\xi\xi^\gamma)\right]
=\det(\alpha,\beta,\gamma)\,
 \mathcal L[\xi^{\alpha+\beta+\gamma-(1,1,1)}].
\]
Here $\alpha,\beta,\gamma$ are nonnegative integer multi-indices,
$\xi^\alpha=\prod_{j=1}^3\xi_j^{\alpha_j}$, and
$\det(\alpha,\beta,\gamma)$ denotes the determinant of their exponent
rows. If the coefficient is zero, the term is zero; a nonzero term
has nonnegative resulting exponents.

Gauss moments are obtained exactly by taking the polynomial remainder
modulo the Legendre node polynomial and integrating that remainder.
Nonzero pivots modulo the prime $1000000007$, with admissible
denominators, certify rational rank lower bounds. The explicit
compatible spaces give matching upper bounds.

\begin{table}[H]
\centering
\caption{Exact-algebra checks for the scalar potential kernel with full
tensor-product variations. The rank bounds coincide. The theorem
supplies the all-order statement; these finite checks use no
floating-point rank tolerance.}
\label{tab:tensor-kernel-checks}
\begin{tabular}{@{}rrrrr@{}}
\toprule
Degree $p$ & Gauss count $n$ & Candidate dimension & Defect rank & Kernel dimension\\
\midrule
3 & 4 & 64 & 37 & 27\\
4 & 5 & 125 & 98 & 27\\
5 & 6 & 216 & 189 & 27\\
5 & 7 & 216 & 91 & 125\\
6 & 8 & 343 & 218 & 125\\
\bottomrule
\end{tabular}
\end{table}

The $\mathcal S_3$ scalar candidate space has dimension $32$.
The proposed $23$-dimensional kernel has zero modular rank, and its
nine-dimensional complement has nine nonzero modular pivots.
The exact parity proof in \cref{thm:serendipity-compatible-kernel}
establishes the kernel itself.
Anisotropic checks at degree/rule pairs
$(3,1,1)/(4,2,2)$, $(4,1,1)/(5,2,2)$,
$(3,2,1)/(4,3,2)$, and $(5,3,2)/(7,4,3)$
give scalar kernel dimensions $13,13,18,45$.
These include the exceptional single-higher-order-coordinate class.

\subsection{Attainable margins under fixed-candidate enrichment}
\label{sec:enrichment-observations}

Use the checkerboard construction
\eqref{eq:checkerboard-geometry} on a $2\times2\times2$ patch,
with $c_1=c_2=1/2$ and amplitudes $\alpha=1/64,1/32$.
The geometry, Gauss4 point set, and continuous $\mathcal S_3^3$
candidate-potential class are fixed in each local variation-space
comparison. The two explicit potentials are $2\iota$, admitted for
$\mathcal S_3$ variations, and the affine-pullback field
$z(X_h(t))=2t$, admitted after enrichment to $\mathbb Q_3$.

Their classes are determined by the exact kernel theorems.
Tensor nodal arrays evaluate the two explicit polynomial fields and
their physical gradients. The comparison does not assemble or solve
a separate global serendipity equilibrium problem.
The computed material margins and analytic bounds on the
$\mathbb Q_3$ optimum are reported in \cref{tab:enrichment-margins}.

\begin{table}[H]
\centering\small
\caption{Material margins on the checkerboard patch.
The ideal $\mathcal S_3$ optimum is one.
The explicit $\mathbb Q_3$ construction is a feasible candidate,
with its margin evaluated in binary64. The final two columns
bound the optimum over every compatible potential in the same
continuous $\mathcal S_3$ candidate space. The bounds are analytic;
the SDP optimum itself is not computed.}
\label{tab:enrichment-margins}
\begin{tabular}{@{}lrrrr@{}}
\toprule
$\alpha$ & $\mathcal S_3$ identity & $\mathbb Q_3$ explicit
& Optimum lower bound & Optimum upper bound\\
\midrule
$1/64$ & $1$ to rounding & $0.9707486774$ & $0.8906908266$ & $0.9965521842$\\
$1/32$ & $1$ to rounding & $0.9399073749$ & $0.7687423627$ & $0.9934281568$\\
\bottomrule
\end{tabular}
\end{table}

The explicit construction loses approximately $2.93\%$ and $6.01\%$
of the unit material margin. The analytic theorem requires losses
of at least $0.3448\%$ and $0.6572\%$ even at the optimum.
These are distinct statements about one chosen potential and the
entire compatible class. The upper bounds on the possible optimum
loss are conservative, and no sharp constant is inferred from two
amplitudes.

\begin{figure}[!htbp]
\centering
\includegraphics[width=\linewidth]{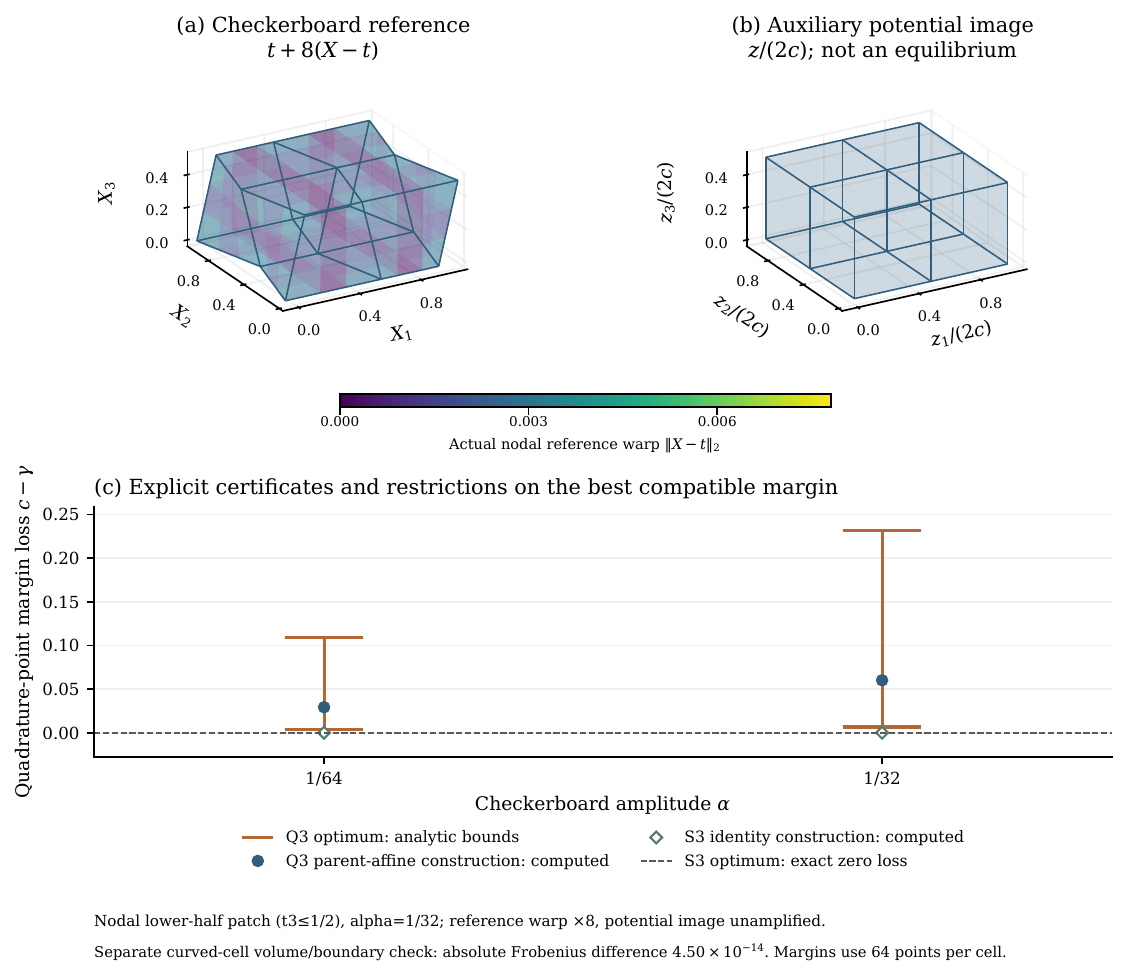}
\caption{Checkerboard geometry and compatible material margins.
At $\alpha=1/32$, panel (a) displays the reference warp magnified as
$t+8(X_h(t)-t)$, with only the lower parameter-half cells shown.
Colours give the actual nodal warp. Panel (b) shows the affine-pullback
potential image $z/(2c)$ without geometric magnification; it is an
auxiliary field. Segments and facets join the stored nodes.
Panel (c) uses both amplitudes: intervals bound the optimal
$\mathbb Q_3$-compatible material-margin loss, and points show
the explicit constructions. The candidate space is the same
continuous $\mathcal S_3$ space throughout.
Here $\gamma=\gamma_h(z)$ and $c=1$. These are finite-point material certificates at the stated quadrature
locations; global spectral and solver-time comparisons belong to a different study.
The separate curved-cell check noted in the graphic refers to the
smooth-reference comparison in \cref{sec:quadratic-reference-observations}.}
\label{fig:enrichment-margin}
\end{figure}

An independent oriented-face action check uses the $\mathbb Q_3$
variation space with zero variations on $t_1=0$. The prescribed test
field in construction coordinates is
\[
v(X_h(t))=
\bigl(t_1,\ t_1t_2(1-t_2),\
             t_1t_3(1-t_3)\bigr)^T.
\]
For the affine-pullback compatible potential, volume and boundary
actions differ by at most $6.78\times10^{-15}$ in Euclidean norm.
The paired interior-face contribution has norm at most
$6.25\times10^{-18}$.
The action comparison tests assembled cancellation for this
specified variation. The loss theorem itself applies to all local
matrix increments.

The amplitude and refinement comparison also includes \(\alpha=1/16\)
and uses \(N=1,2,4\) cells per coordinate. The two \(N=2\) rows
in \cref{tab:enrichment-margins} are reused.
At amplitudes \(1/64,1/32,1/16\), the respective explicit
\(\mathbb Q_3\) margins are approximately
\[
0.97074867735020,\qquad 0.93990737493876,\qquad 0.87290647661997.
\]
Their spreads across \(N\) are
$6.33\times10^{-15}$, $4.55\times10^{-15}$, and $1.02\times10^{-14}$.
At the largest amplitude, the explicit construction loses about
$12.7094\%$ of the unit material margin.
The analytic optimum lies between $0.477012272280111$ and
$0.987984685589522$, so every compatible choice loses at least
$1.20153\%$. No numerical optimum or near-optimality claim is inferred.

The oriented action differences in this comparison are at most
$7.07016\times10^{-15}$, with interior sums at most
$3.14747\times10^{-17}$.
For fixed $N$, these gauge actions are independent of $\alpha$:
both the parent-affine potential $z(X_h(t))=2t$ and the tested
parameter variations are fixed in the pulled-back gauge.
This does not compare complete physical tangents on unchanged domains.
For the explicit candidate, the proof of
\cref{thm:enrichment-certificate-loss} gives the material lower bound
$\gamma=(1-3\delta)/(1-\delta)$.
Thus \cref{lem:parameter-norm-transfer} applies with
$L=1+\delta$, $M_X=(1-\delta)^{-1}$ and
$d_\pm=1\pm3|\alpha|$, where $\delta=\sqrt{11}|\alpha|$.
At $\alpha=1/16$, its continuous-$H^1$ lower coefficient is approximately
$0.151028607$ for the ideal map.
The local margin, this analytic norm bound and an assembled eigenvalue
measure different things.

\FloatBarrier
\subsection{Assembled reference forms and exact compensation}

These comparisons test how a compatible material construction enters
the complete reference form. A smooth quadratic geometry supplies the
order and refinement observations. The subsequent finite-deformation
families examine the three certified constructions and the effect of
nonzero recovered pressure.

\subsubsection{Quadratic tensor geometry with a free boundary}
\label{sec:quadratic-reference-observations}

For $t=(t_1,t_2,t_3)\in[0,1]^3$, take one curved reference cell
\[
X(t)=
\left(t_1,\ t_2,\ t_3+\frac1{32}
           4t_1(1-t_1)\,4t_2(1-t_2)\right)^T.
\]
It has $\det(D_tX)=1$ in these construction coordinates and is globally one-to-one,
since the first two coordinates are unchanged and the third is a
unit-slope translation along each vertical line.
The geometry is in $\mathbb Q_2^3$. Use $\mathbb Q_3^3$
displacements with variations fixed on $t_1=0$, leaving $144$
free displacement coordinates. Four Gauss points per coordinate
define the complete stored energy, with
$c_1=c_2=1/2$ and bulk modulus $10$.
The mechanical state is the identity deformation $y_h(X)=X$,
and the potential is $z_h=2\iota$.

The smallest quadrature-point material margin is
$0.999999999999997$, consistent with the ideal value $c=1$.
Independent volume and oriented-boundary assemblies of its gauge
Hessian differ by $4.4954\times10^{-14}$ in Frobenius norm,
against a volume-gauge norm $10.1421$.
The complete-tangent reconstruction discrepancy is
$4.4951\times10^{-14}$ in the same absolute norm.
The external boundary correction is nonzero; it has norm $10.1421$.
This checks a genuine compensated form with a free boundary.

In the assembled metric $A_{I,h}$ of
\eqref{eq:assembled-identity-metric}, the smallest reference eigenvalue
is $0.9999999999999944$. Its generalized eigenvector is normalized in
that metric, and the absolute Euclidean eigenpair residual is
$5.165\times10^{-15}$.
The reference includes the full positive pressure outer product at rest;
the original complete tangent and its boundary compensation remain
separate matrices. These are numerical observations of the
quadratic-geometry construction in \cref{thm:q2-uniform-coercivity}.

\subsubsection{Displacement order, refinement and the complete pressure term}
\label{sec:tensor-refinement-observations}

Retain the same global quadratic map, material coefficients, potential
$z_h=2\iota$, bulk modulus $10$, and clamp on $t_1=0$.
On one cell, compare displacement orders $p=2,3,4$.
For $p=3$, also subdivide each parameter coordinate into $N=2,4$ cells.
The reference map is restricted from the same smooth function; its
geometry order stays quadratic.
Each case uses $p+1$ Gauss points per coordinate for the complete
material and pressure energy, and $p+2$ for the independent surface
calculation. The preceding one-cell $p=3$ matrices are reused.

For action comparisons, define the two parameter-pullback fields
\[
 v_A(X_h(t))=(t_1,0,0)^T,\qquad
 v_T(X_h(t))=
 \bigl(0,t_1t_2(1-t_2),t_1t_3(1-t_3)\bigr)^T .
\]
We evaluate $v_A,v_T$ and their sum. These are two generating fields,
not three independent test directions.
The refined cases use full matrix-free actions on $882$ and $6084$
free coordinates. No dense global eigensolve is performed there.

\begin{table}[H]
\centering\small
\caption{Smooth-reference compensation across displacement order and
refinement. The matrix rows use an absolute Frobenius difference;
the action rows use the Frobenius norm of the three-column action
difference for $v_A,v_T,v_A+v_T$.
The last column concerns the small $G_{z,h}$ versus $A_{I,h}$ generalized
eigenproblem only.}
\label{tab:tensor-refinement}
\begin{tabular}{@{}rrrrlr@{}}
\toprule
$N$ & $p$ & Free coordinates & Gauge difference & Quantity & Minimum eigenvalue\\
\midrule
1&2&54&$1.91511\,10^{-14}$&Matrix&$1$ to rounding\\
1&3&144&$4.49538\,10^{-14}$&Matrix, reused&$1$ to rounding\\
1&4&300&$1.17030\,10^{-13}$&Matrix&$1$ to rounding\\
2&3&882&$9.89552\,10^{-15}$&Actions&---\\
4&3&6084&$1.00404\,10^{-14}$&Actions&---\\
\bottomrule
\end{tabular}
\end{table}

The $p=2,4$ minimum eigenvalues are
$0.9999999999999961$ and $0.9999999999999879$, with absolute
eigenpair residuals $4.16926\times10^{-15}$ and
$1.50123\times10^{-14}$.
Their eigenvectors have squared $A_{I,h}$ norms within
$7\times10^{-16}$ of one.
The refined oriented interior-face action sums have norms
$1.66163\times10^{-17}$ and $4.92725\times10^{-17}$.
The external correction remains nonzero in every case.
These matrix and action norms use nodal coefficient coordinates;
their variation with $N$ or $p$ is not physical decay or a time comparison.

The local material margin remains one to rounding, but selected
complete-reference Rayleigh values change under refinement.
At this stress-free state the translation identity gives
\[
 G_{z,h}=A_{I,h}+Q_{P0,h},\qquad
 Q_{P0,h}[v,v]
 =\sum_e\frac{\kappa}{V_e^Q}\bigl(Dc_e[v]\bigr)^2.
\]
Here $Q_{P0,h}$ is the sum of the terms in \eqref{eq:qp0-import}.
The pressure values are zero, while their linearized outer contribution
remains present.

For $v_T$, the physical divergence is
$2t_1(1-t_2-t_3)$ because the reference map has determinant one
and its third-coordinate shear does not alter this trace.
Thus $Dc_e[v_T]/V_e$ is its cell average.
On the uniform parameter partition, summing the squared midpoint
averages gives
\begin{equation}
 Q_{P0,h}[v_T,v_T]
 =\frac{2\kappa}{3}
   \left(\frac13-\frac1{12N^2}\right)
   \left(1-\frac1{N^2}\right).
 \label{eq:refined-pressure-energy}
\end{equation}
Indeed, the mean of the squared first-coordinate midpoints is
$1/3-1/(12N^2)$, and the variance of each of the other midpoint
coordinates is $(1-1/N^2)/12$.
The quadrature integrates this cell-volume derivative exactly.
For $\kappa=10$ and $N=1,2,4$, \eqref{eq:refined-pressure-energy}
gives $0,1.5625,2.05078125$.
The observed $p=3$ transverse $A_{I,h}$ value stays
$0.727195767196$, while $G_{z,h}/A_{I,h}$ on that direction changes
$1\to3.148664872\to3.820122644$.
This is a pressure-partition effect with unchanged material margin.

The one-cell order comparison has a separate integration effect.
For $p=2$, the transverse $A_{I,h}$ value is $0.727175925926$,
differing from the Gauss-four/five value by $1.98413\times10^{-5}$.
To identify the term, temporarily write the reference amplitude as
$\alpha$ in $X_h(t)=t+16\alpha t_1(1-t_1)t_2(1-t_2)e_3$.
The only term in the $A_1(I_3)$ density for $v_T$ beyond Gauss-three
exactness is
\[
 512\alpha^2t_1^4(1-t_1)^2(1-2t_2)^2(1-2t_3)^2.
\]
Write $\mathcal G_3$ for three-point Gauss integration on $(0,1)$.
Its first missed moment is
$\int_0^1t^6\,dt-\mathcal G_3[t^6]=1/2800$.
Integrating the transverse factors gives the difference
$32\alpha^2/1575$, which equals $1/50400$ at $\alpha=1/32$.
All other terms are integrated exactly by Gauss three, and Gauss four
integrates the displayed term as well.
Exact integration of the added gauge therefore coexists with a resolved
integration difference in the original material metric.

For the ideal smooth reference map, put $s=4\sqrt2/32$.
The shear has determinant one and satisfies
$\|D_tX_h\|_2,\|(D_tX_h)^{-1}\|_2\le1+s$.
Using $\gamma=1$, $L=M_X=1+s$, and $d_-=d_+=1$ in
\cref{lem:parameter-norm-transfer} gives a continuous-$H^1$ lower
coefficient approximately $0.668011871$, uniform in the selected $h,p$.
Higher-rule $p+3$ evaluations of physical $L^2/H^1$ norms are retained
as numerical observations. They are separate from that analytic
ideal-map certificate. The largest sampled represented-reference
determinant deviation from one is $1.51\times10^{-14}$;
it is not a whole-cell interval enclosure.

\Cref{fig:journal-gauge} places two different quantities side by side.
The checkerboard material margins belong to
\cref{sec:enrichment-observations}; the smooth-reference actions use
the field and pressure partition just defined.
Their distinct behavior illustrates the distinction between a local
material certificate and a complete reference action.

\begin{figure}[H]
\centering
\includegraphics[width=\linewidth]{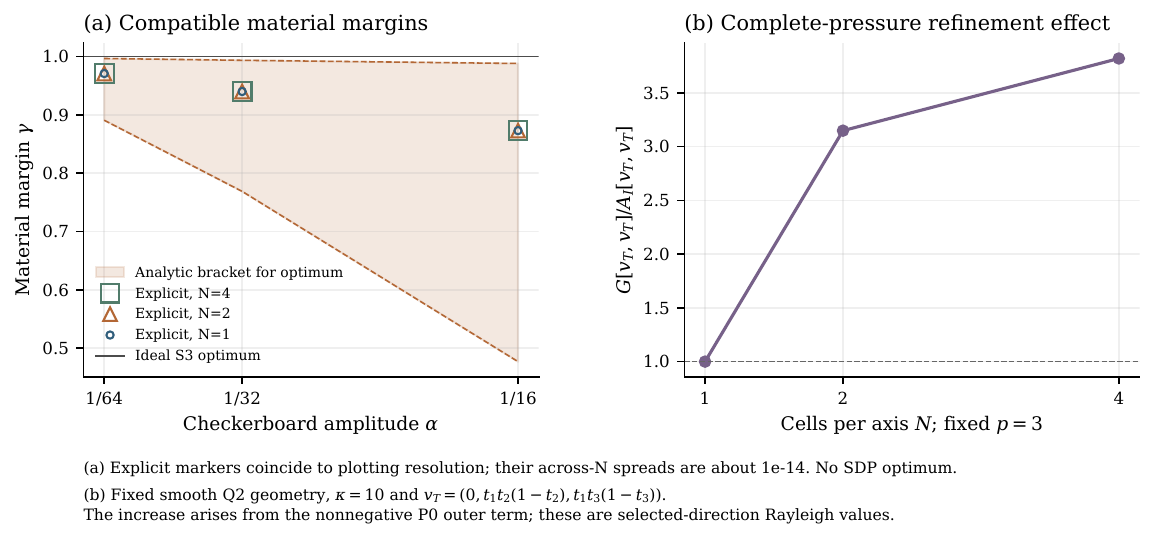}
\caption{Material certificates and complete-reference pressure contributions.
Panel (a) uses the fixed $\mathcal S_3^3$ candidate class:
markers are explicit $\mathbb Q_3$-compatible constructions on $N^3$
cells, and the shaded interval bounds the optimum.
Their overlap illustrates the predicted cell-size-independent local margin.
Panel (b) concerns the separate smooth quadratic map and the field $v_T$
defined in \cref{sec:tensor-refinement-observations}.
Its Rayleigh increase arises from the positive pressure outer contribution
\eqref{eq:refined-pressure-energy}, divided by $A_{I,h}[v_T,v_T]$.
It is not a minimum-eigenvalue curve.}
\label{fig:journal-gauge}
\end{figure}

\FloatBarrier
\subsection{Finite-deformation references and complete equations}
\label{sec:finite-reference-observations}

The tetrahedral comparisons use \(c_1=c_2=1\), hence \(c=2\),
whereas the tensor comparisons use \(c_1=c_2=1/2\).
The reference metrics are those of their respective finite-family
certificates. All complete pressure contributions are retained,
including the small geometric terms caused by rounding in nominally
volume-preserving states.

\subsubsection{Aligned reference and boundary reconstruction}

Use \(D=\operatorname{diag}(2,2,1/4)\) and
\(Z_*=\operatorname{diag}(191/4,191/4,0)\) from
\cref{prop:aligned-shear-chart}.
A parameter cube is split into \(n^3\) cubes and then Freudenthal
tetrahedra; the quadratic reference interpolates
\[
 X(t)=t+\beta\prod_{i=1}^3t_i(1-t_i)e_3,
 \qquad y_h(X)=DX+\epsilon(X_2-1/2)^2e_1.
\]
The \(24\) cases at \(n=2\) use
\(\beta\in\{0,8\}\), \(\epsilon\in\{-1/4,0,1/4\}\),
\(\kappa\in\{10,1000\}\), and either a full boundary clamp or
zero variations on \(t_1=0\).
Four further cases use \(n=4,\beta=8,\epsilon=1/4\), both bulk
moduli and both constraints. Reference orientation and injectivity
follow from the vertex Jacobians and one-component monotonicity,
as detailed in \cref{app:supporting-observations}.

Independent assemblies give the complete \(K_h\), translated \(G_h\)
and oriented-boundary matrix \(B_\Gamma\).
For interior coordinates \(i\) and free trace coordinates \(\Gamma\),
the reference and complete Schur complements are
\[
 S_G=G_{\Gamma\Gamma}-G_{\Gamma i}G_{ii}^{-1}G_{i\Gamma},
 \qquad S_K=S_G-B_\Gamma.
\]
The complete solve uses \(S_K\); when \(S_G\succ0\), the trace pencil
\(S_Kx_\Gamma=\lambda S_Gx_\Gamma\) gives its negative counts.
In \cref{tab:aligned-complete}, reconstruction is
\(\|K_h-G_h+R_\Gamma^TB_\Gamma R_\Gamma\|_F/\|K_h\|_F\),
and solve residual is \(\|K_hx-b\|_2/\|b\|_2\).

\begin{table}[H]
\centering
\caption{Complete-pressure observations for the aligned family.
Rows aggregate the stated cases. Negative counts are numerical,
with the trace pencil used for the partially constrained cases.}
\label{tab:aligned-complete}
\begin{tabular}{@{}llrrrr@{}}
\toprule
$n$ & Constraint & Free / trace & Negative count
& Reconstruction & Solve residual\\
\midrule
2 & Full & $81/0$ & 0 & $\le2.78\,10^{-17}$ & $\le1.31\,10^{-12}$\\
2 & One face & $300/219$ & 23--28 & $\le2.78\,10^{-17}$ & $\le1.31\,10^{-12}$\\
4 & Full & $1029/0$ & 0 & $\le3.32\,10^{-17}$ & $\le5.17\,10^{-12}$\\
4 & One face & $1944/915$ & 80, 109 & $\le3.32\,10^{-17}$ & $\le5.17\,10^{-12}$\\
\bottomrule
\end{tabular}
\end{table}

The clipping reference replaces negative eigenvalues of each
pointwise material Hessian by zero, assembles the resulting Hessians,
and adds the positive cell-volume outer product.
Both references use exact global Cholesky actions in the
minimum-residual (MINRES) method for the same complete equation.
The full-clamp gauge takes one column because
\(G_h=K_h\); the clipping reference takes \(17\).
For the face clamp, the gauge takes \(185\)--\(240\) columns and
reaches relative residual \(10^{-9}\) in nine of twelve cases;
clipping takes \(194\)--\(240\) and reaches it in eleven.
There is no consistent partial-boundary advantage.
The signed boundary Schur solve retains the complete equation.
\Cref{app:supporting-observations} gives the full protocols, the
failed coefficient \(\operatorname{diag}(45,45,0)\), and the eight
nonzero-pressure dilations, whose pointwise and assembled signs
need not coincide.

\subsubsection{Defect-corrected and variable compatible references}

On the certified non-aligned mesh of
\cref{prop:nonaligned-uniform-core}, \(18\) cases use
\(\epsilon=-1/4,0,1/4\), \(\kappa=0,10,1000\), and the full or
one-face clamp. The corrected reference is positive in every
observation, with relative reconstruction at most
\(4.041\times10^{-17}\) and complete solve residual below
\(5.380\times10^{-13}\).
At the face clamp, the complete matrix has \(31,25,20\) negative
eigenvalues at the three bulk moduli. Its sign differs from that of
the reference, as the boundary representation permits.

For the fully curved comparison, use the integer field \(a_h\) of
\cref{app:compatible-geometry} and
\[
 X_h^\alpha(t)=t+(\alpha/471)a_h(t),\qquad
 \alpha\in\{0,1/8,1/4,1/2,7/8\}.
\]
The potential \(z_h=4\overline X_h\) is used at rest and on the
simple-shear family. The stretched-state comparison at
\(D+(1/8)E_{12}\) uses \(z_h=Z_*\overline X_h\).
\Cref{fig:compatible-potential-geometry} distinguishes the observed
pointwise margin from the sufficient distortion bound.

\begin{figure}[H]
\centering
\includegraphics[width=.92\linewidth]{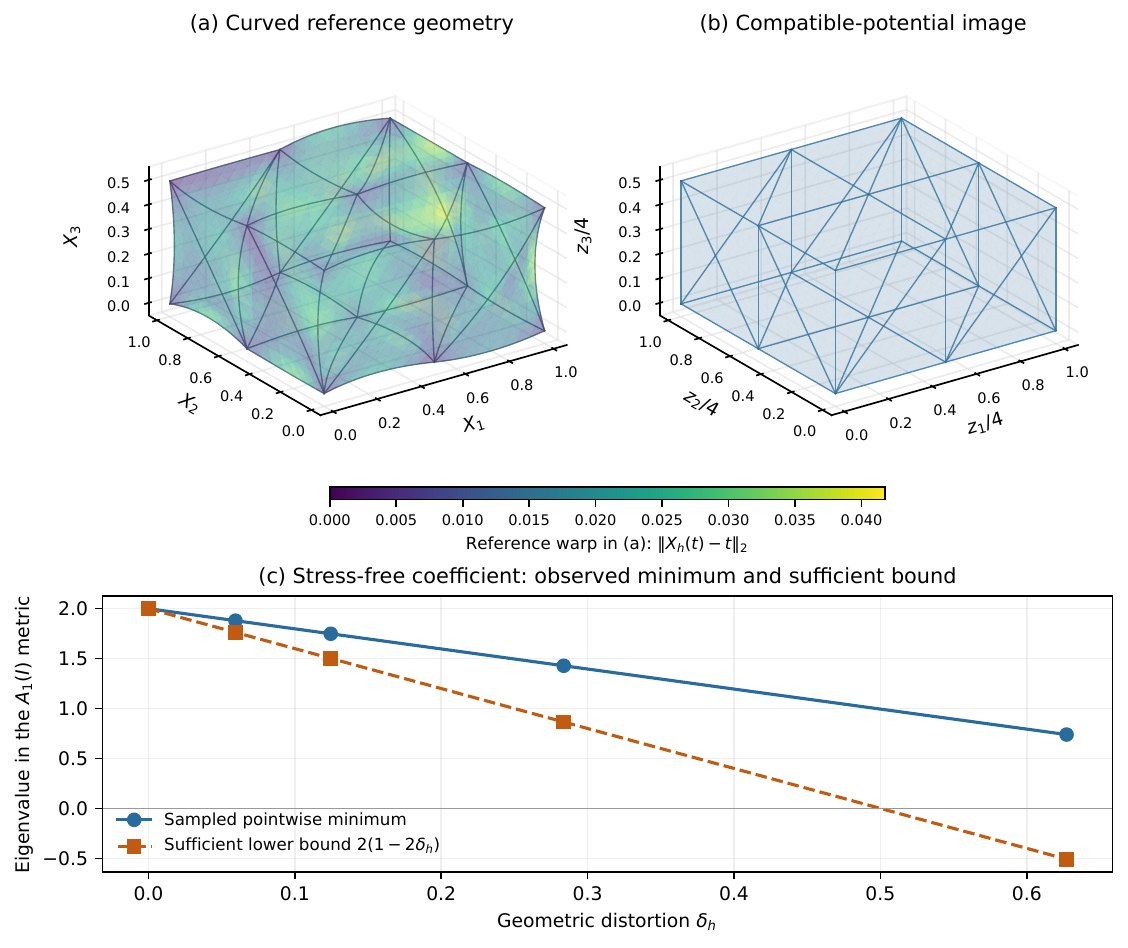}
\caption{Compatible geometry and the stress-free bound.
(a) The $\alpha=7/8$ curved reference at actual scale, with the upper
parameter-half cells omitted; colours give $\|X_h(t)-t\|_2$.
(b) The image of $z_h=4\overline X_h$, divided by four. Its straight
edges describe the coefficient potential on the same connectivity.
(c) The smallest pointwise gauge eigenvalues sampled at the four quadrature
points of every cell, measured in the $A_1(I_3)$ metric, and the sufficient
bound $2(1-2\delta_h)$. A negative bound is inconclusive when the sampled
minimum is positive. The interval-certified family retains its separate
proof scope. These dimensionless views contain no equilibrium deformation
or curvature magnification.}
\label{fig:compatible-potential-geometry}
\end{figure}

The complete \(90\)-case comparison uses physical maps
\(y_h(X)=FX\) with deformation gradients
\(F=I_3\), \(I_3+(1/4)E_{12}\) and \(D+(1/8)E_{12}\),
bulk values \(0,10,1000\), and either the face clamp or a fixed
vertex at \(t=(0,0,0)\).
There are \(88\) positive compatible references. The two failures
occur at \(\alpha=7/8\), zero bulk, and the stretched state with
\(z_h=Z_*\overline X_h\), outside the certified family of
\cref{prop:compatible-uniform-family}.
Their minimum eigenvalues are \(-0.0526\) and \(-0.0533\).
The observed bulk values \(10,1000\) restore positivity through
the volume outer term; the defect-corrected constant references
remain positive in these observations.

For matrix comparisons define
\(d(A;B)=\|A-B\|_F/\max\{1,\|B\|_F\}\).
An additional \(18\)-case interior-curvature control gives matching
potential traces and compatible/corrected operators agreeing to
\(1.31\times10^{-15}\) in \(d\), with the corrected reference as \(B\).
Curved boundary traces can give different gauges: the strongest case
differs by \(0.0775\), normalized by the compatible gauge.
These observations test the trace equivalence and its restrictions.
The complete parameter and arithmetic records remain in
\cref{app:supporting-observations}.

\subsubsection{A nonuniform loaded-state restriction}
\label{sec:compatible-loaded-evidence}

A different $48$-cell family uses the global quadratic reference map
\[
X_h(t)=t+\alpha
\bigl(t_2(1-t_2),t_3(1-t_3),t_1(1-t_1)\bigr)^T,
\qquad \alpha\in\{0,1/4,1/2\}.
\]
The two physical deformations are specified in construction coordinates by
$y_h(X_h(t))=X_h(t)$ and
$y_h(X_h(t))=DX_h(t)+\tfrac14(t_2-1/2)^2e_1$, with the same $D$ as above.
The potential is $4\overline X_h$ at rest and
$Z_*\overline X_h$ in the latter state, where
$Z_*=\operatorname{diag}(191/4,191/4,0)$.
With $c_1=c_2=1$, bulk values $0,10,1000$, and either the full boundary
or the face $t_1=0$ constrained, this gives $36$ complete operators.
Reference and normalized-current gradient bounds certify their
orientation and injectivity. All pressure-geometric contributions are retained.

For $\alpha=1/2$ with one face constrained and the nonuniform current
state, the compatible reference's smallest eigenvalue is $-0.0099687$ at
zero bulk, while the defect-corrected constant reference has minimum
$0.0093759$. Both are positive at bulk $10$. At bulk $1000$, their
minima are $-1.145760$ and $-0.800219$.
For the Euclidean-unit minimum eigenvector of the compatible reference, its
isochoric-plus-gauge term contributes $1.174943$, the positive cell-volume
outer product $0.0048728$, and the pressure-geometric term $-2.325576$.
The sum gives the observed negative eigenvalue, with eigenpair residual
$7.99\times10^{-13}$. Deleting the geometric term would change the energy
Hessian. This example identifies the extra pressure-dependent
control needed to extend the stress-free theorem to a loaded family.

A separate same-trace potential perturbation leaves its exactly
assembled gauge unchanged to the reported arithmetic scale but
changes its four-point approximation. The complete record in
\cref{app:supporting-observations} distinguishes boundary equivalence
from cellwise exactness after an arbitrary potential change.

The supporting complete-pressure and near-kernel arithmetic checks
retain their separate matrix and component scales. They do not turn
a local compatibility identity into a uniform floating-point guarantee.
See \cref{app:supporting-observations}.

\section{Discussion and conclusions}
\label{sec:discussion}
\label{sec:conclusions}

MORTIS gives a discrete design theory for cofactor reference forms:
the exact kernel identifies the available potentials, the coefficient
loss quantifies their attainable material margins, and boundary
compensation preserves the complete tangent.
The tensor, anisotropic, serendipity and Tet10 classifications depend
on the actual variation and potential spaces. Degree counting alone
misses the one-high-axis exception and the mixed serendipity modes.

The optimized margin tests every matrix increment at the prescribed
quadrature points. Attainment concerns the finite evaluation image;
ideal-margin rigidity additionally requires the stated candidate-space
and gradient-sampling assumptions.
Fixed-candidate enrichment can lower the best certificate by linear
order in curvature amplitude, uniformly in cell size.
At the largest reported amplitude, the explicit potential loses
\(12.7094\%\) of its material margin, while every compatible candidate
must lose at least \(1.20153\%\).
The numerical optimum remains uncomputed.
The separate constant-coefficient trade-off is sharp in its specified
basis, norm and pointwise positivity class.

A material certificate, a physical coercivity estimate and a complete
operator spectrum measure different properties. On the smooth tensor
example, pressure refinement changes a selected-direction Rayleigh
value while the material margin is unchanged.
A Gauss-three material-metric difference also coexists with exact
gauge integration. The norm-transfer lemma identifies the additional
geometry and quadrature assumptions required for a physical bound.
An essential constraint removing constants on one mesh does not
alone supply a uniform Poincar\'e constant.

Compatible and defect-corrected potentials give the same exact
assembled gauge when their full potential traces agree. The
prescribed-trace repair changes a local representation of that
operator; its finite-dimensional solvability does not imply a uniform
repair bound or economical computation.
With a full displacement clamp, the compatible reference equals
the complete tangent. The resulting one-column exact-preconditioning
observation is a consequence of this equality.
The partial-boundary MINRES comparisons show no consistent advantage
for the gauge alone. A useful approximation to the signed boundary
Schur action would require separate cost and residual analysis.

The uniform quadratic-tensor theorem and the finite-shear certificates
retain their complete material, geometry, quadrature and boundary
hypotheses. Their recovered pressures vanish, so arbitrary
nonnegative bulk moduli add nonnegative outer terms.
The loaded example has a larger negative pressure-geometric
contribution than its positive material-plus-gauge and volume-outer
contributions. Omitting that term would change the equation.
The exact boundary representation survives; reference positivity
at nonzero pressure needs further control.

Rational and interval certificates prove the specified parameter
ranges; sampled eigenvalues retain their numerical status.
Near-kernel cofactor components can have small absolute error and
large relative action error. Complete-action accuracy therefore
depends on the signed composition and pressure terms
\citep{higham2002accuracy,croci2024mixed}.
These restrictions leave the core distinction explicit:
compatibility selects a representation, material margins certify
local positivity, and the proved coercivity regimes establish when
that choice yields a complete positive reference.

\appendix
\section{Exact Tet10 gauge structure}
\label{app:tet10-gauge-structure}

This appendix gives the complete coordinate derivations for the sharp
tetrahedral instance. They describe a local cubature component before
displacement boundary elimination; global cancellation and pointwise
material margins have their separately stated meanings.

Let
\[
\That=\{\xi\in\R^3:\xi_i\ge0,\ \xi_1+\xi_2+\xi_3\le1\},
\qquad
\lambda_0=1-\xi_1-\xi_2-\xi_3,\quad \lambda_i=\xi_i\ (i=1,2,3).
\]
Write $\mathbb P_r(\That)$ for scalar polynomials of total degree at most
$r$. Tet10 maps and variations belong to $[\mathbb P_2(\That)]^3$.
Their quadratic Lagrange functions are
$N_i=\lambda_i(2\lambda_i-1)$ and $N_{ij}=4\lambda_i\lambda_j$.
The tetrahedral rule $\Qfour$ has weight $1/24$ at each permutation of
\[
(\alpha_4,\beta_4,\beta_4,\beta_4),\qquad
\alpha_4=\frac{5+3\sqrt5}{20},\qquad
\beta_4=\frac{5-\sqrt5}{20}.
\]
It is exact through total degree two \citep{shunn2012symmetric}.
This notation is distinct from tensor polynomial spaces and tensor Gauss
point counts.

For the affine parent Jacobian $J_y=D_\xi y_h$, let $J_{y,a}$ be its
four vertex values and $\overline J_y=\frac14\sum_aJ_{y,a}$. The centered
determinant identity is
\begin{equation}
\Qfour[\det J_y]-\int_{\That}\det J_y
 =c_4\sum_{a=0}^3\det(J_{y,a}-\overline J_y),\qquad
c_4=\frac{3\sqrt5-5}{1800}.
\label{eq:companion-cubic}
\end{equation}
The same identity is used in RUPA \citep{liu2026factor}.
Here is a short derivation. Put $B_a=J_{y,a}-\overline J_y$, so
$\sum_aB_a=0$. All determinant terms containing the mean have spatial
degree at most two and integrate exactly. If $e_3,e_{21},e_{111}$ are
the cubic moment errors for $\lambda_i^3$, $\lambda_i^2\lambda_j$ and
$\lambda_i\lambda_j\lambda_k$ with distinct indices, quadratic exactness
gives $e_{21}=-e_3/3$ and $e_{111}=e_3/3$.
The symmetric determinant expansion with $\sum_aB_a=0$ therefore gives
$(e_3-3e_{21}+2e_{111})\sum_a\det B_a$.
Finally,
$e_3=(\alpha_4^3+3\beta_4^3)/24-1/120$ and
$(8/3)e_3=c_4$, proving \eqref{eq:companion-cubic}.
Polarization connects its Hessian to
$\mathcal T_e(y_h;u_h,v_h)$ in \eqref{eq:common-trilinear-defect}.

\subsection{Current and reference geometry in the quadratic case}

For fixed quadratic variations $u_e,v_e$ and a constant physical
coefficient $R\in\mathbb M$, define
\[
\mathcal E_J(y_e;u_e,v_e)=\mathcal T_e(y_e;u_e,v_e),\qquad
\mathcal E_C(X_e,R;u_e,v_e)=\mathcal T_e(RX_e;u_e,v_e).
\]
The pullbacks \eqref{eq:det-parent} and \eqref{eq:cof-parent} identify
these with the determinant- and constant-coefficient cofactor-curvature
errors, respectively.

\begin{theorem}[Current and reference geometry in the quadratic defect]
\label{thm:dual-crime}
For quadratic maps and variations, both defining integrands have total
degree at most three. The determinant error is linear in the current
parent Jacobian and vanishes when the current map is parent-affine.
The cofactor error is linear in the reference parent Jacobian and
vanishes when the reference map is parent-affine.
For parent-affine maps $\overline y_e,\overline X_e$ and a scalar $s$,
\[
\mathcal E_J(\overline y_e+s(y_e-\overline y_e);u_e,v_e)
=s\mathcal E_J(y_e;u_e,v_e),
\]
\[
\mathcal E_C(\overline X_e+s(X_e-\overline X_e),R;u_e,v_e)
=s\mathcal E_C(X_e,R;u_e,v_e).
\]
\end{theorem}
\begin{proof}
Each parent gradient is affine, and the determinant polarization is
trilinear in its three gradient arguments. The integrand consequently
has degree at most three. A constant first gradient leaves a product
of degree at most two, which $\Qfour$ integrates exactly.
Linearity in that gradient proves both scaling identities.
\end{proof}

The result is specific to quadratic variations and the stated rule.
At higher order, an affine first argument need not make the remaining
gradient product exact. A variable constitutive coefficient also keeps
its variation in the original quadrature form; the cofactor dependence
stated above concerns the fixed coefficient $R$.

\subsection{An opposite-edge determinant representation}

For a quadratic vector map $y_h$, let $y_i$ be its vertex values and
$y_{ij}$ its edge-midpoint values. Define the midside deviations by
$d_{ij}(y_h)=y_{ij}-(y_i+y_j)/2$, and write $d_{ij}=d_{ij}(y_h)$. We use the edge order
$(01,12,20,03,13,23)$, with $d_{20}=d_{02}$.
More generally $d_{ji}=d_{ij}$, since the deviations are independent
of edge orientation.
Let $d\in\R^{18}$ collect these six vectors. Define two three-by-three
matrices by listing their columns:
\[
D_+(d)=[d_{01}+d_{23},\,d_{12}+d_{03},\,d_{20}+d_{13}],\qquad
D_-(d)=[d_{01}-d_{23},\,d_{12}-d_{03},\,d_{20}-d_{13}].
\]
Let $\mathscr V(d)$ be the scalar determinant cubature defect
in \eqref{eq:companion-cubic}.
Its Hessian in the Euclidean midside-deviation coordinates is denoted by
$\mathscr R(d)\in\R^{18\times18}$.

\begin{theorem}[Opposite-edge normal form]
\label{thm:opposite-edge-normal-form}
For every quadratic map,
\[
\mathscr V(d)=16c_4\bigl(\det D_-(d)-2\det D_+(d)\bigr).
\]
\end{theorem}

\begin{proof}
The affine part of $y_h$ drops out of \eqref{eq:companion-cubic}.
Write $g_a=\nabla_\xi\lambda_a$, so
$g_0=(-1,-1,-1)^T$ and $g_1,g_2,g_3$ are the Cartesian unit vectors.
For the quadratic part $4\sum_{i<j}\lambda_i\lambda_jd_{ij}$, the
centered vertex Jacobian is
\[
\delta J_{y,a}
=4\sum_{j\ne a}d_{aj}\otimes g_j
 -\sum_{i<j}d_{ij}\otimes(g_i+g_j).
\]
Put $P=D_+(d)$ and $M=D_-(d)$. Substitution gives
$\delta J_{y,a}=PA_a+MB_a$, with the following constant matrices:
\[
\renewcommand{\arraystretch}{1.25}
\begin{array}{c|cc}
a&A_a&B_a\\ \hline
0&
\begin{bmatrix}2&0&0\\0&0&2\\0&2&0\end{bmatrix}&
\begin{bmatrix}2&1&1\\-1&-1&-2\\1&2&1\end{bmatrix}\\[5pt]
1&
\begin{bmatrix}-2&-2&-2\\0&2&0\\0&0&2\end{bmatrix}&
\begin{bmatrix}-2&-1&-1\\-1&1&0\\1&0&-1\end{bmatrix}\\[5pt]
2&
\begin{bmatrix}0&0&2\\2&0&0\\-2&-2&-2\end{bmatrix}&
\begin{bmatrix}0&1&-1\\1&-1&0\\-1&-2&-1\end{bmatrix}\\[5pt]
3&
\begin{bmatrix}0&2&0\\-2&-2&-2\\2&0&0\end{bmatrix}&
\begin{bmatrix}0&-1&1\\1&1&2\\-1&0&1\end{bmatrix}.
\end{array}
\]
They satisfy
\[
\sum_a\det A_a=-32,\quad \sum_a\det B_a=16,\quad
\sum_a B_a(\cof A_a)^T=0,\quad
\sum_a A_a(\cof B_a)^T=0.
\]
Use
$\det(U+V)=\det U+\det V+\cof U{:}V+\cof V{:}U$.
Cofactor multiplicativity makes the two mixed sums vanish by the
displayed matrix identities. Therefore
$\sum_a\det(\delta J_{y,a})=-32\det P+16\det M$.
Multiplication by $c_4$ proves the result.
\end{proof}

\subsection{Extension to symmetric quadratic-exact rules}
\label{sec:symmetric-rule-extension}

A cubature rule $Q$ on $\That$ is \emph{fully symmetric} if it is
unchanged by every permutation of the four barycentric coordinates.
It is \emph{quadratic-exact} if it integrates every scalar polynomial
of total degree at most two exactly. Symmetric moment reduction is
classical; the invariant-polynomial basis for tetrahedral rules is
treated systematically by \citet{wang2023symmetric}.
We use its degree-three consequence to identify the rule dependence
of the preceding operator.

For such a rule, define its cubic moment defect and coefficient by
\[
m_3(Q)=Q[\lambda_0^3]-\frac1{120},
\qquad c_Q=\frac83m_3(Q).
\]
The integral of $\lambda_0^3$ on the unit parent tetrahedron is $1/120$.

\begin{proposition}[One rule-dependent cubic coefficient]
\label{prop:symmetric-rule-cubic}
Let $J(\lambda)=\sum_{a=0}^3\lambda_aJ_a$ be any affine
three-by-three matrix field, and put $\bar J=\frac14\sum_aJ_a$.
Every fully symmetric quadratic-exact rule satisfies
\[
Q[\det J]-\int_{\That}\det J
=c_Q\sum_{a=0}^3\det(J_a-\bar J).
\]
Consequently the Tet10 scalar defect has the same opposite-edge
normal form as in \Cref{thm:opposite-edge-normal-form}, with
$c_4$ replaced by $c_Q$.
Moreover, $c_Q=0$ if and only if $Q$ is exact through degree three.
\end{proposition}

\begin{proof}
Let $\mathcal L_Q$ denote cubature minus exact integration.
It is invariant under the permutation group and annihilates
quadratic polynomials. Symmetrizing a polynomial therefore
leaves its $\mathcal L_Q$ value unchanged.
Modulo polynomials of degree at most two, the symmetric cubic
space is one-dimensional: the elementary symmetric polynomials
have degrees one, two, three, and four, and the degree-one
polynomial $\sum_a\lambda_a$ is the constant one on $\That$.

For a cubic polynomial $f$ on the barycentric hyperplane
$\sum_a\lambda_a=1$, let $f_3$ be the homogeneous
degree-three part of its Taylor expansion about the barycenter.
Let $v_a$ be the four barycentric vertex vectors and
$\bar v=(1/4,1/4,1/4,1/4)$.
The functional $\mathcal N(f)=\sum_a f_3(v_a-\bar v)$
is also invariant and annihilates quadratics. For
$f=\lambda_0^3$, it equals
$(3/4)^3+3(-1/4)^3=3/8$, so it is nonzero.
The one-dimensional quotient gives
$\mathcal L_Q=c_Q\mathcal N$ on cubics.
For $f=\det J$, its homogeneous cubic part at $v_a-\bar v$
is $\det(J_a-\bar J)$, proving the formula.
The same quotient proves the final assertion.
\end{proof}

When $c_Q\ne0$, the local faithfulness and alignment results below,
and the corresponding constant-gauge defect comparison, transfer to this rule.
Norm formulas use $|c_Q|$ in place of $c_4$; a negative
coefficient reverses the local signs. When $c_Q=0$, these
cubic defects vanish for every geometry, so an alignment
restriction is unnecessary for these components.
This concerns the constant-coefficient minor terms.
The complete material energy and its rational pullbacks
have additional quadrature dependence.
Changing $Q$ in the complete energy defines a different
discrete equation.

\subsection{Norm, rank, and sign of the four-point defect}

For a three-by-three matrix $A$, let $\mathsf D(A)$ denote the
nine-by-nine matrix of $D^2\det(A)$ in a Frobenius-orthonormal entry basis.
Vectorizing $D_+(d)/\sqrt2$ and $D_-(d)/\sqrt2$ defines an orthogonal
change of the eighteen coordinates. Denote that change by $T$.

\begin{corollary}[Faithfulness and an exact norm]
\label{cor:faithful-jet}
The Hessian has the orthogonal block representation
\[
T\mathscr R(d)T^T
=\operatorname{diag}\bigl(-64c_4\mathsf D(D_+(d)),
                          32c_4\mathsf D(D_-(d))\bigr).
\]
Consequently
\begin{equation}
\|\mathscr R(d)\|_F^2
=4096c_4^2\bigl(4\|D_+(d)\|_F^2+\|D_-(d)\|_F^2\bigr).
\label{eq:jet-exact-norm}
\end{equation}
In particular,
\[
64\sqrt2\,c_4\|d\|_2
\le\|\mathscr R(d)\|_F
\le128\sqrt2\,c_4\|d\|_2 .
\]
The map $d\mapsto\mathscr R(d)$ is injective. Its singular values,
for these domain and matrix norms, are $64\sqrt2c_4$ and
$128\sqrt2c_4$, each with multiplicity nine.
\end{corollary}

\begin{proof}
Differentiate \Cref{thm:opposite-edge-normal-form} in the orthogonal
coordinates. Each opposite-edge sum or difference is $\sqrt2$ times
its new coordinate; the two derivative factors give the displayed
blocks. In an entry basis, every entry of $A$ occurs in four signed
entries of $\mathsf D(A)$, and different entries have disjoint supports.
Thus $\|\mathsf D(A)\|_F^2=4\|A\|_F^2$, proving
\eqref{eq:jet-exact-norm}.
The opposite-edge identity
$\|D_+(d)\|_F^2+\|D_-(d)\|_F^2=2\|d\|_2^2$
gives the bounds. The sum and difference subspaces each have dimension
nine; the norm identity gives the two singular values on them.
\end{proof}

These norms concern the stated quotient coordinates. A physical energy
norm requires the element gradient metric. The two-to-one singular-value
ratio concerns the geometry-to-matrix map. The action of an individual
matrix on directions near its kernel can still be sensitive. Relative
action errors then require the magnitude of that action as well as the
matrix norm \citep{higham2002accuracy}.

The determinant-Hessian calculation uses standard constitutive spectral
algebra \citep{poya2023variational,poya2025generalised}.
The normal form identifies its precise occurrence in the Tet10 cubature
operator.

For a real symmetric matrix $S$, let
$\In(S)=(n_+(S),n_-(S),n_0(S))$ count its positive, negative and zero
eigenvalues, including multiplicity.

\begin{corollary}[Local rank and sign]
\label{cor:jet-inertia}
The inertia of $\mathscr R(d)$ is the sum of the inertias of
$-\mathsf D(D_+(d))$ and $\mathsf D(D_-(d))$.
For any $A\in\mathbb M$, the determinant Hessian has inertia
\[
\begin{array}{c|c}
\text{condition on }A&\In(\mathsf D(A))\\ \hline
\operatorname{rank}A=0&(0,0,9)\\
\operatorname{rank}A=1&(2,2,5)\\
\operatorname{rank}A=2&(3,3,3)\\
\det A>0&(4,5,0)\\
\det A<0&(5,4,0).
\end{array}
\]
Every nonzero local cubic defect Hessian is indefinite.
\end{corollary}

\begin{proof}
The first assertion follows from the orthogonal block representation
and $c_4>0$. If $A$ is invertible, write an increment as $H=AX$.
Then
$D^2\det(A)[H,H]=\det A((\tr X)^2-\tr(X^2))$.
The trace direction and three skew directions are positive, and the
five symmetric trace-free directions are negative, when $\det A>0$.
The signs reverse when $\det A<0$.
Invertible changes of left and right coordinates reduce a rank-two
matrix to $\operatorname{diag}(1,1,0)$ and a rank-one matrix to
$\operatorname{diag}(1,0,0)$, up to a nonzero scalar factor.
Direct expansion gives three positive and three negative directions
in the first case and two of each in the second; the remaining
directions vanish. A scalar sign cannot change these balanced counts.
Rank zero gives the zero matrix.
If $d\ne0$, at least one of $D_+(d),D_-(d)$ is nonzero.
Its determinant-Hessian block has both signs.
\end{proof}

Local indefiniteness supplies no universal sign conclusion for
a constrained assembled correction; assembly can remove local directions.

\subsection{Exactness of a geometry-aligned translation}

Let $d_e(X)$ denote the six reference-map midside deviations on one
element, and let
$\mathcal C_X=\operatorname{span}\{d_{ij}(X):i<j\}\subset\R^3$.
For a constant physical cofactor coefficient $Z$, write
$\mathsf E_X(Z)$ for the eighteen-coordinate matrix of
$\mathcal T_e(ZX_e;\cdot,\cdot)$, with the tetrahedral four-point
error functional. The potential pullback is $ZX_e$, whose parent
gradient is $ZJ_{X,e}$.
Here $Zd_e(X)$ means that $Z$ acts on each of the six vectors.

\begin{proposition}[Necessary and sufficient alignment]
\label{prop:exact-gauge-alignment}
The following conditions are equivalent:
\[
\mathsf E_X(Z)=0,\qquad
Zd_e(X)=0,\qquad
Z\mathcal C_X=\{0\},\qquad
ZX_h\text{ is parent-affine}.
\]
\end{proposition}

\begin{proof}
The common polarized functional gives
$\mathsf E_X(Z)=\mathscr R(Zd_e(X))$.
\Cref{cor:faithful-jet} makes its vanishing equivalent to
$Zd_e(X)=0$. The span condition is the same statement.
A quadratic map is affine precisely when all its midside deviations
vanish, giving the last equivalence.
\end{proof}

Thus the determinant tangent is exact under the four-point rule for
every local pair of displacement variations precisely when its current
quadratic map is affine. Exact scalar volume at a curved map is a
weaker property.
Let $W_{\rm iso}$ be any twice-differentiable isochoric density.
A full three-dimensional space $\mathcal C_X$
also obstructs a positive exactly aligned gauge at every deformation:
alignment forces $Z=0$, while scale invariance gives
$D^2W_{\rm iso}(F)[F,F]=0$.

\section{First- and second-minor material calculus}
\label{app:material-calculus}

The identities here close the constitutive derivation used by the reference
forms. Fix $F\in\GLp$ and write
$J=\det F$, $C=\cof F$, $I_1=F{:}F$, $I_2=C{:}C$.
For increments $H,L$, put $\theta_H=F^{-T}{:}H$ and define $\theta_L$
similarly. The form $Q_F$, its rescaling $A_1$ and the coefficient $\pi_1$
were defined in \eqref{eq:qf}--\eqref{eq:pi1}.

\subsection{The first-minor completed square}

The form $Q_F$ is positive definite: if its sum of squares in
\eqref{eq:qf} vanishes, $I_1>0$ forces $\theta_H=0$, and its remaining
square forces $H=0$. Symmetry and bilinearity are immediate.
To derive \eqref{eq:first-invariant-import}, differentiation of
$\bar I_1=J^{-2/3}I_1$ gives
\[
D\bar I_1(F)[H]
 =2J^{-2/3}\left(F{:}H-\frac{I_1}{3}\theta_H\right).
\]
A second differentiation, using
$D\theta_H(F)[L]=-\tr(F^{-1}LF^{-1}H)$, yields
\begin{align*}
D^2\bar I_1(F)[H,L]
={}&2J^{-2/3}H{:}L
-\frac43J^{-2/3}\bigl(\theta_LF{:}H+\theta_HF{:}L\bigr)\\
&+\frac{4I_1}{9}J^{-2/3}\theta_H\theta_L
+\frac{2I_1}{3}J^{-2/3}\tr(F^{-1}HF^{-1}L).
\end{align*}
Expanding $Q_F$ and substituting \eqref{eq:det-second} rewrites this as
$A_1(F)[H,L]-\pi_1(F)D^2J(F)[H,L]$, as claimed.
This is the first-minor Flory identity also used by the companion
RUPA analysis \citep{liu2026factor}; its derivation is included here.

\subsection{The cofactor chain rule}

Write $K_F(H)=D\cof(F)[H]$ as in \eqref{eq:dcof}.
The modified second invariant can be read as a first-invariant Flory function
of the cofactor:
\begin{equation}
\bar I_2(F)
=(\det C)^{-2/3}C{:}C,
\qquad
\det C=J^2.
\label{eq:i2-as-flory}
\end{equation}
This observation permits reuse of the completed square in
\eqref{eq:qf}, followed by an exact cofactor chain rule.

For $M,N\in\R^{3\times3}$, define
\begin{align}
\vartheta_M&=C^{-T}{:}M,\\
\vartheta_N&=C^{-T}{:}N,\\
Q_C(M,N)
&=\left(M-\frac23\vartheta_MC\right){:}
\left(N-\frac23\vartheta_NC\right)
+\frac{I_2}{9}\vartheta_M\vartheta_N.
\label{eq:qc}
\end{align}
The same sum-of-squares proof as for $Q_F$ shows that $Q_C$ is positive
definite.  Introduce
\begin{align}
A_C(F)[H,L]
&=2J^{-4/3}Q_C(K_F(H),K_F(L)),
\label{eq:ac}\\
\pi_C(F)&=\frac{2I_2}{3J^{10/3}},
\label{eq:pic}\\
P_C(F)&=2J^{-4/3}
\left(C-\frac{I_2}{3}C^{-T}\right).
\label{eq:pc}
\end{align}

\begin{theorem}[Completed second-invariant Hessian]
\label{thm:i2-completion}
For every $F\in\GLp$ and $H,L\in\R^{3\times3}$,
\begin{align}
D^2\bar I_2(F)[H,L]
={}&A_C(F)[H,L]
-2\pi_C(F)DJ(F)[H]DJ(F)[L]\nonumber\\
&-\pi_C(F)J D^2J(F)[H,L]
+P_C(F){:}(H\times L).
\label{eq:i2-completion}
\end{align}
The form $A_C$ is positive definite on $\R^{3\times3}$.
\end{theorem}

\begin{proof}
Define $\phi(C)=(\det C)^{-2/3}C{:}C$.  Applying the first-invariant
identity \eqref{eq:first-invariant-import} with $C$ as the matrix variable
gives
\begin{equation}
D^2\phi(C)[M,N]
=2J^{-4/3}Q_C(M,N)-\pi_C D^2\det(C)[M,N].
\label{eq:phi-second}
\end{equation}
Its first derivative is
\begin{equation}
D\phi(C)[M]=P_C{:}M.
\label{eq:phi-first}
\end{equation}

We next evaluate the determinant term after cofactor pullback.  Since
$\det(\cof F)=J^2$, twice differentiating this composition yields
\begin{align*}
D^2\det(C)[K_F(H),K_F(L)]
+D\det(C){:}(H\times L)
=2DJ[H]DJ[L]+2JD^2J[H,L].
\end{align*}
Here $D\det(C)=\cof C=JF$, and \eqref{eq:det-cross} gives
$D\det(C){:}(H\times L)=JD^2J[H,L]$.  Therefore
\begin{equation}
D^2\det(C)[K_F(H),K_F(L)]
=2DJ[H]DJ[L]+JD^2J[H,L].
\label{eq:det-cof-chain}
\end{equation}

The second-order chain rule for $\bar I_2=\phi\circ\cof$ is
\[
D^2\bar I_2[H,L]
=D^2\phi(C)[K_F(H),K_F(L)]
+D\phi(C){:}(H\times L).
\]
Substitution of \eqref{eq:phi-second}, \eqref{eq:phi-first}, and
\eqref{eq:det-cof-chain} proves \eqref{eq:i2-completion}.

The form $Q_C$ is positive definite, $J^{-4/3}>0$, and
$K_F$ is invertible by \Cref{lem:cofactor-derivatives}.  Their pullback
$A_C$ is therefore positive definite.
\end{proof}

The second-invariant factor that will join the first-minor gauge is
\begin{equation}
B_2(F)=A_C(F)-2\pi_C(F)DJ(F)\otimes DJ(F).
\label{eq:b2}
\end{equation}
It is obtained from the positive form $A_C$ by subtracting one scalar outer
product, so its sign can depend on the deformation.

For fixed nonnegative coefficients $c_1,c_2$ with $c_1+c_2>0$, write
$W_{\rm MR}$ for the cell density in \eqref{eq:mooney-energy}.
Define the two-minor gauge
\begin{equation}
G_{\rm MR}(F)=c_1A_1(F)+c_2B_2(F).
\label{eq:gmr}
\end{equation}

Define the combined determinant-curvature coefficient by
\begin{equation}
\Pi_J(F)=c_1\pi_1(F)+c_2\pi_C(F)J.
\label{eq:combined-pressure}
\end{equation}
Combining \eqref{eq:first-invariant-import} and
\Cref{thm:i2-completion} gives the full isochoric Hessian
\begin{align}
D^2W_{\rm MR}(F)[H,L]
={}&G_{\rm MR}(F)[H,L]
-\Pi_J(F)D^2J(F)[H,L]\nonumber\\
&+c_2P_C(F){:}(H\times L),
\label{eq:mr-full-split}
\end{align}
The last two terms are respectively determinant and cofactor curvature.

On a rank-one increment $H=a\otimes n$, both $D^2J(F)[H,H]$ and
$H\times H=2\cof H$ vanish. Thus the complete rank-one isochoric value
in \eqref{eq:mr-full-split} equals $G_{\rm MR}(F)[H,H]$.
Positivity of that chosen gauge is a sufficient rank-one ellipticity
condition. The exact decomposition keeps both curvature channels available
for the discrete compatibility and compensation analysis.

\section{Exact non-aligned reference geometry}
\label{app:nonaligned-geometry}

These data specify the reference mesh in
\Cref{prop:nonaligned-uniform-core}. The deformation, point metric
$A_1(D)$ and full interval-congruence argument are given in
\Cref{sec:coercivity-regimes}; the integer coordinates below complete that
specification.

Divide each axis of $[0,1]^3$ into two intervals.
Split each of the eight cubes into the six tetrahedra
whose vertex paths increment the three coordinate
directions in each possible order.
Use an orientation-preserving local vertex order
and the six midside nodes of the quadratic element.
The resulting parameter nodes form the quarter grid
$t=(i,j,k)/4$, $0\le i,j,k\le4$.
The Cartesian vectors $e_1,e_2,e_3$ determine the coordinate directions;
the current fields use the fixed matrix $D=\operatorname{diag}(2,2,1/4)$
and $E_{12}=e_1\otimes e_2$ defined with the shear family.

Set $X(t)=t$ at every boundary node, and set
$X_2=t_2$, $X_3=t_3$ at every node.
For the interior nodes, define integers $n_{ijk}$
by \Cref{tab:nonaligned-reference} and set
$X_1(i/4,j/4,k/4)=n_{ijk}/2^{55}$.
These rational coordinates specify the geometry
independently of the sampling procedure that
first produced it.

\begin{table}[H]
\centering
\small
\caption{Exact first-coordinate numerators for the
non-aligned reference mesh. The common denominator
is $2^{55}$.}
\label{tab:nonaligned-reference}
\begin{tabular}{@{}ccrrr@{}}
\toprule
$j$&$k$&$n_{1jk}$&$n_{2jk}$&$n_{3jk}$\\
\midrule
1&1&9105314122432956&17060509236021852&28356459511241252\\
1&2&12049108527703530&17816379473272022&27203240556309152\\
1&3&7623885483807398&19502005720601204&26157924656613804\\
2&1&10316193290055142&16881392942710108&27515618967020928\\
2&2&8671466383741623&16989823806708642&24586724701378688\\
2&3&11123479124564504&19221453749166012&24175413339273560\\
3&1&11432525502845932&17898559407134226&29511977488033696\\
3&2&7213745460801190&18134896708814432&27423007014701600\\
3&3&9189194271142126&18086850943127236&27280126714252592\\
\bottomrule
\end{tabular}
\end{table}

\section{Full-curvature compatible-potential mesh}
\label{app:compatible-geometry}

This is the exact reference geometry for
\Cref{prop:compatible-uniform-family}. Its complete finite-shear proof
and the physical quadrature metric $A_{I,h}$ are in
\Cref{sec:coercivity-regimes}. The integer table and its normalization
below specify the certificate inputs.

Partition each cube of the corner grid $\{0,1/2,1\}^3$ into the six
Freudenthal tetrahedra given by the coordinate-order paths.
The resulting quadratic mesh has $48$ cells and $125$ nodes,
at $t=(i,j,k)/4$ with $i,j,k\in\{0,1,2,3,4\}$.
Let $a_h$ be the continuous parent-quadratic vector field whose
nodal values are the integer vectors in \Cref{tab:compatible-noise};
its values are zero when all three indices are even.
The rational reference map is
\[
X_h(t)=t+\frac7{3768}a_h(t).
\]
Thus every reference corner is unchanged, and
$\overline X_h(X_h(t))=t$, where $\overline X_h$ is the continuous
corner-linear potential field. The table uses the three-digit key $ijk$
for the node $(i,j,k)/4$, with vector components $(a_1,a_2,a_3)$.

Let $v_{e,a}$ denote the parameter-cell corners. The matrix $1$- and
$\infty$-norms are the maximum absolute column and row sums.
The integer vertex derivatives satisfy
\[
\max_{e,a}
\norm{D_ta_h(v_{e,a})}_1\,
\norm{D_ta_h(v_{e,a})}_\infty=221760.
\]
The spectral norm is bounded by the square root of this product.
Since $D_ta_h$ is affine on each cell, convexity of the spectral
norm extends the vertex bound to the whole cell. Therefore
\[
\norm{D_tX_h-I_3}_2
\le \frac7{3768}\sqrt{221760}
=\sqrt{\frac{18865}{24649}}<1.
\]
The perturbation is globally Lipschitz on the convex parameter
cube with this constant. The map is strongly monotone and hence
injective; its Jacobian is invertible with positive determinant.
Nonzero integer three-by-three minors of the six edge-deviation
vectors verify full three-directional curvature on all cells.

The compatible-core verification package described in \Cref{sec:certificate-files}
contains the node and connectivity arrays,
integer minor witnesses, the rational Bernstein matrices,
and all interval elimination records for
\Cref{prop:compatible-uniform-family}. Coordinate enclosures
include both this rational map and its retained binary64
realization. The proof concerns the stated exact-real family;
it makes no assertion of uniformly accurate floating-point
arithmetic as the bulk moduli tend to infinity.

\begin{table}[!htbp]
\centering
\footnotesize
\setlength{\tabcolsep}{5pt}
\caption{Integer nodal coefficients of $a_h$.
All-even node indices have coefficient $(0,0,0)$ and are omitted.}
\label{tab:compatible-noise}
\begin{tabular}{@{}crrr@{\quad}crrr@{\quad}crrr@{}}
\toprule
$ijk$ & $a_1$ & $a_2$ & $a_3$ &
$ijk$ & $a_1$ & $a_2$ & $a_3$ &
$ijk$ & $a_1$ & $a_2$ & $a_3$\\
\midrule
\texttt{001} & -5 & 15 & 7 & \texttt{132} & -14 & -9 & 7 & \texttt{314} & 9 & -15 & -11 \\
\texttt{003} & 5 & -2 & -5 & \texttt{133} & -11 & -2 & 7 & \texttt{320} & 6 & 16 & -12 \\
\texttt{010} & 12 & -10 & -10 & \texttt{134} & -6 & -6 & 9 & \texttt{321} & -7 & 12 & 8 \\
\texttt{011} & 2 & -4 & 2 & \texttt{140} & -7 & -13 & -4 & \texttt{322} & -5 & -2 & 7 \\
\texttt{012} & -10 & 4 & 12 & \texttt{141} & 9 & -5 & 9 & \texttt{323} & -3 & -11 & -6 \\
\texttt{013} & 13 & 2 & -4 & \texttt{142} & 4 & 1 & 0 & \texttt{324} & 5 & 15 & 14 \\
\texttt{014} & 8 & 9 & -9 & \texttt{143} & 12 & 10 & 14 & \texttt{330} & -2 & 13 & 10 \\
\texttt{021} & -15 & 9 & 6 & \texttt{144} & -1 & -9 & -12 & \texttt{331} & -8 & -14 & -8 \\
\texttt{023} & 7 & -10 & 10 & \texttt{201} & -3 & 3 & -1 & \texttt{332} & -6 & -10 & -16 \\
\texttt{030} & 4 & 6 & 16 & \texttt{203} & 3 & -7 & -1 & \texttt{333} & 11 & 2 & 7 \\
\texttt{031} & -3 & -11 & -3 & \texttt{210} & 6 & -15 & 11 & \texttt{334} & 3 & 11 & -5 \\
\texttt{032} & -9 & 1 & -4 & \texttt{211} & -3 & -16 & -3 & \texttt{340} & -2 & 9 & 2 \\
\texttt{033} & 0 & -8 & -8 & \texttt{212} & -8 & 7 & -16 & \texttt{341} & 15 & 1 & 11 \\
\texttt{034} & 8 & -8 & 10 & \texttt{213} & -11 & 11 & 1 & \texttt{342} & 0 & 12 & 8 \\
\texttt{041} & 9 & -3 & -8 & \texttt{214} & -1 & -9 & -12 & \texttt{343} & 7 & 12 & -6 \\
\texttt{043} & 14 & -4 & 2 & \texttt{221} & -9 & 15 & 4 & \texttt{344} & 16 & -10 & -6 \\
\texttt{100} & -1 & 11 & 7 & \texttt{223} & -13 & 16 & -9 & \texttt{401} & -14 & 14 & 8 \\
\texttt{101} & 13 & 4 & 1 & \texttt{230} & -3 & -5 & 10 & \texttt{403} & 8 & -14 & -3 \\
\texttt{102} & -5 & -2 & 13 & \texttt{231} & -8 & -12 & -1 & \texttt{410} & -12 & -12 & 5 \\
\texttt{103} & -9 & 4 & -7 & \texttt{232} & 0 & 6 & -9 & \texttt{411} & 11 & -5 & -14 \\
\texttt{104} & 8 & -5 & 1 & \texttt{233} & -9 & -1 & 13 & \texttt{412} & 7 & 16 & -2 \\
\texttt{110} & -16 & 3 & -8 & \texttt{234} & 2 & 6 & -11 & \texttt{413} & 7 & -14 & 12 \\
\texttt{111} & 12 & 0 & 9 & \texttt{241} & -2 & 11 & 1 & \texttt{414} & 11 & -2 & -13 \\
\texttt{112} & 9 & 13 & -12 & \texttt{243} & 11 & -5 & -15 & \texttt{421} & -12 & 6 & 7 \\
\texttt{113} & 6 & -4 & -16 & \texttt{300} & -4 & -5 & -8 & \texttt{423} & -7 & 10 & -2 \\
\texttt{114} & 2 & 2 & 15 & \texttt{301} & 8 & -9 & -7 & \texttt{430} & -7 & -13 & 10 \\
\texttt{120} & 7 & -7 & 11 & \texttt{302} & -4 & -15 & -2 & \texttt{431} & 16 & -13 & 7 \\
\texttt{121} & -13 & -10 & -6 & \texttt{303} & -14 & 12 & -3 & \texttt{432} & 13 & -12 & -10 \\
\texttt{122} & -14 & 9 & 5 & \texttt{304} & -6 & -10 & -10 & \texttt{433} & 11 & -14 & -11 \\
\texttt{123} & -14 & -10 & 2 & \texttt{310} & 2 & 15 & -10 & \texttt{434} & 3 & 1 & 12 \\
\texttt{124} & 8 & -5 & 11 & \texttt{311} & 9 & -3 & 3 & \texttt{441} & -8 & -8 & 1 \\
\texttt{130} & 11 & 3 & 6 & \texttt{312} & -10 & -10 & -16 & \texttt{443} & -9 & 14 & -14 \\
\texttt{131} & -14 & 9 & 7 & \texttt{313} & 2 & -13 & -9 &  & & &  \\
\bottomrule
\end{tabular}
\end{table}
\FloatBarrier

\section{Exact representatives with a prescribed potential trace}
\label{app:trace-representative}

There is a finite-dimensional way to test whether a desired potential
trace admits an exact representative. The construction uses ordinary
orthogonal projection and least-squares algebra
\citep{higham2002accuracy}; its constants depend on the selected
coefficient and residual metrics.
This is an auxiliary representation problem after the potential trace
has been fixed. By \cref{cor:compatible-maximality}, all its exact
representatives give the same assembled added operator.

Give $\mathcal Z_h$ a specified Hilbert norm and let
$\mathsf D:\mathcal Z_h\to\mathcal Y$ collect the cell defect forms
$\mathcal T_e(z_e;\cdot,\cdot)$ in fixed local bases, with a specified
Hilbert norm on the residual space $\mathcal Y$.
Its kernel is $\mathcal Z_h^0$.
Let $\tau$ extract the full external \emph{potential} trace, set
$\mathcal Z_b=\ker\tau$, and restrict
$\mathsf D_b=\mathsf D|_{\mathcal Z_b}$.
Potential constants are not removed from these coordinates when their
traces matter. Fix a lift $z_0$ of a desired trace $b=\tau z_0$, and
let $P_b$ be the $\mathcal Y$-orthogonal projector onto
$\operatorname{range}\mathsf D_b$.

\begin{proposition}[Trace obstruction and least-norm correction]
\label{prop:trace-compatible-repair}
The vector $\eta_b=(I-P_b)\mathsf D z_0$ is independent of the
chosen lift of $b$, and
\[
\min_{\tau z=b}\|\mathsf D z\|_{\mathcal Y}
=\|\eta_b\|_{\mathcal Y}.
\]
An exactly compatible potential with trace $b$ exists precisely when
$\eta_b=0$. Among corrections attaining the minimum defect, the
least-norm correction gives
\begin{equation}
z_*=z_0-\mathsf D_b^\dagger\mathsf D z_0.
\label{eq:trace-compatible-repair}
\end{equation}
Here the pseudoinverse uses the specified two Hilbert metrics.
If $\eta_b=0$ and $\mathsf D_b$ has positive rank, let
$\sigma_b$ and $\sigma_{\max}$ be its smallest positive and largest
singular values. Then
\[
\frac{\|\mathsf D z_0\|_{\mathcal Y}}{\sigma_{\max}}
\le
\min_{\substack{\tau z=b\\\mathsf D z=0}}\|z-z_0\|_{\mathcal Z_h}
=\|\mathsf D_b^\dagger\mathsf D z_0\|_{\mathcal Z_h}
\le
\frac{\|\mathsf D z_0\|_{\mathcal Y}}{\sigma_b}.
\]
\end{proposition}
\begin{proof}
Every same-trace potential is $z_0+v$, with $v\in\mathcal Z_b$.
The residual decomposes orthogonally as
\[
\mathsf D(z_0+v)
=\eta_b+(P_b\mathsf D z_0+\mathsf D_bv).
\]
Changing the lift adds an element of $\operatorname{range}\mathsf D_b$
and leaves $\eta_b$ unchanged.
The second summand can be set to zero by the pseudoinverse, proving
the minimum defect and the exactness criterion.
The singular-value decomposition gives the unique minimum-norm
correction and the two bounds. When $\mathsf D_b$ has rank zero,
the decomposition and \eqref{eq:trace-compatible-repair} still apply;
no positive singular-value denominator is used.
\end{proof}

Under the common-material-sum hypotheses of \cref{sec:optimal-margin},
take the ideal target $z_0=2c\iota$. Suppose $z_0\in\mathcal Z_h$ and
$\eta_b=0$, and put $\delta z=\mathsf D_b^\dagger\mathsf D z_0$.
The exact margin formula in \cref{sec:optimal-margin} gives
\[
\gamma_h(z_*)=c-\max_{e,q}
 \nu\!\left(-\nabla_X\delta z(X_e(\xi_{eq}))\right).
\]
The minus sign is necessary because the loss function $\nu$ need not
be symmetric. Define
\[
L_b=\sup_{\substack{v\in\mathcal Z_b\\\|v\|_{\mathcal Z_h}=1}}
                     \max_{e,q}\nu\!\left(-\nabla_Xv(X_e(\xi_{eq}))\right),
\]
with $L_b=0$ if $\mathcal Z_b=\{0\}$.
Positive homogeneity and the pseudoinverse estimate give
\[
\gamma_h(z_*)\ge c-\frac{L_b}{\sigma_b}\|\mathsf D z_0\|_{\mathcal Y}
\]
when $\sigma_b$ exists. Exact kernel dimension and finite-dimensional
solvability do not bound $L_b/\sigma_b$ uniformly in mesh size or order.
Such a conclusion requires control of the stated physical gradients
and metrics.

Every exact representative of the same trace gives the same assembled
gauge Hessian. Thus the repair constructs a locally exact representation
of the defect-corrected target operator. A solver benefit or a uniform
repair cost requires additional numerical analysis.

\section{Supporting observations and arithmetic controls}
\label{app:supporting-observations}

These records supply the complete parameters, operator comparisons and
arithmetic qualifications behind \cref{sec:evidence}.
They retain the original equations, numerical outcomes and adverse
choices. Computed eigenvalues remain distinct from the rational and
interval sign certificates.

\subsection{Normalization of high-precision bilinear comparisons}
\label{sec:bilinear-normalization}

For the \(70\)-digit checks of the aligned and non-aligned families,
let \(A_e^{64}\in\mathbb R^{30\times30}\) be the computed complete
cell Hessian and let \(u_e,v_e\in\mathbb R^{30}\) be the nodal
variation vectors. Let \(b_{64}\) be their bilinear value evaluated
in binary64 and \(b_{70}\) the independently differentiated scalar
energy value at \(70\) decimal digits. The reported discrepancy is
\begin{equation}
 e_{\rm bil}=\frac{|b_{64}-b_{70}|}{s_{64}},\qquad
 s_{64}=\operatorname{fl}_{64}
       \bigl(\|A_e^{64}\|_F\,\|u_e\|_2\,\|v_e\|_2\bigr).
 \label{eq:bilinear-operator-normalization}
\end{equation}
Here \(\operatorname{fl}_{64}\) denotes the binary64 evaluation of
the norms and their product. Stored nodal values are promoted exactly,
and the independent differentiation uses the ideal four-point rule.
The \(80\)-digit comparisons below use their separately stated
reference-bilinear normalization.

\subsection{Normal form and a fixed geometry-aligned coefficient}
\label{sec:geometry-gauge-evidence}

The opposite-edge normal form was checked independently in integer
arithmetic. All $18^3=5{,}832$ entries of the polarized tensor agreed
exactly, as did $32$ integer cubic examples. The exact Gram matrix
of the geometry-to-Hessian map gives squared singular values
$8192c_4^2$ and $32768c_4^2$, each nine times, in the coordinates of
\Cref{cor:faithful-jet}. The four printed vertex-matrix identities
were also checked as integer equalities.

Three positive-weight rational cubature rules checked the
rule dependence in \Cref{prop:symmetric-rule-cubic}.
A centroid/vertex rule with five points gives $c_Q=1/180$;
a centroid/edge-midpoint rule with seven points gives
$c_Q=-1/360$; an eleven-point combination gives $c_Q=0$.
All quadratic moments were exact. The last rule also
integrated every cubic moment exactly.
Twenty-four independently generated affine matrix fields per
rule satisfied the determinant identity with zero rational
discrepancy. These rules test positive, negative, and
zero defect coefficients.

The aligned coefficient $\operatorname{diag}(45,45,0)$ had
negative gauge curvature near the endpoints of the shear interval
$[-1/4,1/4]$. Its five sampled minimum eigenvalues in the
$A_1(D)$ metric were approximately
$-0.9692,-0.0833,0.4844,-0.0833,-0.9692$.
Allowing the diagonal coefficient to vary in joint selection over the
same interval gave
$Z_*=\operatorname{diag}(191/4,191/4,0)$ after rationalization.
The five observed minima became
$0.2241,0.2890,0.3275,0.2890,0.2241$.
The rational Bernstein certificate in
\Cref{prop:aligned-shear-chart} establishes positivity between
the sampled states. The failed coefficient and the successful
certificate are both recorded.

\subsection{Complete-pressure boundary reconstruction}

The parameter cube $[0,1]^3$ was divided into $n^3$ cubes, each split
by the six Freudenthal tetrahedra. For a real curvature
amplitude $\beta$, shared quadratic nodes were mapped to a
reference geometry by nodal interpolation of
\[
X(t)=t+\beta\prod_{i=1}^3t_i(1-t_i)e_3.
\]
The polynomial in this display prescribes nodal values; each
isoparametric cell map is quadratic. Its first two components are
affine, so its Jacobian determinant is affine on each cell.
Positive vertex determinants establish whole-cell orientation
for this family. The first two components are unchanged and the
third is strictly increasing along vertical lines, giving
global injectivity with the fixed cube boundary.

The current map was the quadratic field
$y_h(X)=DX+\epsilon(X_2-1/2)^2e_1$ associated with
\Cref{prop:aligned-shear-chart}.
For $n=2$, we used $\beta\in\{0,8\}$,
$\epsilon\in\{-1/4,0,1/4\}$, and
$\kappa\in\{10,1000\}$, with either a full boundary constraint
or a constraint on the face $t_1=0$. This gives $24$ complete
operator cases. Four additional cases used $n=4$, $\beta=8$,
$\epsilon=1/4$, both bulk moduli, and both constraints.
The large cases tested the boundary representation and the solve obtained
by eliminating interior coordinates; iterative comparisons were omitted.

The complete matrix $K_h$ and translated matrix $G_h$ were
assembled independently. A third route integrated the cofactor
second variation over oriented external facets to form
$B_\Gamma$. Every pressure contribution was retained, including
the small pressure-geometric term generated by binary64
rounding in the nominally volume-preserving cases.
\Cref{tab:aligned-complete} reports the results. The reconstruction
column is
$\|K_h-G_h+R_\Gamma^TB_\Gamma R_\Gamma\|_F/\|K_h\|_F$.
For a seeded right-hand side $b$ and computed solution $x$,
the solve column is the independently recomputed relative residual
$\|K_hx-b\|_2/\|b\|_2$.

For this elimination, partition the free displacement coordinates into
interior coordinates $i$ and free boundary coordinates $\Gamma$.
The Schur complement of the positive reference is
\[
S_G=G_{\Gamma\Gamma}-G_{\Gamma i}G_{ii}^{-1}G_{i\Gamma},
\qquad S_K=S_G-B_\Gamma.
\]
The reduced complete solve uses $S_K$. Since $G_{ii}\succ0$, all negative
directions of the complete matrix are represented by $S_K$.
When $S_G\succ0$, the trace pencil is the generalized eigenvalue problem
$S_Kx_\Gamma=\lambda S_Gx_\Gamma$; its negative eigenvalues give the
partially constrained counts in the table.

For the $n=2$ comparison, both preconditioners used exact global
Cholesky reference actions. One was the translated gauge; the
other spectrally clipped the pointwise material Hessian and
retained the positive cell-volume outer product.
The complete matrix in the MINRES iteration was identical.
The iteration stopped on an independently evaluated relative
residual of $10^{-9}$ or at $240$ columns; these are diagnostic
stopping settings.
With a full constraint, the translated gauge equals the complete
matrix, so its one-column result follows from exact
preconditioning. The clipping reference took $17$ columns.
With one face constrained, the gauge required $185$--$240$
columns and reached the residual setting in nine of twelve cases;
clipping required $194$--$240$ columns and reached it in eleven.
Neither solver shows a consistent partial-boundary advantage from
the gauge alone. The explicit boundary Schur solve retained the
complete indefinite equation in all cases.

Eight further cases uniformly dilated the $n=2$, $\beta=8$,
$\epsilon=1/4$ field to
$J=1\pm1/1024$, using both bulk moduli and constraints.
The cell pressures were approximately $\pm0.009765625$ and
$\pm0.9765625$. Boundary reconstruction remained below
$2.75\times10^{-17}$, and all complete trace-solve residuals were
below $1.11\times10^{-12}$.
For the compressed $\kappa=1000$ case, the minimum generalized
eigenvalue of the full pointwise density gauge was approximately
$-0.4202$, although the assembled translated reference remained
numerically positive by Cholesky factorization.
That density omits the separate positive
cell-volume outer product. Its pointwise sign alone therefore
does not decide the assembled sign.
A uniform nonzero-pressure extension of
\Cref{prop:aligned-shear-chart} would need its own argument.

Independent $70$-digit differentiation of the complete scalar
condensed energy checked eight local bilinear actions, with
maximum discrepancy $4.916\times10^{-17}$ using the normalization in \eqref{eq:bilinear-operator-normalization}. It used the exact
promoted binary64 nodal values and the ideal four-point rule;
quadrature-node rounding is consequently a separate error source. The $36$ cases are manufactured operator states.
A load-controlled nonlinear trajectory and production-scale
solver performance remain unmeasured.

\subsection{Aligned and unaligned curvature under two rules}

A complete-operator alignment comparison used the same $n=2$
mesh and a scalar continuous quadratic field $\psi_h$
with zero boundary values. Its interior nodal values were
seeded normal samples. For an axis $a\in\{1,3\}$, let
$\psi_{h,a}$ be this field divided by the largest absolute
vertex value of $\partial_a\psi_h$ over all cells.
The reference map was
$X(t)=t+\theta\psi_{h,a}(t)e_a$, with
$\theta\in\{0,1/4,1/2,3/4,0.9\}$.
Its determinant is $1+\theta\partial_a\psi_{h,a}$.
The vertex-slope normalization and the affine derivative
give whole-cell orientation; the observed minimum relative
reference Jacobian was $0.10$. Fixed boundary values and
one-component monotonicity give global injectivity.

The current deformation was $y=DX$. Its physical gradient
is exactly the specified $D$, so the condensed pressure
coefficient vanishes and the complete positive pressure
outer product remains. The two axes, five amplitudes,
three bulk moduli $0,10,1000$, and two rules gave $60$
full-Dirichlet operator observations on $81$ free coordinates.
The second rule was the positive eleven-point cubic-exact
rule above. Each rule defined its own complete energy.

With curvature in the first direction, the four-point
relative difference $\|G_h-K_h\|_F/\|K_h\|_F$ reached
$5.86\times10^{-4}$. The cubic-exact rule reduced this
cofactor compensation to the arithmetic scale.
With curvature in the third direction, the coefficient
$Z_*$ is aligned and both rules have compensation at that
scale. All $60$ complete matrices remained numerically
positive, with smallest eigenvalues at least $1.3796$ in
the given nodal coordinates.
The comparison agrees with the rule/alignment
classification; it does not demonstrate an instability of
the four-point complete equation, since a nonzero separated
defect can coexist with assembled coercivity.

\subsection{Defect-corrected reference observations on the certified mesh}

The non-aligned geometry in
\Cref{app:nonaligned-geometry} supplies $18$
complete-operator observations. They use deformation
parameters $\epsilon=-1/4,0,1/4$, bulk moduli
$\kappa=0,10,1000$, and either the full boundary
or the face $t_1=0$ constrained.
The reference defect is assembled once.
Each translated matrix then subtracts that same
defect, and an independent oriented-facet integral
supplies the exact boundary form.

The largest relative recomposition discrepancy
$\|K_h-\widehat G_h+R_\Gamma^TB_\Gamma R_\Gamma\|_F/
\|K_h\|_F$ is $4.041\times10^{-17}$.
Every corrected reference is numerically positive, and
all complete trace solves have independent
relative residual below $5.380\times10^{-13}$.
For the one-face constraint, the numerical
negative counts of the complete matrix are
$31$, $25$, and $20$ at the three bulk moduli,
respectively, at every sampled deformation.
The corrected reference and the complete matrix
therefore have distinct sign roles.

The largest arithmetic cell pressure is
$5.785\times10^{-13}$; all resulting pressure-geometric
terms were retained.
Three independent $70$-digit scalar-energy
bilinear checks have maximum discrepancy $8.223\times10^{-17}$
under the normalization in \eqref{eq:bilinear-operator-normalization}.
The interval theorem concerns the exact
parametric deformation family; these observations
also expose the rounding in its binary64 nodal
construction. No uniform floating-point guarantee
as the bulk modulus tends to infinity is inferred.

\subsection{Complete finite-state and boundary-equivalence observations}

Let $a_h$ be the integer-nodal perturbation field of
\Cref{app:compatible-geometry}.
The reference maps $X_h^\alpha=t+(\alpha/471)a_h(t)$ use
$\alpha\in\{0,1/8,1/4,1/2,7/8\}$.
\Cref{fig:compatible-potential-geometry} displays the strongest curved
reference and its corner-linear coefficient potential.

With $D=\operatorname{diag}(2,2,1/4)$, the physical maps
$y_h(X)=FX$ with $F=I_3$, $I_3+(1/4)E_{12}$, and
$D+(1/8)E_{12}$, the bulk values $0,10,1000$, and two
essential-constraint sets give $90$ complete operator observations.
The compatible potential is \(4\overline X_h\) for the first two states
and \(Z_*\overline X_h\) for the third.
The face constraint fixes variations on $t_1=0$; the vertex
constraint fixes all three components at $t=(0,0,0)$.
All four nonzero-amplitude meshes have
exactly certified three-directional curvature on every cell.
For a computed matrix $A$ and a specified reference matrix $B$,
define the normalized difference
$d(A;B)=\|A-B\|_F/\max\{1,\|B\|_F\}$.
The volume and oriented-face gauge matrices agree to
$6.66\times10^{-15}$ with the volume gauge as reference.
The constant-gauge four-point defect reaches $0.0220$ with its
exactly integrated gauge as reference.
Three independent determinant-moment/face-minor checks agree
exactly in rational arithmetic, including cancellation on
interior faces. Six $80$-digit differentiations of the complete
material energy agree with the implemented bilinear forms to
$1.61\times10^{-14}$ after dividing the absolute difference by
$\max\{1,|b_{\rm ref}|\}$, where $b_{\rm ref}$ is the
high-precision reference bilinear value.

\FloatBarrier

The $90$ rows contain $88$ positive compatible references.
The two failures occur at the strongest curved geometry,
zero bulk modulus, and the more strongly stretched state
$D+(1/8)E_{12}$ with potential
$z_h=Z_*\overline X_h$ from
\Cref{prop:aligned-shear-chart}.
Their minimum eigenvalues are approximately $-0.0526$ and
$-0.0533$ for the face and vertex constraints.
The observed bulk values $10$ and $1000$ restore positivity
through the additional volumetric outer term. The
defect-corrected constant-gauge references remain positive in
these observations. Exact gauge quadrature alone therefore
does not imply an admitted positive reference at every material state.

An additional $18$-row control curves only interior nodes.
The two potential traces then agree, and the compatible and
defect-corrected constant-gauge operators agree to
$1.31\times10^{-15}$ in $d$, with the corrected reference as $B$, as required by
\Cref{cor:compatible-maximality}.
With curved boundary traces they can differ; the strongest
reference gives a normalized difference of $0.0775$ between the
two gauge matrices, with the compatible gauge as reference.
These findings identify a quadrature-compatible local
construction and its coercivity regime, without claiming
general nonlinear-solver speedup.

\subsection{A trace-preserving potential change}

Use the nonuniform family of \cref{sec:compatible-loaded-evidence}
and the normalized matrix difference \(d\) defined above.
A separate potential comparison changes only one interior nodal value
while preserving its entire boundary trace. Its exactly integrated
assembled gauge changes by $4.248\times10^{-16}$ in the normalization
$d$ defined above; the four-point matrix changes by $0.0037862$.
Boundary equivalence therefore does not imply cellwise quadrature exactness
after an arbitrary quadratic change of potential. This comparison
clarifies the distinct assumptions on the potential trace and its
cellwise polynomial space.

\subsection{Supporting complete-pressure and arithmetic checks}

The material identities were independently checked through direct
differentiation of the scalar Mooney--Rivlin energy and through the
cofactor chain rule in \cref{app:material-calculus}.
The continuous-family certificates use exact rational Bernstein
coefficients and outward-enclosed matrix inequalities.
Ordinary computed eigenvalues remain separate from those sign proofs.

A physical-state reconstruction with $176{,}982$ displacement
coordinates and $34{,}295$ Tet10 cells illustrates the complete-pressure
distinction. Of these cells, $26{,}613$ use condensed $P_0$ pressure
and $7{,}682$ use quadrature-local volumetric penalties.
Both compared matrices retain the latter cells' contributions.
The partial Hessian in this comparison omits the condensed
pressure contributions $Q_{P0}+P_{P0}$ in
\eqref{eq:qp0-import}--\eqref{eq:pp0-import}.
On that state it has $22$ negative eigenvalues reported by the sparse factorization,
whereas the full reduced-energy Hessian has zero. The gauges and
complete tangent were reconstructed from the same state, and the
scope-difference identity closes to $2.149\times10^{-16}$ relatively.
These cuDSS counts are numerical factorization observations, without
perturbation certification.

A small component also requires its own accuracy interpretation.
For the constant-coefficient cofactor-defect component of that
physical state, the matrix Frobenius norm is $5.7130\times10^{-8}$.
Direct and parent-coordinate matrices differ in Frobenius norm by
$1.9532\times10^{-16}$ absolutely, or $3.4189\times10^{-9}$
relative to that component. Restricted affine fields increase
the relative action difference to $1.69\times10^{-7}$.
This is near-kernel forward sensitivity, consistent with standard
backward-versus-forward error analysis
\citep{higham2002accuracy,croci2024mixed}. A componentwise relative
ratio is not the error of the complete tangent action.
The observations limit claims of uniformly accurate independent
application of a tiny correction. The exact compatibility and
compensation identities themselves are unaffected.

\subsection{Verification files for the finite-family certificates}
\label{sec:certificate-files}

The authors retain the proof objects and verification scripts for these
finite families, available upon reasonable request and planned for public
release in a subsequent version. The file \texttt{aligned-shear.json}
gives the exact Bernstein matrices and rational pivots for the $1/5$ margin
in \Cref{prop:aligned-shear-chart}.
The nonaligned-core package contains the exact geometry,
the $23/100$ material certificate, and the interval
matrix endpoints and binary64 triangular transform used in the
congruence proof. Its driver \texttt{verify\_core\_certificate.py}
checks the geometry, material inequality and matrix enclosure.

The compatible-core package contains the integer geometry
definition, rational material coefficients, all $8640$ retained pivot
intervals, and the distortion
witness. Its driver \texttt{verify.py} reconstructs the interval certificate.
Each package has a README with extraction and execution instructions.
These files specify the particular computational proofs reported above;
the mathematical matrices and reconstruction formulas are also defined
in \cref{sec:coercivity-regimes} and
\cref{app:nonaligned-geometry,app:compatible-geometry}.
The certificate programs use Python, NumPy, SciPy and mpmath, independently
of the simulation framework.

\section*{CRediT author statement}
Yanlin Liu originated the central theoretical ideas and mathematical
formulations and led the theoretical development.

\noindent\textbf{Yanlin Liu:} Conceptualization (central theoretical ideas);
Methodology (mathematical formulation and method development);
Formal analysis (theoretical analysis and derivations);
Software (research-specific development, framework extensions and performance
optimization); Investigation; Validation; Writing---original draft;
Writing---review and editing.

\noindent\textbf{Kaixiang Yao:} Conceptualization (application motivation
for reducing warm execution time in repeated robotics simulations);
Writing---review and editing.

\noindent\textbf{Chao Huang:} Software (development of the foundational
simulation framework); Conceptualization (introduction of the underlying
numerical problem and the near-incompressibility regime).

\noindent\textbf{Yao Shen:} Supervision; Resources (provision of computational
resources); Conceptualization (broad problem introduction and research context).

\section*{Data and code availability}
Interval certificates, exact verification data, constitutive verification
scripts, and controlled inputs are available from the authors upon
reasonable request. A public reproducibility repository is planned for a
subsequent version. The private simulation platform and its production
solvers are outside the planned release.

\section*{Declaration of generative AI use}
OpenAI Codex was used to assist with experimental coding, preliminary
mathematical review, and preliminary drafting. Anthropic Claude was
used for draft reviewing and editing. Yanlin Liu provided the central
theoretical ideas and initial mathematical formulations.
Responsibility for the claims, citations, computational results,
and final manuscript rests with the named authors.

\bibliographystyle{unsrtnat}
\bibliography{references}

\end{document}